\documentclass[12pt]{amsart}
\usepackage{amsmath}
\usepackage{amssymb}
\usepackage{amsfonts}
\usepackage{amsthm}
\usepackage{bbm}
\usepackage{cite}
\usepackage{latexsym}
\usepackage[pdftex]{graphicx}
\usepackage{color}
\usepackage{comment}
\usepackage{url}
\usepackage{hyperref}
\usepackage{todonotes}
\input xy
\xyoption{all}

\newcommand{\R}{\mathbb{R}}
\newcommand{\C}{\mathbb{C}}
\newcommand{\N}{\mathbb{N}}
\newcommand{\Z}{\mathbb{Z}}

\newcommand{\T}{\mathbb{T}}
\newcommand{\scr}{\mathcal}

\newcommand{\Open}{\mathrm{Op}}
\newcommand{\wt}{\widetilde}
\newcommand{\sub}{\mathrm{sub}}

\theoremstyle{plain}
\newtheorem{thm}{Theorem}[section]
\newtheorem{lem}[thm]{Lemma}
\newtheorem{prop}[thm]{Proposition}
\newtheorem{quest}[thm]{Question}
\theoremstyle{definition}
\newtheorem{defn}[thm]{Definition}
\newtheorem{exam}[thm]{Example}
\newtheorem{cor}[thm]{Corollary}

\newtheorem*{claim*}{Claim}
\theoremstyle{remark}
\newtheorem{rmk}[thm]{Remark}

\title{Getting a handle on contact manifolds}
\author{Kevin Sackel}

\begin{document}

\begin{abstract}
We develop the details of a surgery theory for contact manifolds of arbitrary dimension via convex structures, extending the 3-dimensional theory developed by Giroux. The theory is analogous to that of Weinstein manifolds in symplectic geometry, with the key difference that the vector field does not necessarily have positive divergence everywhere. The surgery theory for contact manifolds contains the surgery theory for Weinstein manifolds via a sutured model for attaching critical points of low index. Using this sutured model, we show that the existence of convex structures on closed contact manifolds is guaranteed, a result equivalent to the existence of supporting Weinstein open book decompositions.
\end{abstract}

\maketitle
%\tableofcontents
%\newpage

%------------------------------------------------------------------------
%--------------------START-HERE------------------------------------------

\section{Introduction}
Three-dimensional contact geometry is a gem of a mathematical garden, with many thoroughly developed tools which have been fleshed out over the past few decades. One particularly fruitful tool comes in the guise of convex surfaces, as pioneered by Giroux \cite{Giroux}. There are many aspects to the theory in three dimensions which are quite powerful, built loosely out of the following two facts. First, convexity is a $C^{\infty}$-generic condition, so such surfaces exist in abundance in contact 3-manifolds. Second, a neighborhood of a convex surface is completely encoded by the data of its dividing set, with flexibility in the characteristic foliation (which can be made to realize certain Legendrians \cite{Kanda,Honda}). Taken together, for example, one can classify the tight contact structures on many $3$-manifolds (possibly with boundary), including $T^3$ \cite{Kanda}, $T^2 \times I$, and $L(p,q)$ \cite{Honda, Hondafix}. (See also \cite{EH,GLS1,GLS2,Honda2,Matkovic,Wu} for more examples.) These facts and their resulting toolkit are now somewhat classical and have been used repeatedly throughout $3$-dimensional contact geometry in the years since classification of tight contact structures on such manifolds became possible.

In higher dimensions, convex hypersurfaces are not $C^{\infty}$-generic, and it is an open question whether they are $C^0$-generic. While we do not yet understand this genericity, the angle suggested in this paper for studying convex hypersurfaces is to consider them as level sets of so-called convex contact manifolds, for which we develop a surgery theory. Convex contact manifolds were first studied in the same aforementioned seminal paper of Giroux \cite{Giroux}, based on the notion of convexity defined by Eliashberg and Gromov \cite{EG}, in which a full surgery theory was developed in three dimensions. That a higher-dimensional surgery theory for convex contact manifolds should exist has been expected by experts for a long time, though the devil is in the details. We finally pin down this long-delayed theory and some of its most basic consequences.

The surgery theory we develop is meant to be analogous to that of Weinstein manifolds, which one can find in great detail in the book \cite{CE} of Cieliebak and Eliashberg. In the symplectic setting, these Weinstein handle decompositions have been useful in understanding all sorts of symplectic information, notably including pseudo-holomorphic curve invariants and flexibility and subflexibility phenomena (see, for example, \cite{BEE, CE,CM,MS}). It is hoped that the theory developed in this paper can eventually be used in similar ways in contact geometry, to understand both rigid and flexible phenomena.

Before outlining the theory, we note that the systematic study of convex hypersurfaces in higher dimensions was initiated by Honda and Huang \cite{HH} concurrently with the author's foray into the field. Their viewpoint involves the study of bypass attachments along convex hypersurfaces, which can be thought of as a certain smoothly cancelling pair of $n$- and $(n+1)$-handles. The more general surgery-theoretic viewpoint presented here should be seen as complementary to their work.

\subsection{The surgery theory}

A convex contact cobordism is a quadruple $(M^{2n+1},\xi,X,\phi)$ such that $(M,\xi)$ is a contact manifold, possibly with boundary, $X$ is a contact vector field transverse to the boundary, and $\phi$ is a Morse function for which $X$ is a pseudo-gradient. The boundary $\partial M$ is then automatically a convex hypersurface, realized as such by the transverse contact vector field $X$. In fact, any neighborhood of a convex hypersurface provides an example of a convex contact cobordism: any $(\Sigma^{2n} \times I, udt + \beta, \partial_t, t)$, where $u \in C^{\infty}(\Sigma)$ and $\beta \in \Omega^1(\Sigma)$ satisfying $(ud\beta - ndu \wedge \beta) \wedge d\beta^{n-1} \neq 0$, is an example. Away from the critical points of $\phi$ (points where $X=0$), the level surfaces of $\phi$ sweep out these trivial cylindrical convex contact cobordisms, and so it suffices to understand the convex contact structure around a critical point.

Suppose $p$ is a critical point of $\phi$ of index $0 \leq k \leq n$. Such a critical point will be called a subcritical point. The quadruple
\begin{itemize}
	\item $M = \R^{2n+1}$
	\item $\alpha = dz+\frac{1}{2}\sum_{i=1}^{n}(x_idy_i-y_idx_i)$
	\item $X = z\partial_z + \sum_{i=1}^{k}(-x_i\partial_{x_i} + 2y_i\partial_{y_i}) + \frac{1}{2}\sum_{i=k+1}^{n}(x_i\partial_{x_i} + y_i\partial_{y_i})$
	\item $\phi = z^2+\sum_{i=1}^{k}(-x_i^2+y_i^2) + \sum_{i=k+1}^{n}(x_i^2+y_i^2)$
\end{itemize}
defines a convex structure around the origin of index $k$. Similarly, if $\phi$ has index $n+1 \leq k \leq 2n+1$, which we call supercritical, then we can take the subcritical neighborhoods just described but with the pair $(-X,-\phi)$ instead. These models turn out to be sufficient to studying critical points -- any convex contact structure can be made to look like these standard models up to nice homotopies. As in standard Morse theory, these models allow us to define standard handles, which are specified neighborhoods $H_k$ of these critical points, and the convex contact structure along the attaching region encodes how the handle can be attached along the boundary of a given convex contact cobordism. We call a convex contact cobordism which arises by handle attachments along a trivial cylinder a convex contact handlebody. The main theorem of this paper, whose proof has just been outlined, is the following.

\begin{thm}[= Theorem \ref{thm:main_thm_body}] \label{thm:main_thm}
Every convex contact cobordism $(M^{2n+1},\xi,X,\phi)$ is strictly convex homotopic to a convex contact handlebody.
\end{thm}

By itself, this theorem is a bit weak, in that handle attachments, a priori, may depend upon continuous choices of data. We will prove that this is not the case, by understanding precisely what data encodes a given handle attachment.

Consider a convex contact cobordism with boundary $\partial M = \Sigma$. The set of points along $\partial M$ in which $X \in \xi$ is a smooth submanifold, in fact a contact submanifold, referred to as the dividing set. If we wish to specify a handle attachment, it suffices to specify a framed isotropic submanifold in the dividing set in the subcritical case, and a certain special framed coisotropic submanifold, called a balanced coisotropic sphere, in the supercritical case, up to some natural equivalence. Hence, we obtain a surgery-theoretic picture analogous to the case of Weinstein manifolds, in which any convex contact cobordism is encoded by a sequence of (framed) isotropic and coisotropic submanifolds in sequential level sets.

\subsection{Relation to open books}

Suppose we are given a contact manifold $M$ with convex boundary, and suppose $X$ is a contact vector field transverse to the boundary. Let $\Gamma$ be the dividing set on $\partial M$. The surgery theoretic picture above implied that subcritical points were attached along spheres in the dividing set, so in large part, the geometry of convex contact structures is encoded in the dividing sets of its level sets. It is therefore natural to `suture' $M$ along $\Gamma$ (\`a la \cite{CGHH}) to produce a contact manifold with corners, singling out the geometry near $\Gamma$. Near the sutured region, the contact manifold looks of the form $([-1,1] \times (-1,0] \times \Gamma, Cds + e^{\tau}\beta_0)$, where $(s,\tau) \in [-1,1] \times (-1,0]$ and $\beta_0$ is a contact form on $\Gamma$. We think of this as $[-1,1] \times L$, where $L$ is the Liouville domain $(-1,0] \times \Gamma$ with Liouville form $e^{\tau}\beta_0$. In other words, near the suture, a sutured contact cobordism looks like a thickened Liouville collar. One can then further include convex information $(X,\phi)$ to define the notion of a sutured convex contact cobordism.

It is not hard to go back and forth between the sutured and the smooth model for convex contact structures. From a sutured model, one can round out the corners. From a smooth model, one simply sutures along the dividing set in a standard way. These two operations are inverse to each other up to a natural notion of homotopy.

This sutured model is particularly convenient if we wish to study subcritical handle attachment. Smoothly, subcritical handles are attached in a neighborhood of the dividing set of the convex boundary. If we use the sutured model instead, one sees, at least if we forget about the convex data $(X,\phi)$, that this corresponds to a thickened Weinstein handle attachment. In fact, the framing data for subcritical handle attachment matches the framing data for Weinstein handle attachment. That is, to the neighborhood $[-1,1] \times L$ of the suture, we attach $[-1,1] \times H$ where $H$ is a Weinstein handle which attaches to $L$.

Hence, any sequence of subcritical handle attachments can be seen as appending a thickened Weinstein cobordism. If we then have a convex contact cobordism with only subcritical points, then the underlying contact manifold is simply $[-1,1] \times W$, where $W$ is a Weinstein domain. If we unsuture this manifold, this this is just half of an open book with page $W$ which is supported by the contact structure (see Section \ref{sec:OBD} for details and definitions). Hence, if we have a convex contact manifold (without boundary) which is \emph{split}, meaning that its subcritical points have smaller critical values than its supercritical points, then one can split the contact manifold between the two classes of critical points and view it as two halves of a supporting open book glued together in the middle. This proves the following result, stated more precisely later.

\begin{thm}[= Theorem \ref{thm:split_is_OBD}] \label{thm:proto_split_is_OBD} On a contact manifold, there is a correspondence between split convex structures and supporting Weinstein open books.
\end{thm}

By a result of Giroux and Mohsen \cite{Giroux2}, every closed contact manifold has a supporting Weinstein open book.

\begin{cor}[= Theorem \ref{cor:CCS_exist}] \label{cor:existence}
Every closed contact manifold has a (split) convex contact structure, and is therefore a convex contact handlebody.
\end{cor}

We remark that the procedure of obtaining a convex structure from an open book was already discovered by Courte and Massot \cite[Proposition 3.5, Remark 3.6]{CoMa} using essentially the same methods, and has been a folk theorem for an even longer time before that. Our perspective here is slightly different, in that we demonstrate that convex structures give rise to explicit handle decompositions, and we instead emphasize the opposite procedure of obtaining an open book from a convex structure. %In Section \ref{sec:Giroux_correspondence}, we will delve a little deeper into how this theory relates with the well-known Giroux correspondence of three-dimensional contact geometry, and how it appears naturally in a long exact sequence of a certain Serre fibration. Remark \ref{rmk:do_not_split} offers some reasons why, going forward in the theory, it may be better to avoid working only with convex contact manifolds which are split, since it may fundamentally be ignorant of higher homotopical information in contact geometry.

\subsection{Birth-death and convex contact homotopy}

Finally, as in the Morse-theoretic or Weinstein manifold setting, we wish to allow for birth-death type singularities in our notion of homotopy. It is possible, for example, that a pair of convex handles may cancel in a way such that the underlying contact geometry is essentially trivial.

Given the relation between subcritical handle attachments and thickened Weinstein handle attachments, one finds that subcritical handles cancel if and only if the corresponding Weinstein handle attachments cancel. This also allows us to understand cancelling supercritical handles by the duality offered by flipping the pair $(X,\phi)$ to $(-X,-\phi)$. We therefore obtain the following result.

\begin{cor}[= Proposition \ref{prop:subcrit_cancel}] \label{cor:subcrit_cancel_intro} For $0 \leq k \leq 2n$, $k \neq n$, an index $k$- and $(k+1)-$ handle form a cancelling pair so long as there is one and only one (non-degenerate) trajectory between them.
\end{cor}

This leaves the case of cancelling $n$ and $n+1$-handles. We provide a model for a cancelling pair in this paper. It will turn out that this model recovers the notion of a trivial bypass attachment due to Honda and Huang \cite{HH}. Unfortunately, it is not obvious that this is the only possibility. There is a clear proof strategy: one must understand the contact geometry of embryonic critical points as one does in the Weinstein setting \cite{CE}. This will be undertaken in future work, and will complete the basic surgery theory for convex contact manifolds.

\subsection{Outline}
For the reader's convenience, we include an outline of the rest of the paper. The experienced reader may wish to skip Sections \ref{sec:background} and \ref{sec:Weinstein}, referring back to them as reference. The same experienced reader may also be interested in Section \ref{subsec:nbhd}, which describes neighborhood theorems in a very general way.

\begin{itemize}
	\item Section \ref{sec:background} - We review the basic submanifolds in contact geometry - isotropic, coisotropic, and convex. Coisotropic submanifolds remain relatively under-studied in contact geometry, and are likely the least familiar to the reader. We also discuss (in a more general way than one typically encounters in the literature) how topological data along these manifolds encodes their contact neighborhoods.
	\item Section \ref{sec:BCS} - We describe the particular neighborhood theorems which we will need to understand handle attachments. In particular, we define the notion of a balanced coisotropic sphere prove it has a standard neighborhood.
	\item Section \ref{sec:Weinstein} - Here we recall what surgery theory looks like for Weinstein manifolds. Our surgery theory for contact manifolds in fact includes this case; see \emph{Act 3}.
	\item Section \ref{sec:CCMs} - We describe the basics of convex contact cobordisms. The main result is that one can homotope every critical point into standard form in a controlled way.
	\item Section \ref{sec:attach} - We explicitly analyze how to attach the handles to a given convex contact cobordism. We prove the main result of the paper, Theorem \ref{thm:main_thm}, and describe the underlying surgery theory, including an explicit description of the attaching data.
	\item Section \ref{sec:suture} - We suture our convex contact cobordisms, allowing us to think about these objects as having thickened Weinstein collars. We show that subcritical handle attachment corresponds to a thickened Weinstein handle attachment.
	\item Section \ref{sec:OBD} - The correspondence of the previous section gives rise to a way to go between closed convex contact manifolds and supporting Weinstein open book decompositions, Theorem \ref{thm:proto_split_is_OBD}, hence also proving that every closed contact manifold has a handle decomposition, Corollary \ref{cor:existence}.
	\item Section \ref{sec:homotopy} - We discuss how pairs of critical points cancel, yielding a notion of convex contact homotopy. The correspondence of Section \ref{sec:suture} easily reduces all cases to the Weinstein setting with the exception of the middle-dimensional handles. In the middle-dimensional case, we provide a model which recovers the trivial bypasses of Honda and Huang \cite{HH}.
\end{itemize}

\subsection{Acknowledgements}

I am extremely grateful to my Ph.D. advisor, Emmy Murphy, for suggesting this project, for edits and improvements on the thesis version of this paper, and for constant encouragement and helpful discussions throughout my life as a graduate student. I am also thankful to Ko Honda for expressing interest in this work, and for explaining aspects of his related research on bypasses. This work is a condensed version of my thesis, the majority of which was performed with the support of an NSF Graduate Research Fellowship.

\section{Background in contact geometry} \label{sec:background}

In this section, we lay out some of the basic concepts in contact geometry, with a view towards the surgery theory we will develop. Notably, we are interested in three particular classes of submanifolds in a contact manifold: isotropics, coisotropics, and convex hypersurfaces. We assume familiarity with isotropic and coisotropic submanifolds in symplectic geometry, as well as their corresponding neighborhood theorems; on the other hand, we assume no familiarity with contact geometry, although the reader may wish to consult Geiges' excellent book \cite{Geiges} for a thorough and gentle introduction. Proofs of all results in this section which aren't otherwise presented here can be found in the author's thesis \cite{Sackel_thesis}.

Although the results in this section are of a classical flavor, the presentation differs from standard sources in two ways. First, we provide a general neighborhood theorem for contact geometry, clarifying that the typical case-by-case neighborhood theorems one encounters have essentially the same proof. This includes also a statement of a parametric relative version which, to the author's knowledge, is otherwise missing from the literature. Second, the study of coisotropic submanifolds has only relatively recently been initiated in work of Huang \cite{Huang1, Huang2, Huang3}.

\subsection{Contact manifolds}

\begin{defn} A \textbf{contact manifold} is a pair $(M,\xi)$ consisting of an odd-dimensional manifold $M^{2n+1}$ and a (codimension $1$) hyperplane distribution $\xi \subset TM$ which is \textbf{maximally non-integrable}, which means that if $\xi = \ker(\alpha)$ for a $1$-form $\alpha$ on some chart $U \subset M$, then $\alpha \wedge d\alpha^n \neq 0$. We say the contact manifold is \textbf{coorientable} if the line bundle $TM/\xi$ is trivial, and \textbf{cooriented} if we choose a coorientation, meaning an orientation for the line bundle $TM/\xi$. Equivalently, $\xi$ is coorientable if there is a global (nonzero) $1$-form $\alpha$, called a \textbf{contact form} with $\xi = \ker(\alpha)$, which is well defined up to scaling by a global nonzero function. A coorientation is then a choice of $\alpha$ up to scaling by a positive function.
\end{defn}

For the rest of this paper, we assume that all of our contact manifolds are cooriented.

\begin{rmk}
The relevant data is always $\xi$ (with its coorientation) and not $\alpha$. In practice, however, we will sometimes do computations in which we have chosen a particular $\alpha$.
\end{rmk}

\begin{defn}
The \textbf{Reeb vector field} $R_{\alpha}$ for a contact form $\alpha$ is the unique vector field such that
$$\alpha(R_{\alpha})=1,~~~~d\alpha(R_{\alpha},-) = 0.$$
\end{defn}

We see that $TM = \langle R_{\alpha} \rangle \oplus \xi$. Since $\alpha$ vanishes on $\xi$ but $\alpha \wedge d\alpha^{n} \neq 0$, we have that $d\alpha^n|_{\xi} \neq 0$. Hence, $d\alpha$ determines a symplectic structure on the vector bundle $\xi \rightarrow M$. If we used instead the contact form $f\alpha$, with $f$ a positive function, then we instead have
$$d(f\alpha)^n|_{\xi} = (df\wedge \alpha + fd\alpha)^n|_{\xi} = f^n(d\alpha)^n|_{\xi},$$
where the last equality uses that $\alpha|_{\xi} = 0$. Hence, $d\alpha|_{\xi}$ is only determined up to positive scaling since $\alpha$ is only determined up to positive scaling.

\begin{defn}
The \textbf{conformal symplectic structure} on a contact manifold $(M,\xi)$ is the symplectic structure on $\xi$ modulo scaling by postitive functions.
\end{defn}

\subsection{Neighborhood theorems} \label{subsec:nbhd}

A neighborhood theorem is a way of interpolating between topological data on a submanifold and geometric data on an open neighborhood. The fundamental example in smooth differential topology is the tubular neighborhood theorem, which asserts that the normal bundle to a submanifold determines (a germ of) an open neighborhood. In fact, one has the following slightly stronger statement. Suppose we have, for $i=1,2$, submanifolds $M_i \subseteq N_i$, together with a bundle isomorphism $\Phi \colon TN_1|_{M_1} \rightarrow TN_2|_{M_2}$ such that $\Phi|_{TM_1} = d\phi$ for some diffeomorphism $\phi \colon M_1 \rightarrow M_2$. Then one can find neighborhoods $U_1$ and $U_2$ of $M_1$ and $M_2$, respectively, and a diffeomorphism $\psi \colon U_1 \rightarrow U_2$, such that $\psi|_{M_1} = \phi$, and $d\psi|_{TU_1|_{M_1}} = \Phi|_{TN_1|_{M_1}}$.

We present a neighborhood theorem for contact geometry which is more general than one typically encounters. There is a similar symplectic version; see the author's thesis for a statement and proof \cite{Sackel_thesis}. Typically in the literature, one only encounters these statements in specific instances, e.g. \cite{Loose}. Instead, we present the general version here, choosing to see these instances as corollaries. This choice is motivated by the fact that we will inevitably be faced with coisotropic submanifolds whose characteristic foliations have singularities, namely the spin-symmetric spheres of Section \ref{sec:BCS}.

In what follows, given a closed subset $B$ of a manifold $A$, we shall use the notation $\Open_A(B)$ for an open neighborhood of $B$ in $A$, possibly changing from line to line (typically becoming smaller). If $A$ is implicit, and we shall write simply $\Open(B)$ instead.

\begin{thm}[General contact neighborhood theorem] \label{thm:gen_nbhd_contact}
Suppose, for $i=1,2$, that $M_i \subseteq (X_i,\xi_i)$ is a submanifold of a contact manifold. Suppose that there is a diffeomorphism $\phi \colon M_1 \rightarrow M_2$ and lying over it a bundle isomorphism $\Phi \colon TX_1|_{M_1} \rightarrow TX_2|_{M_2}$ such that $\Phi|_{TM_1} = d\phi \colon TM_1 \rightarrow TM_2$ and $\Phi$ sends $\xi_1$ to $\xi_2$ and preserves the conformal symplectic structure. Then there is a contactomorphism $\psi \colon \Open(M_1) \rightarrow \Open(M_2)$ such that $d\psi|_{TX_1|_{M_1}} = \Phi$.
\end{thm}
\begin{proof}
Since $\Phi$ maps $\xi_1$ along $M_1$ to $\xi_2$ along $M_2$, we may choose contact forms $\alpha_i$ for each $\xi_i$ such that $\Phi^*\alpha_2 = \alpha_1$ along $M_1$. We choose a diffeomorphism $\psi_0 \colon \Open(M_1) \rightarrow \Open(M_2)$ realizing $\Phi$ as in the smooth neighborhood theorem.

Set $\alpha_0 = \psi_0^*\alpha_2$. Set $\alpha_t = (1-t)\alpha_0 + t\alpha_1$ for $t \in [0,1]$. Since $\alpha_0 = \alpha_1$, and $d\alpha_0|_{\xi}$ is some $C^{\infty}$ multiple of $d\alpha_1|_{\xi}$, we have that $\alpha_t$ remains a contact form in some $\Open(M_1)$. What follows next is a modification of the Moser trick as typically applied in Gray stability on closed contact manifolds, and follows the proof from Geiges' book \cite[Theorem 2.2.2]{Geiges}.

We wish to find a family of diffeomorphisms $\phi_t \colon \Open(M_1) \rightarrow \Open(M_1)$ such that $\phi_t^*\alpha_t = \lambda_t\alpha_0$ for some positive function $\lambda_t$ and such that $\phi_t|_{M_1}$ is the identity and $d\phi_t|_{TX_1|_{M_1}}$ is also the identity, since then $\psi := \psi_0 \circ \phi_1^{-1} \colon \Open(M_1) \rightarrow \Open(M_2)$ yields a contactomorphism as desired.

We suppose that $\phi_t$ is the flow of a vector field $V_t$. Then
$$\left(\frac{d}{dt}\log(\lambda_t)\right)\phi_t^*\alpha_t = \left(\frac{d}{dt}\lambda_t\right)\alpha_0 = \frac{d}{dt}(\phi_t^*\alpha_t) = \phi_t^*(\alpha_1-\alpha_0 + \scr{L}_{V_t}\alpha_t).$$
Writing $\mu_t = ((\phi_t)^*)^{-1}(\frac{d}{dt}\log \lambda_t)$, we find
$$\mu_t\alpha_t = \alpha_1-\alpha_0 + \scr{L}_{V_t}\alpha_t.$$
Suppose now that $V_t \in \xi_t = \ker(\alpha_t)$, so that $\scr{L}_{V_t}\alpha_t = i_{V_t}d\alpha_t$. Then we simply need
$$i_{V_t}d\alpha_t = -\alpha_1 + \alpha_0 + \mu_t\alpha_t.$$
Suppose $R_t$ is the Reeb vector field for $\alpha_t$. The above equation has a unique solution $V_t \in \xi$ so long as it is valid when we input $R_t$, i.e.
$$(-\alpha_1+\alpha_0 + \mu_t\alpha_t)(R_t) = (-\alpha_1+\alpha_0)(R_t) + \mu_t =  0.$$
Setting $\mu_t = (\alpha_1-\alpha_0)(R_t)$ ensures this is possible. Furthermore, along $M_0$, we have $\alpha_0 = \alpha_1$, and so $\mu_t = 0$ along $M_0$ as well, and hence $V_t = 0$ along $M_0$. Thus, $\phi_t$ indeed fixes $M_0$.

In fact, slightly more is true. Note that since $\alpha_0 = \alpha_1$ along all of $TX_1|_{M_1}$, and $d\alpha|_{\xi}(X,Y) = -\alpha([X,Y])$, we actually have that $d\alpha_0 = d\alpha_1$ along all of $TX_1|_{M_1}$. Hence $\alpha_1-\alpha_0 = O(\epsilon^2)$, where $\epsilon$ is the distance to $M_1$ for a fixed metric. Hence also $\mu_t = O(\epsilon^2)$, and so also $|V_t| = O(\epsilon^2)$. Therefore, the flow $\phi_1$ of $V_t$ has $d\phi_1$ equal to the identity along $TX_1|_{M_1}$.
\end{proof}

\begin{cor} [Darboux]
Any point $p$ in a contact manifold $(M^{2n+1},\xi)$ has a neighborhood contactomorphic to some $\Open(0)$ in $\left(\R^{2n+1},\ker (dz - \sum y_idx_i)\right)$, where $p$ is identified with $0$.
\end{cor}
\begin{proof}
All symplectic vector spaces of the same dimension are isomorphic, so the result follows immediately from Theorem \ref{thm:gen_nbhd_contact}.
\end{proof}

\begin{rmk}
Herein lies a crucial point of contact geometry: there is no interesting local geometry, since all contact manifolds are locally isomorphic. Any interesting feature of contact geometry is fundamentally global. The same is true in symplectic geometry.
\end{rmk}

Furthermore, the general neighborhood theorem holds if we include relative data and consider higher parametric families. Let $\scr{B}$ (for `bundle data') be the space of pairs $(\Phi,\phi)$ of bundle maps $\Phi \colon TX_1|_{M_1} \rightarrow TX_2|_{M_2}$ lying over a diffeomorphism $\phi \colon M_1 \rightarrow M_2$ so that $\Phi$ sends $\xi_1$ to $\xi_2$ and preserves the conformal structure. Let $\scr{G}$ (for `germs') denote the space of germs of contactomorphisms $\psi \colon \Open(M_1) \rightarrow \Open(M_2)$ such that $\psi|_{M_1}$ is a diffeomorphism onto $M_2$. There is a natural continuous map $\pi \colon \scr{G} \rightarrow \scr{B}$ given by taking the derivative of $\psi$ along $TX_1|_{M_1}$.

Meanwhile, suppose $A \subset M_1$ is a fixed closed subset. Denote by $\scr{R}$ the space of pairs $(\eta,\psi)$ consisting of germs of embeddings $\eta \colon \Open_{M_1}(A) \hookrightarrow M_2$ and germs of symplectomorphisms $\psi \colon \Open_{X_1}(A) \rightarrow \Open_{X_2}(\phi(A))$ such that $\psi|_{\Open_{M_1}(A)} = \eta$. Denote by $\scr{B}_{\scr{R}}$ the set of $((\Phi,\phi),(\eta,\psi)) \in \scr{B} \times \scr{R}$, such that $\pi(\psi)$ realizes $\Phi$ over $\Open(A)$. There is again a forgetful map $\pi_{\scr{R}} \colon \scr{G} \rightarrow \scr{B}_{\scr{R}}$.

\begin{exam} If $A = \emptyset$, then $\scr{B} = \scr{B}_{\scr{R}}$, and $\pi_{\scr{R}} = \pi$. If $A = M_1$, then $\scr{B}_{\scr{R}} = \scr{G}$, and $\pi_{\scr{R}}$ is just the identity.
\end{exam}

\begin{rmk}
We always take the $C^{\infty}$-topology, since it ensures each $\pi_{\scr{R}}$, which simply takes a derivative, is continuous. For example, let us make the topology on $\scr{G}$ explicit. We fix once and for all some neighborhood $U = \Open(M_1)$. Then the space of maps from $U$ to $X_2$ can be given the $C^{\infty}$-topology. The subset along which $M_1$ maps diffeomorphically onto $M_2$ can be given the subset topology. Finally, $\scr{G}$ is endowed with the quotient topology by the equivalence relation defining germs.
\end{rmk}

\begin{thm}[Parametric relative neighborhood theorem] \label{thm:par_rel}
For all commutative diagrams
$$\xymatrix{S^{k-1} \ar[d]_{\beta} \ar[r] & D^k \ar[d]^{\alpha} \\ \scr{G} \ar[r]_{\pi_{\scr{R}}} & \scr{B}_{\scr{R}}}$$
there exists an extension $\wt{\beta}$ of $\beta$ to $D^k$ commensurate with $\alpha$, i.e. fitting into the following diagram:
$$\xymatrix{S^{k-1} \ar[d]_{\beta} \ar[r] & D^k \ar[dl]_{\wt{\beta}} \ar[d]^{\alpha} \\ \scr{G} \ar[r]_{\pi_{\scr{R}}} & \scr{B}_{\scr{R}}}$$
As a result, $\pi$ induces a weak homotopy equivalence $\pi \colon \scr{G} \rightarrow \scr{B}_{\scr{R}}$. (The diagram for $k$ yields injectivity on $\pi_{k-1}$ and surjectivity on $\pi_k$.)
\end{thm}

\begin{rmk}
Note that the parametric relative neighborhood theorem implies a fiberwise version, yielding a homotopy equivalence for each fixed diffeomorphism $\phi \colon M_1 \rightarrow M_2$.
\end{rmk}

\subsection{Isotropic submanifolds}

\begin{defn}
An \textbf{isotropic submanifold} of a contact manifold $(M,\xi)$ is a submanifold $V \subseteq M$ for which $TV \subseteq \xi$. Equivalently, $V$ is isotropic if and only if $\R \times V \subseteq (\R \times M, \omega = d(e^t\alpha))$ is isotropic (in the symplectic sense).
\end{defn}

For any isotropic submanifold $V \subseteq (M,\xi)$, $TV \subseteq \xi$ is isotropic with respect to the conformal symplectic structure on $\xi$. Hence, in a contact manifold of dimension $2n+1$, an isotropic submanifold can have dimension at most $n$. These isotropics of maximal dimension are of particular interest.
\begin{defn}
An isotropic submanifold of dimension $n$ embedded in a contact manifold of dimension $2n+1$ is called a \textbf{Legendrian submanifold}.
\end{defn}

\begin{rmk}
By an h-principle, the study of isotropic submanifolds which are not of maximal dimension is purely algebro-toplogical (see, e.g. \cite[Theorem 12.4.1]{EM}). Meanwhile, Legendrians remain central to research in contact geometry. On the one hand, Legendrians are not classified (up to isotopy through Legendrians) purely by topological information \cite{Chekanov}. On the other, a subclass of Legendrians, termed loose \cite{Murphy}, do satisfy an h-principle, and so are purely topological. This dichotomy is especially relevant in understanding surgery-theoretic information in both contact and symplectic geometry.
\end{rmk}

We turn now to understanding what data characterizes a neighborhood of an isotropic submanifold $L \subset (M,\xi)$. We may decompose the tangent space along $L$ as
$$TM|_L = TL \oplus (TL)^{d\alpha}/TL \oplus \xi|_L/(TL)^{d\alpha} \oplus TM|_L/\xi_L.$$
Notice that $d\alpha|_{TL^{d\alpha}}$ vanishes on $TL$, and so $d\alpha$ descends to a symplectic structure on the vector bundle $(TL)^{d\alpha}/TL$. This symplectic structure is determined up to scaling, and hence determines a conformal symplectic structure.

\begin{defn}
On an isotropic submanifold $L \subset (M,\xi)$, the vector bundle $(TL)^{d\alpha}/TL$ over $L$, together with its conformal symplectic structure, is called the \textbf{conformal symplectic normal bundle}.
\end{defn}

\begin{thm} \label{thm:INT}
Suppose $L_i \subset (M_i,\xi_i)$ are isotropic submanifolds of contact manifolds. Suppose there is a diffeomorphism $\phi \colon L_1 \rightarrow L_2$ and over $\phi$ an isomorphism of conformal symplectic normal bundles $\Phi \colon (TL_1)^{d\alpha_1}/TL_1 \rightarrow (TL_2)^{d\alpha_2}/TL_2$ (intertwining their conformal symplectic structures). Then $L_1$ and $L_2$ have contactomorphic open neighborhoods via a contactomorphism realizing $\phi$ on its restriction to $L_1$.
\end{thm}

\subsection{Coisotropic submanifolds}

%Unlike isotropic submanifolds, coisotropic submanifolds are not as well understood in contact geometry. The reader may be interested recent results of Huang, focusing on coisotropic submanifolds with Legendrian characteristic foliations \cite{Huang1,Huang2,Huang3}.

\begin{defn}
A \textbf{coisotropic submanifold} $V \subseteq (M,\xi = \ker \alpha)$ is a submanifold such that $TV \cap \xi$ is a coisotropic subspace of $\xi$ at each point of $V$. Equivalently, $V$ is coisotropic if and only if $\R \times V \subseteq (\R \times M, d(e^t\alpha))$ is a coisotropic submanifold in the symplectic setting.
\end{defn}

\begin{exam} \label{exam:3d_coisotropic}
In $3$-dimensional contact geometry, a submanifold $V^k \subseteq (M^3,\xi)$ of dimension $k$ is coisotropic if and only if either:
\begin{itemize}
	\item $k \geq 2$
	\item $k=1$ and $V$ is Legendrian.
\end{itemize}
\end{exam}

We shall also need the following proposition, which essentially exchanges the role of contact and symplectic geometry in what we just proved.

\begin{prop} \label{prop:coisotropic} $V$ is a coisotropic submanifold of an exact symplectic manifold $(M,\omega = d\alpha)$ if and only if $\R \times V \subseteq (\R \times M, dz - \alpha)$ is coisotropic.
\end{prop}

\begin{rmk} This statement is clearly false for isotropic submanifolds.
\end{rmk}

\begin{rmk} The same proposition holds where $(M,\omega)$ is not exact and we instead use the Boothby-Wang construction.
\end{rmk}

\begin{defn} For a coisotropic submanifold $V \subseteq (M,\xi)$, one forms the (possibly singular) \textbf{characteristic foliation} $\scr{F}_{V,\xi} := (TV \cap \xi)^{d\alpha}$. The points at which $TV \subseteq \xi$ are called the \textbf{singular set} $S(V)$. A coisotropic submanifold for which $S(V) = \emptyset$ is called \textbf{nonsingular}.
\end{defn}

\begin{prop} \label{cor:char_fol_is_fol} The characteristic foliation of a coisotropic submanifold $V$ is integrable away from the singular set $S(V)$, hence justifying the use of the word foliation (as opposed to distribution).
\end{prop}

Along $S(V)$, $TV \cap \xi$ is one dimension higher than expected, and so $\scr{F}_{V,\xi}$ is one dimension lower. In general, the contact geometry near the singular set is unclear, although see papers of Huang \cite{Huang1, Huang3}, which study allowable singular sets in the case of Legendrian foliations for coisotropics $V^{n+1} \subset (M^{2n+1},\xi)$, and when this foliation determines a neighborhood of $V$ up to contactomorphism. We will come back to studying singular coisotropic submanifolds in the form of spin-symmetric coisotropic spheres in Section \ref{sec:BCS}.

For now, consider a nonsingular coisotropic $C \subset (M,\xi)$. Then,
$$TM|_C = (TC \cap \xi)^{d\alpha} \oplus (TC \cap \xi)/(TC \cap \xi)^{d\alpha} \oplus \xi|_C/(TC\cap \xi) \oplus TC/(TC \cap \xi).$$

\begin{defn} On a coisotropic, the conformal symplectic vector bundle $(TC \cap \xi)/(TC \cap \xi)^{d\alpha}$ is called the \textbf{conformal cosymplectic normal bundle}.
\end{defn}

\begin{rmk}
The use of the term `normal' is meant to indicate that the bundle is normal to the foliation, $\scr{F}$, though it is internal to $C$.
\end{rmk}

\begin{cor} [Nonsingular coisotropic neighborhood theorem] \label{cor:contact_coiso_nbhd}
Suppose $C_i \subseteq (M_i,\xi_i)$ are nonsingular coisotropic submanifolds of contact manifolds. Then $C_1$ and $C_2$ have contactomorphic neighborhoods precisely when there exists a diffeomorphism $\phi \colon C_1 \rightarrow C_2$ preserving the characteristic foliations and the conformal symplectic structure on the cosymplectic normal bundle, and such a contactomorphism of neighborhoods can be made to realize $\phi$ on its restriction to $C_1$.
\end{cor}

There is more structure coming from parallel transport along the foliation. That is, suppose $v \in \Gamma(\scr{F})$ is a section of the foliation and $j \colon C \hookrightarrow M$ the inclusion. Then
$$\scr{L}_v j^* \alpha = i_v j^* d\alpha$$
is a $1$-form vanishing on $TC \cap \xi$, so $v$ preserves $TC \cap \xi$. Hence, this form is $\nu j^* \alpha$ for some function $\nu$ on $C$. Then also,
$$\scr{L}_v j^*d\alpha = d(\nu j^*\alpha) = d\nu \wedge j^*\alpha + \nu j^*d\alpha,$$
and since $j^*\alpha|_{\xi}=0$, $v$ also preserves the cosymplectic normal bundle. Hence, $\alpha$ descends to the leaf space of the foliation, and locally, it descends as a contact form (since $\pi^*$ is injective and sends the volume form $\alpha \wedge (d\alpha)^k$ to a nonzero form). So we can think of the bundle data necessary for determining the neighborhood of a nonsingular coisotropic as a contact form on $C/\scr{F}$ (at least if the leaf space is a smooth manifold).

\begin{exam} \label{exam:p_in_coisotropic}
Let $C^{2n+1-k} \subset (M^{2n+1},\xi)$ be a nonsingular coisotropic submanifold (with $0 \leq k \leq n$), and let $p \in C$ be any point. Then on some $\Open_C(p)$, the foliation is trivial, and applying the Darboux theorem to the local leaf space, we have that there exist local coordinates $x_1,\ldots,x_n, y_1,\ldots,y_n, z$ on some $\Open_M(p)$ such that
$$\xi = \ker \left( dz + \frac{1}{2} \sum (-y_i dx_i + x_i dy_i) \right)$$
and $C$ is just given by the set $x_1 = \ldots = x_k = 0$. In these coordinates, the leaves of the foliation are parallel to the plane in the $y_1$ through $y_k$ directions.
\end{exam}

We may strengthen this example by including also the data of a complementary isotropic submanifold.

\begin{thm} \label{thm:complementary_isotropic}
Let $C^{2n+1-k} \subset (M^{2n+1},\xi)$ be a coisotropic submanifold and suppose $L^k$ is an isotropic submanifold transverse to $C$ with $L \cap C = \{p\}$. Then on some $\Open_M(p)$, there exist local coordinates $x_1,\ldots,x_n,y_1,\ldots,y_n,z$ such that
$$\xi = \ker \left( dz + \frac{1}{2} \sum (-y_i dx_i + x_i dy_i) \right)$$
and $C$ is the plane given by $x_1 = \cdots = x_k = 0$ whereas $L$ is the plane given by $x_{k+1} = \cdots = x_n = y_1 = \cdots = y_n = z = 0$.
\end{thm}
\begin{proof}
Since $C$ has a transverse complementary isotropic submanifold at $p$, $C$ is automatically nonsingular at $p$. By Example \ref{exam:p_in_coisotropic}, we may fix a neighborhood where $\xi$ is of the desired form and $C$ is the desired plane. Our goal is then to align $L$ with the $x_1\cdots x_k$-plane.\\

\noindent
\textbf{\emph{Step 1}} We can suppose $T_pL$ is a subspace of the span of the vectors $\partial_{x_1}, \ldots, \partial_{x_k}, \partial_{y_1}, \ldots, \partial_{y_k}$. \\

By symplectic linear algebra at $p$, we have, $S_p = T_pL^{\omega} \cap (T_pC \cap \xi_p)$ is a slice representing the conformal symplectic normal bundle at $p$. By Corollary \ref{cor:contact_coiso_nbhd}, we may have assumed furthermore, in our invocation of Example \ref{exam:p_in_coisotropic}, that $S_p$ is precisely spanned by the vectors $\partial_{x_{k+1}}, \ldots ,\partial_{x_n}, \partial_{y_{k+1}}, \ldots, \partial_{y_n}$. Taking the orthogonal complement, we have $T_pL + \scr{F}_p$ is the span of the vectors $\partial_{x_1}, \ldots, \partial_{x_k}, \partial_{y_1}, \ldots, \partial_{y_k}$, hence proving the claim. \\

\noindent
\textbf{\emph{Step 2}} We can suppose $T_pL$ is precisely the span of $\partial_{x_1},\ldots,\partial_{x_k}$. \\

Since $T_pL$ is complementary to $T_pC$, and from the above step, we have that we can write
$$T_pL = \mathrm{span}\langle \partial_{x_i} + \sum_{j=1}^{k}c_{ij}\partial_{y_j} \rangle_{i=1}^{k}$$
for some constants $c_{ij}$. Since $L$ is isotropic, we must have $\omega(w_i,w_j) = 0$, which implies $c_{ij} = c_{ji}$ for all $i,j$. This in turn implies that the isomorphism
$$\phi \colon (x_1,\ldots,x_n;y_1,\ldots,y_n;z) \mapsto (x_1,\ldots,x_n;y_1-\sum c_{1j}x_j,\ldots, y_n - \sum c_{nj}x_j; z)$$
is a contact automorphism, preserving the contact form $\alpha$ (and not just $\xi$) as well as the positioning of $C$. The only thing it changes is that it sends $T_pL$ to the span desired.\\

\noindent
\textbf{\emph{Step 3}} We can ensure $L$ is parallel to the $x_1\cdots x_k$-plane in local coordinates as desired. \\

In order to do this, consider an isotopy of isotropic embeddings $j_t \colon L \rightarrow M$ such that $j_0$ is the standard inclusion and $j_1$ embeds $L$ as the $(x_1,\ldots,x_k)$-plane (all in a neighborhood of $p$), and such that $(dj_t)_p$ is constant. We refer to \cite[Theorem 2.6.2]{Geiges}, which proves that there is some contact vector field $Y_t$ whose flow realizes $j_t$, where $Y_t$ is chosen to correspond to some contact Hamiltonian $H_t$ (with respect to a fixed contact form). The construction of $H_t$ in this proof allows for $H_t$ to be equal to $0$ along $C$, so that $C$ remains fixed. This concludes the proof.
\end{proof}

Finally, we remark how characteristic foliations interact between symplectic and contact geometry. The following proposition follows essentially by computation.

\begin{prop} \label{prop:project_foliation} Suppose either
\begin{itemize}
	\item $V \subset (M,\xi)$ is a coisotropic in a contact manifold, or
	\item $V \subset (M,\omega = d\alpha)$ is a coisotropic in an exact symplectic manifold,
\end{itemize}
so that $\R \times V$ is a coisotropic in the symplectization or contactization, respectively. Both $V$ and $\R \times V$ have characteristic foliations, $\scr{F}$ and $\scr{G}$ respectively. Then,
\begin{itemize}
	\item $\scr{G}$ is $\R$-invariant.
	\item $\pi \colon \R \times V \rightarrow V$ satisfies $\pi_* \scr{G} = \scr{F}$.
	\item If $\scr{F}$ is nonsingular at a point $p \in V$, then $\pi_*$ induces an isomorphism between $\scr{G}_{(p,t)}$ and $\scr{F}_p$. In particular, in the setting for which $V$ is coisotropic in an exact symplectic manifold, $\scr{G}$ is always nonsingular.
\end{itemize}
\end{prop}

\subsection{Convex hypersurfaces} \label{subsec:convex_surfaces}

Finally, the study of convex hypersurfaces, which are a class of coisotropic submanifolds, originated with work of Eliashberg and Gromov \cite{EG} and was largely solidified by Giroux \cite{Giroux}. Unfortunately, these methods have only been satisfactory in understanding the theory for convex surfaces in 3-dimensional contact manifolds, and part of the motivation for this work is to extend this understanding to higher dimensions. We begin with some reminders of contact vector fields, which will be important later when we study convex contact cobordisms.

\begin{defn} \label{defn:cvf_and_expansion}
A \textbf{contact vector field} $X$ on a contact manifold $(M,\xi)$ is a vector field whose flow preserves the contact distribution $\xi$. Equivalently, if we fix a contact form $\alpha$ for $\xi$, then one requires $\scr{L}_X\alpha = \mu \alpha$ for some function $\mu$. We refer to $\mu$ as the \textbf{expansion coefficient} of $X$ with respect to $\alpha$.
\end{defn}

The value of $\mu$ depends on the choice of contact form $\alpha$, and is not an invariant of the contact structure $\xi$ itself. If instead we use the form $\alpha' = e^f\alpha$, then the expansion coefficient is $\mu_{\alpha'} = \mu_{\alpha} + df(X)$. The value of $\mu$ at a point where $X = 0$, however, is invariant.

\begin{prop} There is a 1-1 correspondence between contact vector fields $X$ on $(M,\xi = \ker \alpha)$ and smooth functions $H$ on $M$. This correspondence is given by the equations
$$i_X \alpha = H$$
$$i_X d\alpha = dH(R_{\alpha})\alpha - dH.$$
\end{prop}

\begin{defn} The function $H$ is referred to as a \textbf{contact Hamiltonian} for the contact vector field $X$.
\end{defn}

\begin{rmk}
This correspondence also depends on the choice of contact form $\alpha$. One can state the correspondence in a more invariant way by instead associating $H$ with a section of $TM/\xi$ via the trivialization given by $\alpha$. The correspondence between contact vector fields and sections of $TM/\xi$ is canonical. Indeed, the pairs $(\alpha, H)$ and $(e^f\alpha,e^fH)$ yield the same contact vector field.
\end{rmk}

\begin{rmk} \label{rmk:mu_and_H}
We have
$$\mu \alpha = \scr{L}_X\alpha = di_X\alpha + i_Xd\alpha = dH + (dH(R_{\alpha})\alpha - dH) = dH(R_{\alpha})\alpha.$$
This yields the expression $\mu = dH(R_{\alpha})$.
\end{rmk}

\begin{defn}
A \textbf{convex hypersurface} is a codimension $1$ hypersurface $\Sigma$ embedded in a contact manifold $(M,\xi)$ such that there exists a transverse contact vector field to $\Sigma$.
\end{defn}

\begin{rmk}
This definition is local to the hypersurface, since if such a contact vector field exists locally, then it has a locally defined contact Hamiltonian $H$, which can then be multiplied by a bump function supported in a small neighborhood of $\Sigma$ to produce a global contact Hamiltonian for a contact vector field which matches the initial one near $\Sigma$.
\end{rmk}

%\begin{rmk} \label{rmk:convex_needs_germs}
%When we use the term convex surface, we always assume that there is an ambient contact manifold in a neighborhood, since this is the only way in which the convexity condition makes any sense. Hence, if we say that a manifold $\Sigma^{2n}$ is convex, what we shall mean is that there is a germ of a contact structure on $\Sigma$ so that it is convex.
%\end{rmk}

Suppose $\Sigma$ is a convex hypersurface, and $X$ is a given convex vector field. Then one has that the flow of $X$ preserves $\xi$, and so we can find a neighborhood of $\Sigma$, namely $\Sigma \times (-\epsilon,\epsilon)$, on which $\xi$ is $t$-invariant (where $t$ is the coordinate of $(-\epsilon,\epsilon)$). In fact, the contact form itself can be made $t$-invariant via the following argument which can be found in more detail in \cite[Section 4.6.2]{Geiges}. First, pick any contact form $\alpha$ in this neighborhood. Then $\scr{L}_{\partial_t}\alpha = \mu \alpha$ for some function $\mu$. Taking $\lambda(p,t) = \exp\left(-\int_0^t \mu(p,\tau)d\tau\right)$ for $(p,t) \in \Sigma \times (-\epsilon,\epsilon)$ produces a $t$-invariant form $\lambda \alpha$.

So now, we may write $\xi = \ker(udt + \beta)$ where in these coordinates $X = \partial_t$, $\Sigma = \{t=0\}$, $u \in C^{\infty}(\Sigma)$, and $\beta \in \Omega^1(\Sigma)$. We will use this expression numerous times throughout the rest of this paper. The contact condition is then that $dt \wedge (ud\beta + n\beta \wedge du) \wedge (d\beta)^{n-1} > 0$.

\begin{defn} Given a convex hypersurface $\Sigma$ with a transverse contact vector field $X$, the \textbf{dividing set} $\Gamma$ is the subset of points $p \in \Sigma$ in which $X_p \in \xi_p$.
\end{defn}

\begin{prop} The dividing set $\Gamma$ (with respect to a given contact vector field $X$) is a smooth codimension $1$ hypersurface of $\Sigma$ which is a contact submanifold of $(M,\xi)$ (i.e. $\alpha|_{\Gamma}$ is a contact form).
\end{prop}

Any convex hypersurface is automatically coisotropic since it is codimension 1, and therefore carries its $1$-dimensional singular foliation $\scr{F}$. The following lemma proves that the singular set always avoids the dividing set.

\begin{lem} The foliation $\scr{F}$ is non-singular along $\Gamma$, i.e. $S(\Sigma) \cap \Gamma = \emptyset$. Furthermore, $\scr{F}$ is transverse to $\Gamma$.
\end{lem}

\begin{prop} \label{prop:div_set_nbhd}
One can find a neighborhood $U = \Open_M(\Gamma) \cong \Gamma \times \Open_{\R^2}(0)$ such that, with $\pi \colon U \rightarrow \Gamma$ the projection, and with $u,t$ the coordinates on $\R^2$, we have that $\alpha = udt + \pi^*(\alpha|_{\Gamma})$, and such that in these coordinates, $X = \partial_t$ and $\Sigma = \{t=0\}$.
\end{prop}

\begin{proof}[Proof Sketch]
The foliation on $\Sigma$ is nonsingular near the dividing set, and the dividing set carries the conformal cosymplectic normal bundle of $\Sigma$, so this almost follows as a corollary of Corollary \ref{cor:contact_coiso_nbhd}, except that we need to take a little more care to ensure that $X = \partial_t$. The idea is to use $\scr{F}$ as a $u$-axis, and the result follows almost immediately.
\end{proof}

The dividing set $\Gamma$ divides $\Sigma$ into two regions, $R_+$ and $R_-$, on which, with respect to the form $\alpha = udt+\beta$, one has $u > 0$ and $u < 0$ respectively. Furthermore, away from the set where $u = 0$, we can write $\alpha = dt + (\beta/u)$, and so the contact condition is precisely that $d(\beta/u)$ is symplectic on $R_{\pm}$. So each of these is an exact symplectic manifold with respect to the primitive $\lambda = \beta/u$, which is independent of the particular choice of contact form.

On an exact symplectic manifold $(W,d\lambda)$, one can define a Liouville vector field $Z_{\lambda}$ by the condition $i_{Z_{\lambda}}d\lambda = \lambda$. In the case described, since we are using $\lambda = \beta/u$, which blows up near the dividing set, $Z_{\lambda}$ will point towards the dividing set (one may deduce this from Proposition \ref{prop:div_set_nbhd}). This type of convexity condition ensures that $R_{\pm}$ are actually (finite-type) Liouville manifolds. We will exclude the prefix finite-type for the remainder of this paper, encoding this technicality in the following definition.

\begin{defn}
A \textbf{Liouville manifold} is an exact symplectic manifold $(W,d\lambda)$ such that $W = W' \cup \partial W' \times [0,\infty)$, where $W'$ is a compact domain, and such that $\lambda = e^t\beta$ on $\partial W' \times [0,\infty)$, where $\beta$ is a $1$-form on $\partial W'$. It is then automatic that $Z_{\lambda} = \partial_t$, and that $\beta = \lambda|_{\partial W'}$ is a contact form on $\partial W'$. The Liouville manifold $(W,\lambda)$ is said to be a \textbf{(completed) Liouville filling} of the contact manifold $(\partial W',\xi = \ker \beta)$ (which is well-defined up to contactomorphism by the data of $(W,\lambda)$ without reference to the choice of compact part $W'$ but with reference to a particular choice of primitive $\lambda$).

For later, we also consider the compact part $(W',\lambda)$, through which the Liouville field is outwardly transverse to the boundary. This is called a \textbf{Liouville domain}, and we shall also refer to $(W',\lambda)$ as a \textbf{Liouville filling} of its boundary $(\partial W', \lambda|_{\partial W'})$, which is a contact manifold.

Finally, we will also consider \textbf{Liouville cobordisms}, which are compact with Liouville vector field transverse to the boundary, possibly inwardly transverse along some boundary components.
\end{defn}

\begin{rmk}
For any connected Liouville cobordism, the Liouville vector field $Z_{\lambda}$ is always outwardly transverse to at least one boundary component. This is because $\scr{L}_X\omega = \omega$, so $X$ is volume expanding.
\end{rmk}

To summarize, the contact germ of $\Sigma$, with explicit transverse contact vector field $X$, is precisely equivalent to the data of $(R_{\pm},\lambda_{\pm})$, Liouville fillings of contact manifolds $(\Gamma,\pm \ker \beta)$ with reversed coorientation, which match along the dividing set as in Proposition \ref{prop:div_set_nbhd}. (Note that $\beta$ is not actually an invariant, but $\ker \beta$ is.)

\begin{rmk}
On $R_{\pm}$, we have the characteristic foliation $\scr{F}$ (as a convex hypersurface of a contact manifold) is spanned by $Z_{\lambda}$.
\end{rmk}

\begin{rmk} \label{rmk:ideal}
The structure on $\overline{R_{\pm}} = R_{\pm} \cup \Gamma$ is what Giroux has defined to be an \textbf{ideal Liouville domain} \cite{Giroux_Ideal} (see also \cite{MNW} for its first use in the literature). This notion is better than the notion of a completed Liouville manifold. The basic reason is that it allows one to forget about the precise choice of primitive $\lambda$ for $\omega$, because the data of $\omega$ together with the ideal boundary as a smooth boundary component automatically determines the contact structure on the boundary, and hence the behavior at infinity. Meanwhile, for a completed Liouville manifold without the ideal boundary, one cannot even determine the contact boundary from the data of $\omega$ itself. Furthermore, if we wish to consider Lagrangian fillings (in $R_{\pm}$) of a Legendrian (in $\Gamma$), one usually needs to be careful about the asymptotics at infinity. For example, one can be strict and consider cylindrical ends parallel to the Liouville vector field, but this again depends upon $\lambda$, not $\omega$. In the world of ideal Liouville manifolds, we can instead consider exact Lagrangian submanifolds with boundary on $\Gamma$ which are transverse to $\Gamma$, a very natural asymptotic condition. And perhaps most importantly, it is easiest to state and prove Moser-type theorems for ideal Liouville domains since they naturally encode the behavior at infinity required in the completed Liouville manifold setting.
\end{rmk}

\section{Balanced coisotropic spheres} \label{sec:BCS}

The goal of this section is to describe the neighborhoods of the submanifolds of convex hypersurfaces which will form our attaching regions for contact handles. The main focus is on balanced coisotropic spheres (see Definition \ref{defn:bcs}) which are the attaching spheres for high-index handles. The low-index handle attaching spheres are simply neighborhoods of isotropic spheres in dividing sets, and are therefore only discussed minimally in this section (Proposition \ref{prop:attaching_data}).

\subsection{Spin-symmetric coisotropic spheres in contact manifolds}

Although we know how coisotropic spheres look near nonsingular points, the singular points may in general be quite complicated. We describe a certain type of coisotropic sphere with singularities which can be dealt with explicitly.

\begin{defn}
Consider a sphere $S^k \subset \R^{k+1}$. Fix an axis $A = \R^r$, and consider all hyperplanes $H_s$ of dimension $r+1$ containing this axis $A$, parametrized by $s \in \mathbb{P}A^{\perp} \cong \R P^{k-r}$. One obtains a singular foliation consisting of leaves $(H_s \setminus A) \cap S^k$ with singular set $A \cap S^k \cong S^{r-1}$. A singular foliation $\scr{F}$ on a sphere $S^r$ is said to be \textbf{spin-symmetric} if it matches this model up to diffeomorphism. The singular sphere is called the \textbf{binding}. In particular, a \textbf{spin-symmetric coisotropic sphere} is a coisotropic sphere with spin-symmetric characteristic foliation. See Figure \ref{fig:Spin-symmetric}.
\end{defn}

\begin{figure}[h!]
\centering
  \includegraphics[width=\textwidth]{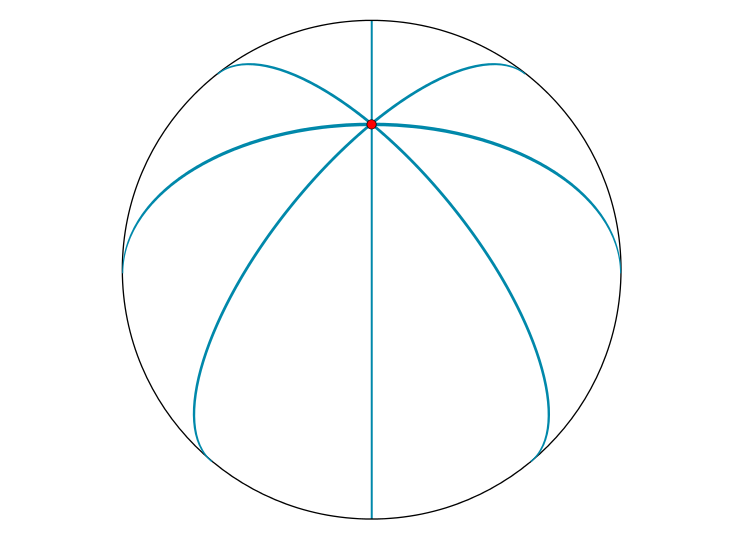}
  \caption{A spin-symmetric foliation for $k=2$ and $r=1$. The binding is an $S^0$, where the other point in the binding is out of view on the backside of the sphere. This case matches the foliation of the standard convex $S^2$ given as the boundary of a standard Darboux ball $D^3 \subset \R^3$.}\label{fig:Spin-symmetric}
\end{figure}

\begin{rmk}
The nonsingular leaves of the foliation $\scr{F}$ have dimension $r$.
\end{rmk}

\begin{rmk}
We always have $k+r$ is odd for a spin-symmetric coisotropic sphere. In particular, $k+r$ is always the dimension of the ambient contact space.
\end{rmk}

\begin{rmk}
In the case when $k=r+1$, this foliation yields an open book, explaining the use of the term binding.
\end{rmk}

For a spin-symmetric coistropic sphere $C \cong S^k$ in a contact $2n+1$-manifold, with $k > n$, at any point $p$ in the singular set $S = S^{2n-k}$, we can consider the conformal cosymplectic normal bundle $(TC \cap \xi)/(TC \cap \xi)^{d\alpha}$. The discussion after Corollary \ref{cor:contact_coiso_nbhd} shows that the flow of any vector field tangent to the foliation preserves the conformal cosymplectic normal bundle, even when there is a singular set. Hence, any vector field tangent to $S$ and all of the nonsingular leaves preserves the conformal cosymplectic normal bundle over $S$. This allows us to prove a neighborhood theorem for spin-symmetric coisotropic spheres which only requires as input the conformal cosymplectic normal bundle data at a single point in $S$.

\begin{thm} \label{thm:spin-sym_nbhd}
Any two spin-symmetric coisotropic spheres of the same dimension in contact manifolds of the same dimension have contactomorphic neighborhoods via a contactomorphism lying over any diffeomorphism preserving the foliations and preserving the conformal cosymplectic normal bundle at a single point $p \in S$.
\end{thm}

\begin{proof}
We wish to construct $\Phi$ as in the input of Theorem \ref{thm:gen_nbhd_contact}. There are two main steps. The first is to prove that any such diffeomorphism preserving the foliations and the conformal cosymplectic normal bundle at a single point $p \in S$ automatically preserves the conformal symplectic structure throughout the entire coisotropic. The second step is to then construct $\Phi$ by including data normal to the coisotropic.

Let us dive into the first step. As discussed just prior to the statement of this theorem, any vector field parallel to the foliation preserves the conformal cosymplectic normal bundle, even along the singular set. For any $q \in S$, we can find a vector field which is tangent to $S$ and which flows from $p$ to $q$. Therefore, the conformal cosymplectic normal bundle at $p$ determines the bundle along all of $S$. Similarly, consider any leaf $L$ of the foliation, and consider its conformal cosymplectic normal bundle $\nu_L = TL^{d\alpha}/TL$. This bundle extends continuously to $\overline{L} = L \cup S$. At $q \in S$, we have $(\nu_L)_q$ is just the coisotropic reduction of $T_qS^{d\alpha}/T_qS$ with respect to the isotropic $T_qL/T_qS$. Hence, $(\nu_L)|_S$ is determined along $S$ by the conformal cosymplectic normal bundle to $S$. We pick a neighborhood $\Open_C(S) \cong S \times D^{2\ell}$ where the leaves are just $S \times r$ with $r$ a radial ray in $D^{2\ell}$. The vector field which is radial in the $D^{2\ell}$ coordinates is everywhere tangent to the leaves, and hence the conformal cosymplectic normal bundle of each $\nu_L$ is radially invariant in this neighborhood. Since we have determined the conformal symplectic structure on $(\nu_L)|_S$, the conformal symplectic structure on $L \cap \Open_C(S)$ is also determined. One then extends this data to all of $L$ using vector fields tangent to the foliation. To summarize, the conformal cosymplectic normal bundle data on all of a given spin-symmetric coisotropic submanifold is completely determined by the data at a single point $p \in S$ together with the topological data of the foliation. Hence, any diffeomorphism of coisotropic submanifolds preserving these two pieces of information automatically intertwines the conformal cosymplectic normal bundle data. That is, it preserves $TC \cap \xi$ and the conformal class of $d\alpha|_{TC \cap \xi}$.

For the second step, we now suppose that we have such a diffeomorphism $\phi \colon C_1 \rightarrow C_2$, and ask how we can construct $\Phi \colon TM_1|_{C_1} \rightarrow TM_2|_{C_2}$ so as to complete the proof. Essentially, we can construct $\Phi$ over $\Open_C(S)$, which is sufficient to obtain the result.

To begin the construction on $\Open_C(S)$, let us be a bit more precise about the behavior of the contact structure near $S$. Consider the neighborhood $\Open_C(S) \cong S \times D^{2\ell}$, where the leaves of the foliation of the form $S \times r$ for $r$ a radial ray of $D^{2\ell}$. Using coordinates $x_1,y_1,\ldots,x_\ell,y_\ell$ for $D^{2\ell}$, the above argument proved that $\xi$ was invariant under radial scaling and determined by the conformal symplectic structure along $S = S \times\{0\}$. It follows that $\xi \cap TC = \ker \sum (x_kdy_k - y_kdx_k)$, where the conformal symplectic structure along $S$ is $\sum dx_k \wedge dy_k$ (up to scaling). Hence, we must have $\alpha|_C = f\sum(x_kdy_k - y_kdx_k)$ where $f$ is a function which is everywhere positive except possibly at the origin. But we also know $d\alpha|_C$ is nonzero along $S$, and so $f$ must also be nonvanishing at $S$. Therefore, by scaling, we may assume $\alpha|_C=\sum (x_kdy_k - y_kdx_k)$.

To be a bit more explicit, we have two neighborhoods $\Open_{C_0}(S_0) \cong S^k \times D^{2\ell} \cong \Open_{C_1}(S_1)$, where the neighborhoods are identified by $d\phi$, and where each $\alpha_i$ on $C_i$ is scaled so that $\alpha_i|_{C_i} = \sum (x_kdy_k - y_kdx_k)$.

On $U = \Open_C(S) = S \times D$, note that the Reeb field $R_{\alpha}$ is transverse to $C$, since $TC|_S \subset \xi$. Then we have
$$TM|_U = TS \oplus A \oplus TD \oplus \R\langle R_{\alpha}\rangle,$$
where $A$ is chosen so that $TS \oplus A \subset \xi$ is the symplectic orthogonal complement to $(TC \oplus \R\langle R_{\alpha}) \rangle \cap \xi$ with respect to $d\alpha$. Fix a metric on $S$. We define an automorphism $J \colon TM|_U \rightarrow TM|_U$ by
\begin{itemize}
	\item $J \colon TS \oplus A \rightarrow TS \oplus A$ is an almost complex structure compatible with $d\alpha$ and intertwining the direct summands, and such that $d\alpha(Jv,v) = |v|^2$ for $v \in TS$.
	\item $J \colon TD \oplus \R\langle R_{\alpha} \rangle \rightarrow TD \oplus \R\langle R_{\alpha} \rangle$ is given by $v+tR_{\alpha} \mapsto iv+(t+\alpha(v)-\alpha(iv))R_{\alpha}$ where $i$ is the standard `multiplication by $i$' map on the tangent bundle of $TD^{2\ell} \cong T\C^{\ell}$.
\end{itemize}
Notice that on $\xi$, $J$ is an almost complex structure compatible with $d\alpha$, and with $JR_{\alpha} = JR_{\alpha}$. We define $\Phi \colon TM_1|_{U_1} \rightarrow TM_2|_{U_2}$ by matching up the summands $TS$ and $TD$ via $d\phi$, $\R\langle R_{\alpha} \rangle$ by the identity, and by defining $\Phi \colon A_1 \rightarrow A_2$ by $\Phi(J_1v_1) = J_2d\phi(v_1)$ for all $v_1 \in TS_1$. Then $\Phi$ preserves the conformal symplectic structure $d\alpha$. (Note that this is stronger than $\Phi$ being $(J_1,J_2)$-holomorphic, and requires us to have picked the same metric on $S_1$ as on $S_2$ under the identification by $d\phi$.)

We have constructed the neighborhood over $\Open_C(S)$, and the remaining region is a nonsingular coisotropic, so it suffices to prove a neighborhood theorem for nonsingular coisotropics where we have already specified the neighborhood at the boundary. We appeal to the parametric relative neighborhood theorem, Theorem \ref{thm:par_rel}. (We actually only need the result here at the level of $\pi_0$-equivalence, though everything about this proof holds in higher parametric families.)
\end{proof}

\begin{rmk}
If the reader is not content with appealing to Theorem \ref{thm:par_rel}, an alternative argument can be made, in which one also constructs $\Phi$ over $C \setminus S$ explicitly. The idea is similar to the construction near $S$, but we instead use the decomposition
$$TM|_U = TL \oplus JTL \oplus (TC \cap \xi)/TL \oplus TC/(TC \cap \xi),$$
where $L$ is a leaf. We choose an explicit slice for the conformal symplectic bundle $(TC \cap \xi)/TL$ and a fixed $J$ compatible with the conformal symplectic normal bundle on this slice, as well as a section of the line bundle $TC/(TC \cap \xi)$. This decomposition again lets us define $\Phi$ term by term. The final step is to then piece this together with $\Phi$ over $\Open_C(S)$. This is a fairly straightforward partition-of-unity type argument. Details can be found in a paper of Huang \cite[Lemma 3.6]{Huang1}, which proves the above result for the case of an $(n+1)$-dimensional spin-symmetric coisotropic sphere.
\end{rmk}

\begin{rmk}
The normal bundle to a spin-symmetric coisotropic sphere $C$, which near the singular set we saw was just $JTS \oplus \R$, is trivial because $TS$ is trivial after a single stabilization for any sphere $S$.
\end{rmk}

\subsection{Neighborhoods for submanifolds of convex hypersurfaces}

The following proposition is all we will need to understand low-index (subcritical) handle attachments.

\begin{prop} \label{prop:attaching_data}
Any isotropic submanifold $L$ of the diving set $\Gamma$ has a standard neighborhood of the form $udt + \beta$ where $\beta$ is a contact form on $\Open_{\Gamma}(L)$ depending parametrically only upon framing data for the conformal symplectic normal bundle of $L$ in $\Gamma$. In particular, if the conformal symplectic normal is trivial of rank $2\ell$, then germs of symplectic neighborhoods of $L$ are parametrized by homotopy classes of maps from $L$ to $Sp(2\ell,\R)$.
\end{prop}
\begin{proof}
This follows from Proposition \ref{prop:div_set_nbhd} and the parametric (fiberwise) version of Theorem \ref{thm:INT}.
\end{proof}

We shall typically be concerned with the case when $L$ is an isotropic sphere $S^k$ with trivial conformal symplectic normal bundle in the dividing set $\Gamma^{2n-1}$ of a convex hypersurface in a contact manifold $M^{2n+1}$, in which case the symplectic normal bundle is of rank $2\ell = 2(n-k-1)$. At the most basic level, we shall only be concerned with equivalence up to $1$-parametric homotopy (i.e. $\pi_0$), corresponding to $[L=S^k,Sp(2\ell,\R)]=\pi_k(Sp(2\ell,\R))$ (the symplectic group is connected and $\pi_1 = \Z$ is abelian, so this association is canonical). For example, if $k=0$, there is only one choice of framing up to homotopy, while if $k=1$ and $2n-1 \geq 5$, then there is a $\Z$-worth. If $k=n-2$, then $Sp(2,\R) = SL(2,\R)$ deformation retracts onto $SO(2,\R) = S^1$, and so we see that the set of framings up to homotopy is trivial unless $k=1$, in which case $2n-1 = 5$. In particular, framing data only becomes relevant when $2n+1 \geq 7$. As we shall see later, there is no coincidence that framing data is similarly relevant for Weinstein handle decompositions of dimension $2n$ when $2n \geq 6$.

We now turn to the neighborhood theorem we will need for high-index (supercritical) handle attachments. We again suppose that we have a convex hypersurface $\Sigma$ with an explict transverse contact vector field $X$, so that there is a well-defined dividing set.

\begin{defn}
The \textbf{standard nonsingular foliation} of rank $r\leq k$ on the compact ball $D^k$ in standard Euclidean space $\R^k$ consists of a singular leaf $S^{r-1}$ given by $\R^r \cap \partial D^k$, where $\R^r$ is some coordinate plane, and with nonsingular leaves given by ellipsoids (of dimension $r$) passing through $S^{r-1}$. See Figure \ref{fig:Standard_fol}.
\end{defn}

\begin{rmk}
Note that, by definition, the characteristic foliation along $\partial D^k$ is spin-symmetric, with binding given by $S^{r-1}$.
\end{rmk}

\begin{figure}[h!]
\centering
  \includegraphics[width=\textwidth]{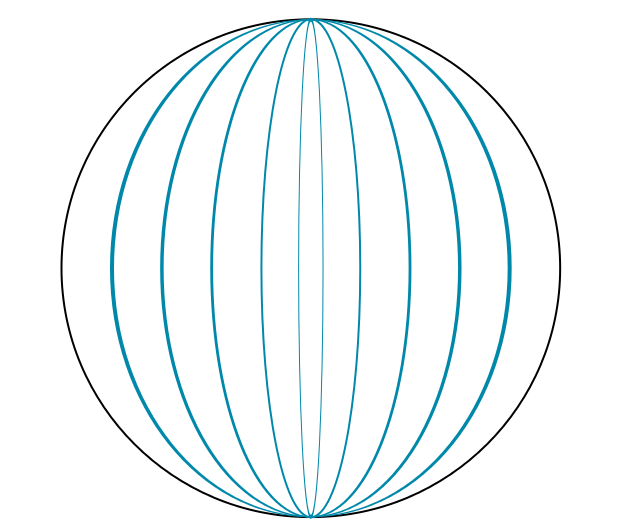}
  \caption{The standard foliation with $k=2$ and $r=1$. We see that along the boundary, we have a spin-symmetric sphere with $k=1$ and $r=1$. Notice that near the binding of the boundary, the foliation is not simply radial. Warning: this example does not realize a foliation which will appear in our setting, since for a balanced coisotropic disk, the foliation on each $D_{\pm}$ has $k+r$ is even. It is included here since it is the simplest nontrivial example of a standard foliation one can hope to draw.}\label{fig:Standard_fol}
\end{figure}

\begin{defn} \label{defn:bcs}
A \textbf{balanced coisotropic sphere} in a convex hypersurface $\Sigma$ (with transverse germ of contact vector field) consists of a smoothly embedded sphere $S^k = D_+^k \cup_{S^{k-1}} D_-^k \hookrightarrow \Sigma$, such that the embedding is transverse to the dividing set along the \textbf{equator} $S^{k-1}$, and such that the hemispheres, $D_{\pm}$, are embedded in $\overline{R_{\pm}}$ so that their characteristic foliations (with respect to the Liouville form on $R_{\pm}$ in the interior, as well as with respect to the contact structure on $\Gamma$ on the boundary) is the standard nonsingular foliation.
\end{defn}

\begin{rmk}
If the ambient contact manifold has dimension $2n+1$, then we must have $k \geq n$, and the rank of the characteristic foliation is then given by $r = 2n-k$.
\end{rmk}

\begin{rmk}
In the case of a balanced coisotropic sphere with $k=n$, we have simply a pair of Lagrangian disk fillings of a Legendrian. Note that since $H^1(D^n) = 0$, Lagrangian disk fillings are automatically exact Lagrangians.
\end{rmk}

\begin{thm} \label{thm:bcs_nbhd}
For fixed $n,k$, any two balanced coisotropic spheres $C_1,C_2$ in a convex hypersurface $\Sigma$ of a contact manifold  have neighborhoods which are symplectomorphic in the sense that there is a diffeomorphism of neighborhoods $\phi \colon U := \Open_{\Sigma}(C_1) \rightarrow V:= \Open_{\Sigma}(C_2)$ such that
\begin{itemize}
	\item $\phi$ restricts to a contactomorphism from $U \cap \Gamma$ to $V \cap \Gamma$, and
	\item $\phi$ is a symplectomorphism from $U \cap R_{\pm}$ to $V \cap R_{\pm}$
\end{itemize}
This result holds parametrically, in that the germ of $\phi$ is determined parametrically up to framing of the conformal symplectic normal bundle at a single point in the binding of the equator.
\end{thm}

\begin{proof}
Since the equators are spin-symmetric coisotropic spheres of the same dimension, they automatically have contactomorphic neighborhoods in $\Gamma$. Fix such a contactomorphism of neighborhoods $\psi \colon \Open_{\Gamma}(C_1 \cap \Gamma) \rightarrow \Open_{\Gamma}(C_2 \cap \Gamma)$.

Next we wish to understand the conformal cosymplectic normal bundle on $C_i \cap R_{\pm}$. Since flowing along the characteristic foliation preserves this bundle, all of the data is encoded in a neighborhood of the binding of the equator. Also, notice that on constant-$u$ slices $S$ near $\Gamma$, we have that $d\lambda|_S = d(\beta/u)|_S = \frac{1}{u}d\beta|_S$. So by taking constant-$u$ slices of $C_i$ parallel to $\Gamma$, we find that the leaf space $(C_i \cap R_{\pm})/\scr{F}$ is just an infinite standard symplectic disk $\R^{2\ell}$ of the correct dimension. Fix a diffeomorphism $\wt{\psi} \colon C_1 \rightarrow C_2$ which is $\psi$ along $C_1 \cap \Gamma$ and which preserves this conformal cosymplectic bundle.

Finally, we have on $R_{\pm}$ an almost complex structure $J$ which is standard in the typical way along the symplectization. (Here, recall that $|u| = e^{-\tau}$ where $\tau$ is the symplectization coordinate, and the condition is that $J$ is $\tau$-invariant, with $J\xi = \xi$ and $J\partial_{\tau} = R_{\beta}$.) Then $J\scr{F}$ is everywhere transverse to $C \cap R_{\pm}$. Furthermore, the asymptotics imply that $J\scr{F}$ extends continuously to $\Gamma$ precisely as the normal bundle we had for the equator in the neighborhood theorem for spin-symmetric coisotropic spheres, lying over the diffeomorphism $\wt{\psi}$ we previously cooked up. Hence, we may apply Moser's trick with this extra asymptotic condition to yield the result.

For the parametric version, one simply needs to check that every step can be performed parametrically.
\end{proof}

\section{Recollections of Weinstein surgery theory} \label{sec:Weinstein}

Here we recall the definition of a Weinstein manifold, along with its surgery theory. We assume that the reader is familiar with standard Morse theory. We will not discuss any proofs of these results, though the reader interested in more details can find them in the book of Cieliebak and Eliashberg \cite[Part 4]{CE}.

\begin{defn}
A \textbf{Weinstein cobordism} consists of a quadruple $(W,\omega,X,\phi)$ where:
\begin{itemize}
	\item $(W,\omega)$ is a compact symplectic manifold with boundary.
	\item $X$ is a Liouville vector field for $(W,\omega)$, meaning that $\scr{L}_X\omega = \omega$. In particular, $(W,\omega)$ is an exact symplectic manifold with $\omega = d\lambda$ where $\lambda = i_X\omega$.
	\item $\phi \colon W \rightarrow \R$ is a Morse function.
	\item $(X,\phi)$ form a \textbf{gradient-like (or Lyapunov) pair}, meaning that for a fixed metric on $W$, there is a constant $\delta > 0$ such that $d\phi(X) \geq \delta(|X|^2+|d\phi|^2)$.
	\item The boundary $\partial W = \partial_+ W \sqcup \partial_- W$ consists of two regular level sets for $\phi$, the maximum and minimum.
\end{itemize}

If $\partial_-W = \emptyset$, we call $W$ a \textbf{Weinstein domain}.
\end{defn}

\begin{rmk}
The last condition can be replaced by the weaker condition that $X$ is transverse to $\partial W$, but in accordance with the choices made in this paper, we choose this slightly stronger definition.
\end{rmk}

\begin{rmk}
If one drops the Morse function in a Weinstein cobordism one obtains a Liouville cobordism. However, this extra taming of the Liouville field is important for the discussion that follows. In general, we know very little about what a Liouville cobordism can look like.
\end{rmk}

The geometry of a Weinstein cobordism is strongly controlled. The fact that $\scr{L}_X\omega = \omega$ means that the vector field $X$ must always be expanding. Hence, for example, we see that if $W$ has dimension $2n$, there can never be an index $2n$-critical point, since then $X$ would contract. In fact, more is true.

\begin{prop} The descending manifold of a critical point is an isotropic submanifold, and the ascending manifold is coisotropic. In particular, if $W$ is of dimension $2n$, then every critical point has index at most $n$.
\end{prop}

From this theorem, one already can find a coordinate neighborhood around a given critical point so that the symplectic form is of some specified standard model. The only other data local to a critical point must therefore arise from the pair $(X,\phi)$. Although there may be interesting configurations, in reality, one is only concerned about Weinstein manifolds up to some notion of equivalence. We allow for homotopies of the form $(\omega_t,X_t,\phi_t)$ through Weinstein structures. The following is a sort of Morse Lemma for the Weinstein setting.

\begin{prop}{\cite[Proposition 12.12]{CE}}
One can homotope the Weinstein structure on a cobordism in a neighborhood of the critical points so that the critical points, their critical values, and the descending isotropic submanifolds remain fixed, and such that in the end, a neighborhood of each critical point matches a specified model depending only upon the index. In particular, every Weinstein manifold is Weinstein homotopic to one such that the critical points are standard in this way.
\end{prop}

The level sets of a Weinstein manifold are naturally contact, with contact form given by the restriction of $\lambda = i_X\omega$. Hence, the descending manifolds of a critical point of a Weinstein manifold will intersect a given level set along an isotropic sphere. Furthermore, the geometry of the situation implies that the the conformal symplectic normal bundle of the attaching sphere matches up with the symplectic normal bundle to the isotropic descending disk.

We make this a little more precise. Consider the following data on $\R^{2n}$, for a fixed index $0 \leq k \leq n$:
\begin{itemize}
	\item $\omega = \sum_{i=1}^{n} dx_i \wedge dy_i$
	\item $X = \sum_{i=1}^{k} (2x_i\partial_{x_i}-y_i\partial_{y_i}) + \frac{1}{2}\sum_{i=k+1}^{n} (x_i\partial_{x_i} + y_i\partial_{y_i})$
	\item $\phi = \sum_{i=1}^{k} (x_i^2-y_i^2) + \sum_{i=k+1}^{n} (x_i^2+y_i^2)$
\end{itemize}
Then the origin provides the aforementioned standard model for the critical points. By taking a neighborhood of the descending isotropic submanifold along the $y_1,\ldots,y_k$-plane, we may view this as a handle attachment of index $k$ along the level set $\phi^{-1}(-\epsilon)$. This is what we mean by a \textbf{Weinstein handle attachment}. (We keep this slightly imprecise for now; we will see the nitty gritty details in the contact setting in Section \ref{sec:attach}.) One checks that the attaching sphere is indeed a framed isotropic submanifold of the level set $\phi^{-1}(-\epsilon)$ (where, by a framing, we mean a framing of the conformal symplectic normal bundle).

More generally, suppose we have a Weinstein cobordism $W^{2n}$ along with a framed isotropic sphere $S^{k-1} \subset \partial_+W$. By matching this framed isotropic sphere with the framing attaching sphere determined by the standard model for a critical point of index $k$, one can find a new Weinstein domain $W'$, with $W \hookrightarrow W'$, such that there is one critical point in $W' \setminus W$, with $W' \setminus W$ equivalent to the model provided above. In other words, $W'$ is the result of a handle attachment along $W$ with attaching data given by our framed isotropic sphere. Furthermore, the resulting Weinstein cobordism $W'$ is essentially unique, since the gluing occurs along a neighborhood of the attaching sphere, which is determined up to contactomorphism by the framing data.

The discussion above proves the following theorem.

\begin{thm} Up to Weinstein homotopy, even without birth-death type singularities, a Weinstein cobordism $W$ can be built out of a sequence of handle attachments along $\partial_-W$.
\end{thm}

Furthermore, suppose $W'$ is built from $W$ by attaching a single handle along a framed isotropic sphere $S' \subset \partial_+ W$, and similarly $W''$ is built by attaching a single handle along a possibly different framed isotropic sphere $S'' \subset \partial_+ W$. If $S'$ and $S''$ are isotopic through framed isotropic submanifolds, then that isotopy can be extended to an ambient contactomorphism of $\partial_+ W$. But one can simply attach a cylinder $\partial_+ W \times I$ with a Weinstein structure whose holonomy map $\phi \colon \partial_+ W \times \{0\} \rightarrow \partial_+W \times \{1\}$ (given by following trajectories of the Liouville vector field) realizes this contactomorphism. By including this cylinder before attaching the handle which forms $W''$, the new descending disk for the critical point attaches along $\partial_+ W \times \{0\}$ along $S'$, and so $W'$ and $W''$ are handle attachments along the same framed sphere, hence homotopic.

The notion of homotopy through Weinstein structures discussed so far is a little bit too weak, since for example, it fixes the number of critical points. In standard Morse theory, if there are two critical points of neighboring index with a single trajectory between them, the critical points can be smoothly cancelled; by reversing the homotopy, we allow for critical points to be created. However, such a homotopy, since it creates or destroys critical points, cannot be such that each $(X_t,\phi_t)$ is gradient-like and each $\phi_t$ is Morse. Instead, there must be some time along the homotopy at which the two critical points coincide. For genericity reasons, it suffices to consider a specific kind of creation and cancellation in $1$-parameter families.

\begin{defn} \label{def:birth-death} An \textbf{embryonic critical point} for a 1-parameter family of functions $\phi_t$ is one locally modelled in coordinates $(x_1,\ldots,x_n)$ by
$$\phi_t(x_1,\ldots,x_n) = x_1^3 + \sum_{i=2}^{n} \pm x_i^2.$$
A \textbf{birth-death type singularity} is a $1$-parameter family of of vector fields and functions $(X_t,\phi_t)$ such that the gradient-like condition $d\phi(X) \geq \delta(|X|^2+|d\phi|^2)$ is satisfied, and such that $\phi_t$ is modelled in coordinates $(x_1,\ldots,x_n)$ by
$$\phi_t(x_1,\ldots,x_n) = x_1^3 \pm tx_1 + \sum_{i=2}^{n} \pm x_i^2,$$
hence passing through an embryonic critical point at $t=0$. In the case of a plus sign in front of the variable $t$ in these equations, we have a death-type singularity (there are two critical points for $t < 0$ and none for $t > 0$), and similarly the minus sign case is a birth-type singularity.
\end{defn}

\begin{rmk} A generic $1$-parameter family passing through an embryonic critical point will be of birth-death type. For an example of a non-generic $1$-parameter family, one could consider $\phi_t(x_1,\ldots,x_n) = x_1^3 + t^2x_1 + \sum_{i=2}^{n} \pm x_i^2$, which again passes through an embryonic critical point, but is not of birth-death type.
\end{rmk}

\begin{defn} A \textbf{Weinstein homotopy} on a compact manifold with boundary $W$ is a family of Weinstein structures $(\omega_t,X_t,\phi_t)$, $0 \leq t \leq 1$, which is allowed to have birth-death type singularities in $(X_t,\phi_t)$.
\end{defn}

This yields a good notion of equivalence, since Weinstein homotopies induce symplectomorphisms on their completions and allow for handle cancellation and creation as in the smooth theory. In the Weinstein setting, we understand handle cancellation quite well.

\begin{prop}
Suppose $W$ is a Weinstein cobordism with two critical points $p$ and $q$ of index $k$ and $k+1$ respectively, for some $k \geq 0$. Suppose furthermore that there is precisely one trajectory along $X$ from $p$ to $q$ along which the ascending manifold of $p$ and the descending manifold of $q$ intersect transversely. Then there is a Weinstein homotopy which cancels these critical points.
\end{prop}

In this way, there is a calculus for surgery data of Weinstein cobordisms. In four dimensions, for example, there can only be $0$, $1$, and $2$-handles. Suppose $\partial_-W = \emptyset$. If $W$ is connected, we can always cancel $0$ and $1$ handles until there is precisely one $0$-handle. The framing data for the $1$-handles is trivial. Hence, all of the information is encoded simply by the number of $1$-handles, followed by the Legendrian knots along which the $2$-handles are attached. See Gompf's paper \cite{Gompf} for more details on the precise surgery diagrams one can draw. In this framework, $1$- and $2$-handle cancellations can also be realized diagramatically. See, for example, \cite[Proposition 2.18]{CM}, also for the higher dimensional $(n-1)$- and $n$-handle cancellation.

\section{Convex contact cobordisms and their critical points} \label{sec:CCMs}

As we have seen, the surgery description of Weinstein cobordisms boils down to the key fact that we have a Weinstein Morse lemma for the critical points. With a little more effort, this proved that we can completely understand a Weinstein cobordism by surgery data consisting of framed isotropic attaching spheres along contact level sets.

On the contact side, we will develop a similar surgery theory for so-called convex contact cobordisms in analogous stages.
\begin{itemize}
	\item After defining convex contact manifolds, we will prove the proper version of the Morse Lemma in this setting, Theorem \ref{thm:standard_cp}, stating that we are able to standardize convex contact structures around critical points in such a way that their descending manifolds remain fixed near the critical point.
	\item The goal of the following section is then to use this result to prove that every convex contact cobordism can be realized by handle attachments, Theorem \ref{thm:main_thm_body} (= Theorem \ref{thm:main_thm}). We can understand the attaching data for these handles quite explicitly, allowing for a complete surgery theory up to strict convex contact homotopy. The issue of allowing for creation and cancellation of critical points is relegated to Section \ref{sec:homotopy}.
\end{itemize}

\subsection{Basics of convex contact manifolds}
\begin{defn}
A \textbf{(closed) convex contact manifold} is a quadruple $(M,\xi,X,\phi)$ such that
\begin{itemize}
	\item $(M,\xi)$ is a contact manifold
	\item $X$ is a contact vector field on $(M,\xi)$
	\item $\phi$ is a Morse function on $M$
	\item $(X,\phi)$ is a gradient-like pair
\end{itemize}

A \textbf{convex contact cobordism} is also a quadruple of the above form, but such that $M$ possibly has boundary $\partial M = \partial_+ M \sqcup \partial_- M$ such that
\begin{itemize}
	%\item $X$ is transverse to $\partial M$, pointing inwards along $\partial_- M$ and outwards along $\partial_+ M$
	\item $\partial_{+} M$ and $\partial_{-}M$ are each regular level sets for $\phi$
	\item $\phi$ attains its minimum along $\partial_- M$ and its maximum along $\partial_+ M$.
\end{itemize}

In either the closed or relative case, the pair $(X,\phi)$ is referred to as a \textbf{convex contact structure} on $(M,\xi)$.
\end{defn}

\begin{rmk} \label{rmk:equiv_notions_convex_contact_cobordism}
That $\partial_{\pm}M$ are regular level sets automatically implies that $X$ is transverse to $\partial M$, pointing inwards along $\partial_- M$ and outwards along $\partial_+ M$. One may wish to work with a weaker notion of convex contact cobordism, in which one requires only this transversality. Nothing is lost; in Remark \ref{rmk:def_retract}, we show that the inclusion of the strict version to the weak one induces a weak homotopy equivalence.
\end{rmk}

We will care about convex contact manifolds, and more generally cobordisms, up to some natural notion of homotopy. At a first pass, one may wish to study paths in the space of convex contact manifolds. We call this a strict convex homotopy, formalizing the definition below. This is too strong since it does not allow for birth-death type singularities which appear even in smooth Morse theory. Nonetheless, until Section \ref{sec:homotopy}, when we discuss more general homotopies, we will only care about convex contact manifolds up to strict homotopies.

\begin{defn} \label{defn:strict_homotopy}
A \textbf{strict convex homotopy} on a cobordism $M$ consists of a $C^{\infty}$-family of triples $(\xi_t,X_t,\phi_t)$ for $t \in [0,1]$ such that for each fixed $t$ the triple is a convex contact cobordism. In particular, $\phi_t$ is a Morse function for every $\phi$.
\end{defn}

In our definition of strict convex contact homotopy, we allow that $\xi$ could also change through the homotopy. \textbf{This includes allowing the underlying contact structure to change near the boundary.} If $\xi$ remains fixed near the boundary, then one can use the Moser trick to find a family of diffeomorphisms $\psi_t \colon M \rightarrow M$, fixed near the boundary, so that $d\psi_t \xi_0 = \xi_t$. Pulling back by $\psi_t$ allows us to think of every strict homotopy with $\xi$ fixed near the boundary as a strict homotopy with fixed $\xi$ everywhere, up to diffeomorphism. However, by allowing $\xi_t$ to vary along the boundary, we are allowing for the contact structure to flow into and out of the boundary. We will address this point further in Section \ref{sec:attach}.

\begin{rmk} \label{rmk:general_conv_struct}
For the purpose of considering higher parametric families, one would prefer to define a more general class of convex structures, in which a generic choice matches our definition, and a generic path passes only through birth-death type singularities. Given such a definition, then our notion of homotopy is enough to study the space of convex contact structures at the level of $\pi_0$. This is precisely what is done in the setting of Weinstein structures throughout \cite{CE}, although possible less strict definitions are discussed by Eliashberg elsewhere \cite{Eliashberg}.
\end{rmk}

The main order of business for this section is to understand the critical points of the pair $(X,\phi)$. We will see that it is convenient to separate such critical points into two distinct cases.

\begin{defn} A critical point of index $k \leq n$ is called \textbf{subcritical}. A critical point of index $k \geq n+1$ is called \textbf{supercritical}.
\end{defn}

We conclude this subsection with two examples of convex contact manifolds to indicate that these structures do occur naturally in some basic examples. We will prove later, as Corollary \ref{cor:CCS_exist}, that all closed contact manifolds do have a convex contact structure.

%\begin{exam}
%Our first examples of convex contact structures come from studying $1$-dimensional manifolds, where the theory is explicitly computable. A contact structure is a rank $0$ subbundle of the tangent bundle, and is trivial. A coorientation corresponds to an orientation of the tangent bundle. Any vector field preserves the contact structure. Hence, a convex structure on a closed manifold is simply a choice of Morse function and any vector field which is gradient-like for the Morse function (which is a contractible choice). Notice also that critical points must alternate between index $0$ and $1$. Hence, up to strict homotopy, convex contact structures on the circle are given by a positive number $n$, so that there are $2n$ total critical points ($n$ of index $0$). If our $1$-manifold has boundary, then there will be some components which are intervals. In this case, we again must have (possibly zero) critical points of alternating index.
%\end{exam}

\begin{exam}
Consider the standard contact sphere $S^{2n-1} \subset \C^n$, where the contact structure $\xi$ is the bundle of complex tangencies, $\xi_p = T_pS \cap iT_pS$. Associating $\C^n \equiv \R^{2n}$ with coordinates $x_1,y_1,\ldots,x_n,y_n$, we can write $iT_pS$ as the kernel of $\alpha = \sum(-y_idx_i + x_idy_i)$, so that $\alpha$ is a contact form when restricted to $S$. The contact vector field associated to the Hamiltonian $H = y_1$ is given by
$$X_H = \frac{1}{2} \left[-\partial_{x_1} + x_1 \sum\left(x_i \partial_{x_i} + y_i \partial_{y_i}\right) -y_1\sum\left(y_i\partial_{x_i} - x_i\partial_{y_i}\right) \right].$$
Consider $\phi = -x_1$, a Morse function on $S$. Then, with respect to the round metric,
$$\left|d\phi|_S\right|^2 = 1-x_1^2$$
$$|X_H|^2 = \frac{1}{4}\left(1-x_1^2+3y_1^2\right).$$
$$d\phi(X_H) = 1 - x_1^2 + y_1^2$$
from which the gradient-like condition follows along $S$.
\end{exam}

\begin{exam}
Let $M$ be any closed manifold, and consider the 1-jet space $J^1 M = \R \times T^*M$ with contact form $\alpha = dz - \lambda$, where $\lambda = \sum p_i dq_i$ is the Liouville form on $T^*M$ (where $q_i$ are coordinates on $M$ and $p_i$ are the dual coordinates). Let $\phi \colon M \rightarrow \R$ be a Morse function. Fix any Riemannian metric $g$ on $M$. Then consider the Hamiltonian function
$$H = z - \lambda(\nabla \phi).$$
In this case, a computation shows that in normal coordinates (so that the $\partial/\partial q_i$ are orthonormal):
$$X_H = \nabla \phi + \sum_{i=1}^{n} \left(p_i + \sum_{j=1}^{n} \frac{\partial^2\phi}{\partial q_i \partial q_j}p_j \right) \frac{\partial}{\partial p_i} + z \partial_z.$$
Let the matrix $\scr{H}$ denote the Hessian $\frac{\partial^2\phi}{\partial q_i \partial q_j}$. Then
$$|X_H|^2 = |\nabla \phi|^2 + P^T(I+\scr{H})^2P + z^2,$$
where $P$ is the column vector with entries $p_1$ through $p_n$ and $I$ is the identity matrix. Meanwhile, set
$$\wt{\phi}(q,p,z) = \phi(q) + \frac{1}{2}|p|^2 + \frac{1}{2}z^2,$$
so that
$$|d\wt{\phi}|^2 = |\nabla \phi|^2 + |P|^2 + z^2.$$
Finally, notice that
$$d\wt{\phi}(X) = |\nabla \phi|^2 + P^T(I+\scr{H})P + z^2.$$
It follows that $X_H$ is gradient-like for $\wt{\phi}$ so long as $\scr{H}$ is small enough, which we can arrange for by scaling $\phi$. This gives a convex contact structure on $J^1 M$.
\end{exam}

\begin{rmk}
Any contact vector field $X$ is given by a Hamiltonian $H$. It is, however, difficult in general to determine by looking at the Hamiltonian whether the corresponding vector field comes from a convex structure (so that there is some $\phi$ for which $X$ is a pseudo-gradient). For example, the author is aware of no good method for producing a convex contact structure on the standard $3$-torus $\T^3$ with contact form $\cos\theta dx - \sin\theta dy$. In Section \ref{sec:OBD}, we will prove that such a convex contact structure must exist, and it will be implicit in the discussion that a convex structure with a minimal number of critical points is intimately connected with supporting open books with pages of maximal Euler characteristic.
\end{rmk}

\subsection{Standard neighborhoods for critical points}

We would like a natural local model for critical points of any index. We need only specify such a model for subcritical points, since sending $(X,\phi)$ to $(-X,-\phi)$ exhibits a duality between subcritical and supercritical points.

\begin{defn} \label{def:standard_cp}
Let $0 \leq k \leq n$ be an integer. Let $\R^{2n+1}_{\mathrm{std}}$ be the standard contact structure on $\R^{2n+1}$ given by the kernel of $\alpha = dz + \frac{1}{2}\sum_{i=1}^{n}(-y_idx_i + x_idy_i)$. The vector field
$$X_k = z \partial_z + \sum_{i=1}^{k}(-x_i\partial_{x_i} + 2y_i\partial_{y_i}) + \sum_{i=k+1}^{n}\left(\frac{1}{2}x_i \partial_{x_i} + \frac{1}{2}y_i\partial_{y_i} \right).$$
is a pseudo-gradient for the Morse function
$$\phi_k = z^2 + \sum_{i=1}^{k}(-x_i^2 + y_i^2) + \sum_{i=k+1}^{n}(x_i^2+y_i^2).$$
This yields a convex contact structure of index $k$,
$$\scr{C}_k := (\R^{2n+1}_{\mathrm{std}},X_k,\phi_k).$$
If instead $n+1 \leq k \leq 2n+1$, take
$$\scr{C}_k := (\R^{2n+1}_{\mathrm{std}},-X_{2n+1-k},-\phi_{2n+1-k}).$$
In either case, the model is called the \textbf{standard convex contact structure of index $k$}. A critical point $p$ of a convex contact cobordism $(M,\xi,X,\phi)$ is called \textbf{standard} if one can find a neighborhood $\Open_M(p)$ isomorphic to a neighborhood of $0$ in $\scr{C}_k$ up to shifting $\phi$ by a constant. If all critical points are standard, then the cobordism is also called standard.
\end{defn}

\begin{rmk}
Recall that the expansion coefficient $\mu$ is defined by the equation $\scr{L}_X\alpha = \mu \alpha$. In the standard model, one computes $\mu = 1$ in the subcritical case and $\mu = -1$ in the supercritical case.
\end{rmk}

\begin{rmk}
It is not the case that every critical point is standard in some coordinates: even in one dimension, taking $X = 2x\partial_x$ is not standard, since a simple computation shows that the value of $\mu$ at a critical point does not depend on the choice of contact form $\alpha$ for a given contact structure $\xi = \ker(\alpha)$, whereas $\mu = 2$ for this example.
\end{rmk}

We come now to our first main theorem, which asserts that every convex contact cobordisms is standard up to strict convex contact homotopy. In fact, we have the following stronger statement, in which there are strict conditions on the homotopy itself.

\begin{thm} \label{thm:standard_cp}
Suppose $(M,\xi,X,\phi)$ is a convex contact cobordism. Then one can find a strict convex contact homotopy $(X_t,\phi_t)$ on $(M,\xi)$, $0 \leq t \leq 1$, such that the following conditions are satisfied:
\begin{itemize}
	\item $(X_0,\phi_0) = (X,\phi)$
	\item the critical points and their critical values remain fixed for all $t$, and the homotopy is supported in an arbitrarily small neighborhood $U$ of these critical points
	\item on some smaller neighborhood $V \subset U$ of the critical points, the ascending and descending manifolds stay fixed
	\item $(M,\xi,X_1,\phi_1)$ is standard
\end{itemize}
\end{thm}

\begin{rmk}
In this theorem, the underlying contact manifold $(M,\xi)$ remains constant through the homotopy.
\end{rmk}

\subsection{Descending and ascending manifolds}

The first step towards proving Theorem \ref{thm:standard_cp} is to understand as much as we can about the neighborhood of the critical points of a (possibly non-standard) convex contact structure. This will allow us to gain enough control, via the various neighborhood theorems described in Section \ref{sec:background}, to construct our homotopy.

We fix notation so that $(M^{2n+1},\xi,X,\phi)$ is a convex contact cobordism. That the convex structure places strong restrictions on how the Morse theory of $X$ interacts with the contact geometry of $(M,\xi)$ is no surprise, given the corresponding statements for Weinstein manifolds. We prove that the ascending and descending manifolds in a convex contact manifold are isotropic and coisotropic, depending on the index of the critical point.

\begin{prop} \label{prop:nonzero_mu}
Let $p \in M$ be a critical point. Then $\mu(p) \neq 0$.
\end{prop}
\begin{proof}
We have that $\mu(p)$ is defined by the equation
$$(i_X d\alpha + di_X\alpha)_p = \mu(p)\alpha_p.$$
Note that $(i_Xd\alpha)_p = i_{X(p)}(d\alpha)_p = 0$ since $X(p) = 0$ at a critical point. So we reduce to the equation
$$(di_X\alpha)_p = \mu(p)\alpha_p.$$
On the other hand, we see that for any function $f$ in a neighborhood of $p$,
$$di_X(f\alpha)_p = \alpha_p(X(p))df_p + f(di_X\alpha)_p = f(di_X\alpha)_p.$$
In other words, $di_X$ acts pointwise at $p$, and is hence just a linear map $\Omega^1_p \rightarrow \Omega^1_p$. The equation defining $\mu(p)$ simply asserts that $\mu(p)$ is an eigenvalue for the operator $di_X$ on $\Omega^1_p$ with eigenvector $\alpha_p$ (note that $\alpha_p \neq 0$ since $\alpha$ is a contact form). If we use coordinates $x^1,\ldots,x^{m=2n+1}$, so that $\Omega^1_p$ has the canonical basis $dx^1,\ldots, dx^m$, then a short computation shows $di_X$ acts by the matrix $B_j^i = \partial_j X^i$. This matrix is invertible since $X$ is gradient-like for a Morse function. In particular, $\mu(p) \neq 0$.
\end{proof}

\begin{lem}
Suppose $p \in M$ is a critical point for which $\mu(p) > 0$. Then the descending manifold is isotropic and the ascending manifold is coisotropic.
\end{lem}
\begin{proof}
This proof is just a modification of the proof in the Weinstein setting, \cite[Proposition 11.9]{CE}. Let $V_p^-$ and $V_p^+$ be the descending and ascending manifolds for $p$. Let $\Psi^t$ be the time $t$ flow of $X$.

We first prove that $V_p^-$ is isotropic. Pick a point $q \in V_p^-$ and a vector $v \in T_qV_p^-$. It suffices to prove $\alpha(v) = 0$. Since all points in $V_p^-$ flow to $p$ under $X$, we have that $|d\Psi^t(v)|$ approaches $0$ as $t$ approaches $\infty$. Therefore, $((\Psi^t)^*\alpha)(v) = \alpha(d\Psi^t(v))$ also approaches $0$. Meanwhile, we have $\scr{L}_X\alpha = \mu \alpha$. It follows that $(\Psi^t)^*\alpha = \exp(\int_0^t ((\Psi^\tau)^*\mu) d\tau)\alpha$. Hence, if we look at the value at $q$, we see that for large enough $t$, $(\Psi^t)^*\mu \approx \mu(p)$. But $\mu(p) > 0$ by hypothesis, so this exponential coefficient goes off to $\infty$. Hence, in order for $((\Psi^t)^*\alpha)(v)$ to approach $0$, we must have $\alpha(v) = 0$.

Now we prove that $V_p^+$ is coisotropic. First we recall the $\lambda$-lemma (see e.g. \cite{PdM} for a proof) which states that there exists an open neighborhood $U$ around $p$ such that for any $q \in V_p^+ \cap U$, any submanifold $D$ intersecting $V_p^+$ at $q$ transversely of complementary dimension, and any $\epsilon > 0$, there is some $t_0$ (depending upon $\epsilon$) such that for all $t > t_0$, $\Psi^{-t}(D)$ is $\epsilon$ $C^1$-close to $V_p^- \cap U$.

Suppose we pick a point $q \in V_p^+$ and a vector $v \in (T_qV_p^+ \cap \xi_q)^{\omega}$ where $\omega$ is the conformal symplectic structure on $\xi_q$. We must prove $v \in T_qV_p^+$. Since $\Psi^t$ preserves $TV_p^{+}$, $\xi$, and the conformal class of $\omega$, and since $\Psi^t$ will flow $q$ into $U$ for $t$ negative enough, it suffices to prove the result assuming $q \in U \cap V_p^+$. The final bit of set-up we need is to suppose that there is some underlying Riemannian metric.

For each $n \in \N$, there is some constant $\lambda_n$ such that $\Psi^{-n}(v) = \lambda_nv_n$ where $v_n$ is a unit vector. Then by compactness, some subsequence of the $v_n$ converges to a unit vector $v_{\infty}$ in $T_pM$. Suppose by way of contradiction that $v \notin T_qV_p^+$. Then there is some transverse submanifold $D$ to $V_p^+$ at $q$ such that $v \in T_qD$. Applying the $\lambda$-lemma, we have that $v_{\infty} \in T_pV_p^- \leq (T_pV_p^-)^{\omega}$. But also  since $\Psi^t$ preserves $(TV_p^+ \cap \xi)^{\omega}$, we have also that $v_{\infty} \in (TV_p^+ \cap \xi)^{\omega}$. But then $v_{\infty}$ is $\omega$-orthogonal to both $T_pV_p^-$ and $T_pV_p^+ \cap \xi_p$, which span all of $\xi_p$, and so $v_{\infty} = 0$. But $v_{\infty}$ is a unit vector, so this cannot happen, and we have arrived at a contradiction in our assumption that $v \notin T_qV_p^+$. Therefore, $v \in T_qV_p^+$ as desired.
\end{proof}

\begin{lem}
If instead $\mu(p) < 0$, then the descending manifold is coisotropic and the ascending manifold is isotropic.
\end{lem}
\begin{proof}
Simply apply the previous lemma to the convex contact manifold $(M,\xi,-X,-\phi)$.
\end{proof}

\begin{cor} If $p$ is a critical point for which $\mu(p) > 0$, then $p$ is subcritical. If $p$ is a critical point for which $\mu(p) < 0$, then $p$ is supercritical.
\end{cor}
\begin{proof}
When $\mu(p) > 0$, the descending manifold is isotropic, and the index is just the dimension of the descending manifold, hence at most $n$. If $\mu(p) < 0$, the ascending manifold is instead isotropic, so the descending manifold, and hence the index, has dimension at least $n+1$.
\end{proof}

In what follows, we will be able to forget about the expansion coefficient, $\mu$. In summary, subcritical points have isotropic descending disks, whereas supercritical points have coisotropic descending disks.

\begin{rmk} In Weinstein surgery theory, the descending manifolds were all isotropic and of index at most half the dimension. Meanwhile, on the convex contact side, the inclusion of supercritical points allows for a richer theory. We will see in Section \ref{sec:OBD} that, in fact, if we only allow subcritical points on a contact $(2n+1)$-dimensional manifold, we are studying nothing more than the theory of $2n$-dimensional Weinstein cobordisms.
\end{rmk}

\subsection{Neighborhoods of critical points}

Finally we prove Theorem \ref{thm:standard_cp}. For simplicity, we shall always consider the subcritical case. This choice is purely cosmetic: we can consider the pair $(-X,-\phi)$ instead to obtain a supercritical critical point. We begin by understanding nice coordinates for the ascending and descending manifolds.

\begin{lem} \label{lem:using_isotopy_extension}
Let $(M^{2n+1},\xi,X,\phi)$ be a convex contact manifold with a subcritical point $p$ of index $k$. Then one can find an open neighborhood $U$ of $p$ and a smooth embedding $\Phi \colon (U,p) \hookrightarrow (\R^{2n+1},0)$ so that in the coordinates given by $\R^{2n+1}$, we have
\begin{itemize}
	\item $\xi = \ker \left( dz + \frac{1}{2}\sum_{i=1}^{n}(-y_idx_i + x_idy_i) \right)$
	\item The descending (isotropic) manifold $L$ is the $(x_1,\ldots,x_k)$-plane (i.e. the set with $x_{k+1} = \cdots = x_n = y_1 = \cdots = y_n = z = 0$).
	\item The ascending (coisotropic) manifold $C$ is the set where $x_1 = \cdots = x_k = 0$.
\end{itemize}
\end{lem}

\begin{proof}
The descending isotropic and ascending coisotropic submanifolds are transverse by Morse theory ($X$ is gradient-like for $\phi$), so this is just a restatement of Theorem \ref{thm:complementary_isotropic}.
\end{proof}

We have one trick up our sleeves, following the Weinstein case \cite[Lemma 12.9]{CE}. The idea is that we wish to interpolate contact vector fields. The most natural way to go about doing this is to interpolate between their contact Hamiltonians so that they remain fixed away from the critical point. This requires the use of a bump function, but there is one tricky analytical detail, which is that if we try to use just one bump function, we can only  interpolate for a short amount of time. Instead, what we do is use a sequence of bump functions to interpolate the vector field on smaller and smaller balls around the critical point $p$.

\begin{lem} \label{lem:main_homotopy}
Suppose that in an open neighborhood $U$ of $p$, $X_0$ and $X_1$ are two contact vector fields gradient-like for some fixed Morse function $\phi$ with unique nondegenerate critical point at $p$, such that $\mu_0(p) = \mu_1(p)$. Then there exists a homotopy of contact vector fields $X_t$ such that
\begin{itemize}
	\item $X_t$ remains gradient-like for $\phi$ for all $t$ (so this yields a strict homotopy of convex contact structures with no new critical points created)
	\item $X_t$ remains fixed on $U \setminus W$, where $W$ is a precompact neighborhood of $p$ in $U$ (meaning $\overline{W} \subset U$ is compact)
	\item In some neighborhood of $p$, $X_t = (1-t)X_0 + tX_1$. In particular, if $X_0$ and $X_1$ have the same ascending and escending manifolds, then they remain fixed in this neighborhood (but not necessarily on all of $U$).
\end{itemize}
\end{lem}

\begin{proof}
Let us fix a metric once and for all on $U$ so that it looks like the standard Euclidean metric in some coordinates around $p$. Since $X_0$ and $X_1$ are gradient-like for $\phi$, there is some $\delta > 0$ such that $d\phi(X_i) \geq \delta (|d\phi|^2+|X_i|^2)$.

We begin by proving the lemma under an additional assumption which will easily be removed at the end. In order to formulate this, note that the morphism $d\alpha \colon \xi \rightarrow \xi^*$ induced by interior product is an isomorphism since $d\alpha$ is symplectic on $\xi$. On the ball of radius $\epsilon$ around $p$, we can choose a constant $C_{\alpha}$ so that
$$\frac{1}{C_{\alpha}} \leq \|d\alpha\|_{\xi,\xi^*} \leq C_{\alpha},$$
where $\|d\alpha\|_{\xi,\xi^*}$ is the operator norm of $d\alpha \colon \xi \rightarrow \xi^*$ with respect to the metric and dual metric. Similarly, there is some constant $D_{\alpha}$ on this ball so that $|\alpha| \leq D_{\alpha}$. Our assumption is that the following two conditions are satisfied on the ball of radius $\epsilon$ around $p$:
$$|X_1-X_0| \leq \frac{\delta}{8C_{\alpha}^2}|d\phi|.$$
$$|\mu_1-\mu_0| \leq \frac{\delta}{8C_{\alpha}D_{\alpha}}|d\phi|.$$

For $\epsilon > 0$ small enough, we can pick a bump function $\rho$ constant and equal to $1$ near $p$, supported in an $\epsilon$-ball around $p$, and with $|d\rho| < 2/\epsilon$. We shall use the contact Hamiltonain
$$H_t = (1-t\rho)H_0 + t\rho H_1.$$
The corresponding contact vector field is then
$$X_t = (1-t\rho)X_0 + t\rho X_1 + t(H_1-H_0)Z$$
where $Z$ is a vector field satisfying
$$\alpha(Z)=0, \quad d\alpha(Z,\cdot) = d\rho(R_{\alpha})\alpha - d\rho.$$
Therefore,
$$d\phi(X_t) \geq \delta(|d\phi|^2+\min\{|X_0|^2,|X_1|^2\}) - |H_1-H_0||Z||d\phi|.$$
Notice that $d\alpha$ induces an isomorphism from $\xi$ to $\xi^*$. So we can write
$$Z = (d\alpha)^{-1}(d\rho(R_{\alpha})\alpha - d\rho),$$
where we think of $d\rho(R_{\alpha})\alpha - d\rho$ as a a section of $\xi^*$. But since $\alpha|_{\xi} = 0$, we have that this is just the same as $-d\rho$, and so
$$|Z| \leq C_{\alpha}|d\rho| \leq \frac{2C_{\alpha}}{\epsilon}.$$

Hence, it suffices to prove that $|H_1 - H_0| \leq \frac{\delta\epsilon}{4C_{\alpha}}|d\phi|$ on the support of $\rho$, since then
$$d\phi(X_t) \geq \delta(|d\phi|^2+\min\{|X_0|^2,|X_1|^2\}) - \frac{\delta}{2}|d\phi|^2 \geq \frac{\delta}{2}(|d\phi|^2+\min\{|X_0|^2,|X_1|^2\}),$$
in which case $X_t$ has produced no new critical points, and since $X_t = (1-t)X_0 + tX_1$ near $p$, it remains gradient-like for $\phi$. Here is where we need to use our assumed bounds on $|X_1-X_0|$ and $|\mu_1-\mu_0|$.

Let us work in Euclidean coordinates such that the metric is standard and such that $p$ is at $0$. Let $q$ be a point in the ball of radius $\epsilon$. Then since $H_0 = H_1 = 0$ at the critical point,
\begin{eqnarray*}
|(H_1-H_0)(q)| &=& \left|\int_0^1 \frac{d}{ds}(H_1-H_0)(sq)~ds \right| \\
	&\leq& \int_0^1 |q||d(H_1-H_0)(sq)|~ds
\end{eqnarray*}

Now $d(H_1-H_0) = (\mu_1-\mu_0)\alpha - i_{X_1-X_0}d\alpha.$ Hence
\begin{eqnarray*}
|d(H_1-H_0)(sq)| &\leq& |\mu_1-\mu_0||\alpha| + |X_1-X_0||d\alpha| \\
	&\leq& \frac{\delta}{8C_{\alpha}D_{\alpha}}D_{\alpha}|d\phi(sq)| + \frac{\delta}{8C_{\alpha}^2}C_\alpha|d\phi(sq)| \\
	&\leq& \frac{\delta}{4C_{\alpha}}|d\phi(sq)|
\end{eqnarray*}
Choosing $\epsilon$ small enough, by the nondegeneracy of $d\phi$, $|d\phi(sq)| \leq |d\phi(q)|$. Combining this all together,
$$|(H_1-H_0)(q)| \leq |q|\frac{\delta}{4C_{\alpha}}|d\phi(q)| \leq \frac{\delta\epsilon}{4C_{\alpha}}|d\phi(q)|.$$

To complete the proof of the lemma in general, consider the linear interpolation $X_t = (1-t)X_0 + tX_1$. Then one can find a sequence of times $0=t_0 < t_1 < \ldots < t_M = 1$ such that the assumptions are satisfied on each small interval, i.e.
$$|X_{t_{k+1}} - X_{t_k}| \leq \frac{\delta}{8C_{\alpha}^2}|d\phi|$$
$$|\mu_{t_{k+1}}-\mu_{t_k}| \leq \frac{\delta}{8C_{\alpha}D_{\alpha}}|d\phi|$$
(We can do this because $p$ is a non-degenerate critical point, so that $|d\phi|$ vanishes to order $\epsilon$, where $\epsilon$ is the distance to $p$.) We can then perform the modification just described sequentially on smaller and smaller balls, such that each successive ball occurs in the region where $\rho = 1$ from the previous modification, since in this region, after $k$ steps, the homotoped vector field is just $X_{t_{k+1}}$.
\end{proof}

\begin{lem} \label{lem:handles}
Given a convex structure $(X_0,\phi_0)$ on $(M,\xi)$, one can find a strict homotopy of convex contact structures $(X_t,\phi_t)$, supported in a neighborhood of the critical points, through which the critical points, their critical values, and their expansion coefficient $\mu$ all remain fixed, along with the stable and unstable manifolds in a neighborhood of each critical point, and such that in a neighborhood of any critical point $p$, one can find a coordinate neighborhood such that $(\xi,X_1,\phi_1)$ is standard up to possibly multiplying $X_1$ by a constant. That is, for the subcritical case:
$$\xi = \ker \left(dz + \frac{1}{2}\sum_{i=1}^{n} (-y_idx_i + x_idy_i) \right)$$
$$\phi_1 = \phi_0(p) + z^2 + \sum_{i=1}^{k}(-x_i^2 + y_i^2) + \sum_{i=k+1}^{n}(x_i^2+y_i^2)$$
$$X_1 = \mu\left[z\partial_z+\sum_{i=1}^{k} (-x_i \partial_{x_i} + 2y_i\partial_{y_i}) + \frac{1}{2}\sum_{i=k+1}^{n}(x_i\partial_{x_i}+y_i\partial_{y_i})\right].$$
\end{lem}

\begin{proof}
Lemma \ref{lem:main_homotopy} provides the desired homotopy on the level of the vector field, locally preserving the stable and unstable manifolds, satisfying the desired conditions. It remains to drag $\phi$ along this homotopy.

We begin by studying the gradient-like condition near a critical point. For a gradient pair $(X,\phi)$, in local coordinates, we have $X(Z) = AZ + O(|Z|^2)$ and $\phi(Z) = \phi(p) + \frac{1}{2}Z^TBZ + O(|Z|^2)$ where $A,B$ are matrices, with $B = B^T$ (this is the notation in \cite[Section 9.3]{CE}). The condition of being gradient-like near $p$ is just that $BA$ is positive-definite. But we have $A$ is positive definite on the tangent space to the stable manifold, and negative definite on the tangent space to the unstable manifold. So since $X_0$ and $X_1$ have the same stable and unstable manifolds near $p$, then $B_1A_0 > 0$ since $B_1$ is positive on the tangent space to the stable manifold and negative definite on the tangent space to the unstable manifold by the condition $B_1A_1 > 0$. Hence, $X_0$ is actually already gradient-like with respect to $\phi_1$ in a sufficiently small neighborhood around $p$.

We wish to find a homotopy $\phi_t$ from $\phi_0$ to $\phi_1$ of the desired form around $p$ for which $X_0$ is gradient-like with respect to $\phi_t$ for all time $t$, and such that the critical points and values of the $\phi_t$ are the same. We have that $Z^TB_0A_0Z, Z^TB_1A_0Z \geq \delta|Z|^2$ for some $\delta > 0$. Suppose that $|Z^T(B_1-B_0)Z| \leq \gamma|Z|^2$ for $\gamma > 0$ in the ball of radius $\epsilon$ around $p$, and write $\|A\|$ for the norm of $A$. Take $\rho$ a radial bump function supported in the ball of radius $\epsilon$ with
$$4\|A\|\gamma r\rho'(r) < \delta.$$
Then consider the homotopy given by $\phi_t = (1-t\rho)\phi_0 + t\rho \phi_1$. Near $p$, this is just linear, and $((1-t)B_0 + tB_1)A_0 > 0$, so this is gradient-like near $p$. In the region of interpolation, we need only check that $d\phi_t(X_0) > 0$. We compute that $d\phi_t = (1-t\rho)d\phi_0 + t\rho d\phi_1 + t(\phi_1-\phi_0)d\rho$. Hence, at coordinate $Z$,
\begin{eqnarray*}
d\phi_t(X_0) &=& \left((1-t\rho(Z))Z^TB_0A_0Z + O(|Z|^3)\right) + \left(t\rho(Z)Z^TB_1A_0Z + O(|Z|^3)\right)\\
	&& ~~~~~~~~+~t(Z^T(B_1-B_0)Z+O(|Z|^3) d\rho(X_0)
\end{eqnarray*}
Now, $X$ is $A_0Z + O(|Z|^2)$, so on a small enough ball, $|X| < 2\|A_0\||Z|$. Hence, from our hypothesis on $\rho$, we have $d\rho(X_0) < \delta/2\gamma$. Hence, we see
$$d\phi_t(X_0) \geq \delta|Z|^2 - \delta/2|Z|^2 + O(|Z|^3),$$
and so is positive in the region of interpolation if $\rho$ is chosen to have small enough support.

Hence, we can form the homotopy $(X_t,\phi_t)$ by first homotoping the Morse function in the way described from $\phi_0$ to $\phi_1$ so that $X_0$ remains gradient-like, followed by homotoping the vector field from $X_0$ to $X_1$ so it is gradient-like for $\phi_1$ as in Lemma \ref{lem:main_homotopy}.
\end{proof}

\begin{proof}[Proof of Theorem \ref{thm:standard_cp}]
The final step to go from Lemma \ref{lem:handles} to Theorem \ref{thm:standard_cp} is to scale back the vector field in a neighborhood of $p$. That is, the previous lemma allows us to prove that we can homotope to the triple $(\xi_{\mathrm{std}},\phi_{\mathrm{std}},\mu X_{\mathrm{std}})$, where $\mu > 0$ is some constant. We will simply homotope the vector field from $X_0 = \mu X_{\mathrm{std}}$ to $X_1 = X_{\mathrm{std}}$ such that $d\phi_{\mathrm{std}}(X_t) > 0$. Note that we may assume $\mu < 1$, since if $\mu = 1$ then we are trivially done, while if $\mu > 1$, then we can just repeat our construction for $\mu < 1$ where we scale by $1/\mu$ and run the homotopy backwards.

We use, as usual, a homotopy of the Hamiltonian of the form
$$H_t = (1-t\rho)\mu H_{\mathrm{std}} + t\rho H_{\mathrm{std}}$$
where $\rho$ is a radial bump function supported in a ball of radius $\epsilon$. Then $X_t = (1-t\rho)\mu X_{\mathrm{std}} + t\rho X_{\mathrm{std}} + t(1-\mu)H_{\mathrm{std}}Z$ where $Z$ is the vector field satisfying
$$\alpha(Z) = 0$$
$$d\alpha(Z,-) = d\rho(R_\alpha)\alpha - d\rho.$$
As in the proof of Lemma \ref{lem:main_homotopy}, we see that $|Z| \leq C_{\alpha}|d\rho|$, and now we can use the crude approximations $|d\phi_{\mathrm{std}}| \leq \lambda r$ and $|H_{\mathrm{std}}| \leq \nu r \leq \nu \epsilon$, with $\lambda,\nu > 0$ some constants. But also,
$$d\phi_{\mathrm{std}}(X_t) = (1-t\rho)\mu + t\rho + t(1-\mu)H_{std}d\phi_{\mathrm{std}}(Z).$$
Hence, in order for this to be positive for all time $t$, it suffices to have it so for just $t=0$, which is automatic, and for $t=1$, in which case we need
$$H_{\mathrm{std}}d\phi_{\mathrm{std}}(Z) > \frac{-\mu}{1-\mu}-\rho.$$
Hence, it suffices to have
$$|H_{\mathrm{std}}||d\phi_{\mathrm{std}}||Z| < \frac{\mu}{1-\mu}.$$
We can find a radial bump function $\rho$ such that $r|\rho'(r)| < \frac{\mu}{C_{\alpha}\lambda\nu\epsilon(1-\mu)}$, since $1/r$ has divergent integral around $0$. Then
$$|H_{\mathrm{std}}||d\phi_{\mathrm{std}}||Z| \leq (\nu\epsilon)(\lambda r)(C_{\alpha}|\rho'(r)|) < \frac{\mu}{1-\mu}.$$
\end{proof}

\section{Attaching Handles} \label{sec:attach}

In the previous section, we proved that we could isotope all critical points of a convex contact cobordism so that they were standard. We now wish to use this theory to explicitly describe convex contact cobordisms via handle attachments. In this section, we will define a convex contact handlebody, which is the result of attaching together standard contact handles, which are just subsets of the already described standard models for critical points.

\subsection{Morse function modifications}

Our first step in attaching is to prove that the particular choice of Morse function is essentially irrelevant. This will allow us to easily match up standard handles without needing to ever consider the Morse function on the attaching region. The key technical lemma, which makes no mention of the underlying contact geometry, is as follows.

\begin{lem} \label{lem:Morse_modify}
Consider a trivial cobordism $\Sigma \times [-\epsilon,0]$, and let $t$ be the coordinate in $[-\epsilon,0]$. Suppose that $\phi_0$ and $\phi_1$ are functions such that
$$\frac{\partial \phi_i}{\partial t} > 0,~~~~~\mathrm{and}~~~~~\phi_0|_{\Sigma \times 0} = \phi_1|_{\Sigma \times 0}$$
Then there is a homotopy $\Phi_s$, $0 \leq s \leq 1$, such that
\begin{itemize}
	\item $\Phi_0 = \phi_0$
	\item $\Phi_1 = \phi_0$ near $\Sigma \times \{-\epsilon\}$
	\item $\Phi_1 = \phi_1$ near $\Sigma \times \{0\}$
	\item $\frac{\partial \Phi_s}{\partial t} > 0$ for all $s$
\end{itemize}
Furthermore, suppose there is some $A \subset \Sigma \times \{0\}$ such that $\phi_0 = \phi_1$ on a neighborhood of $A$ in $\Sigma \times [-\epsilon,0]$. Then $\Phi_s$ remains fixed as $s$ varies on a possibly smaller neighborhood of $A$.
\end{lem}

\begin{proof}
First observe that it actually does not matter what $\phi_1$ is outside of a neighborhood of $\Sigma \times \{0\}$, since the only requirement is that $\Phi_1$ be equal to $\phi_1$ in this neighborhood. Hence, we can replace $\phi_1$ with any other function $\wt{\phi}_1$ equal to $\phi_1$ near $\Sigma \times 0$. In particular, we choose $\wt{\phi}_1$ such that:
\begin{itemize}
	\item $\frac{\partial \wt{\phi}_1}{\partial t} > 0$
	\item $\wt{\phi}_1 = \phi_1$ near $\Sigma \times 0$
	\item $\wt{\phi}_1 \geq \phi_0$ on $\Sigma \times [-\epsilon,-\epsilon/2]$
\end{itemize}
Now suppose $\rho \colon [-\epsilon,0] \rightarrow [0,1]$ is a nondecreasing bump function equal to $0$ near $-\epsilon$ and equal to $1$ along $[-\epsilon/2,0]$. Then consider the interpolation
$$\Phi_s = (1-s\rho(t))\phi_0 + s\rho(t)\wt{\phi}_1.$$
The only nontrivial property to prove is that $\frac{\partial \Phi_s}{\partial t} > 0$ for all $s$. We compute directly
$$\frac{\partial \Phi_s}{\partial t} = (1-s\rho(t))\frac{\partial \phi_0}{\partial t} + s\rho(t)\frac{\partial \wt{\phi}_1}{\partial t} + s(\wt{\phi}_1-\phi_0)\frac{\partial \rho}{\partial t}$$
The sum of the first two terms is always positive. The last term is only nonzero when $\frac{\partial \rho}{\partial t} > 0$, but this occurs in the region $[-\epsilon,-\epsilon/2]$, in which $\wt{\phi}_1 \geq \phi_0$, so that this last term is nonnegative.

This proof holds verbatim for the relative version.
\end{proof}

In addition, in our definition of convex contact cobordisms, we required that each of $\partial_{\pm}M$ was a level set for $\phi$. However, when we attach handles, we will naturally arrive at a cobordism where this not the case. We need to account for this discrepancy.

\begin{lem} \label{lem:make_level}
Suppose $(M,\xi,X,\phi)$ satisfies all of the conditions of a convex contact cobordism, except with the slightly weaker condition that $X$ is transverse to the boundary, inward at $\partial_- M$ and outward at $\partial_+ M$, as opposed to these boundaries being level sets. Then there is a strict homotopy supported in a neighborhood of $\partial_+ M$ and $\partial_- M$ such that each becomes a regular level set.
\end{lem}
\begin{rmk}
Recall from the discussion following Definition \ref{defn:strict_homotopy} that since we allow strict homotopies which modify the underlying contact structure near the boundary, the modification of this lemma does not imply that the resulting contact structure on $M$ is contactomorphic to the original. We will get at the heart of the matter in Corollaries \ref{cor:uncompleted_contactomorphism} and \ref{cor:completed_contactomorphism}.
\end{rmk}
\begin{proof}
It suffices to prove the lemma for $\partial_+ M$. By Lemma \ref{lem:Morse_modify}, we can modify $\phi$ to $\wt{\phi}$ near $\partial_+ M$ so that $\phi|_{\partial_+ M} = \wt{\phi}|_{\partial_- M}$ and $d\wt{\phi}(X) = 1$, and such that for the whole homotopy, $X$ remains gradient-like. Then on a neighborhood $U$ of $\partial_+ M$, the convex structure is of the form
\begin{itemize}
	\item $U = \partial_+ M \times [-\epsilon,0]$, with $t$ the coordinate in $[-\epsilon,0]$
	\item $\xi = \ker(udt+\beta)$ where $u$ and $\beta$ are a function and $1$-form on $\partial_+ M$
	\item $X = \partial_t$
	\item $\wt\phi = f+t$ where $f$ is a function on $\partial_+ M$
\end{itemize}
We may extend this structure to $\partial_+ M \times [-\epsilon,\infty)$ by naturally extending the $t$-invariant expressions for $\xi$, $X$, and $\wt\phi$. We pick the domain $D$ given by $\wt\phi \leq \max f$. Extending out to this domain $D$ by flowing along $X$ yields an explicit diffeotopy which one can use to pull back to the original $M$ to complete the proof.
\end{proof}

\begin{rmk} \label{rmk:def_retract}
Lemma \ref{lem:make_level} holds parametrically as well, and one obtains a weak homotopy equivalence of the form described in Remark \ref{rmk:equiv_notions_convex_contact_cobordism}.
\end{rmk}

For the remainder of this section, we will mostly ignore the Morse function, since the results of this subsection essentially tell us that the specific choice is completely unimportant.

\subsection{Trivial cylinders, holonomy, and convexomorphisms}

Although Theorem \ref{thm:standard_cp} provides all of the necessary information about what convex contact cobordisms look like up to strict homotopy in a neighborhood of their critical points, we do not yet know that there is no interesting information encoded away from the critical points. We spend this subsection discussing when a given convex contact cobordism has no critical points.

\begin{defn} A \textbf{cylindrical convex contact cobordism} is a convex contact cobordism with no critical points.
\end{defn}

\begin{defn}
Suppose we have a cylindrical convex contact cobordism $(M, \xi, X, \phi)$. Then the \textbf{holonomy map} is the diffeomorphism $\psi \colon \partial_-M \rightarrow \partial_+M$ determined by the existence of an $X$-integral curve from $p \times \{0\}$ to $\psi(p) \times \{1\}$.
\end{defn}

To understand what cylindrical convex contact cobordisms may look like, we first wish to understand what the possible holonomy maps are.

\begin{defn} \label{defn:conv_surf}
The \textbf{space of convex hypersurfaces of dimension $2n$}, denoted by $\scr{C}_{2n}$, consists as a set of germs of convex hypersurfaces with transverse contact vector field. Equivalently, each element is a closed $2n$-dimensional manifold $\Sigma$ together with a decomposition $\Sigma = (R_+,\lambda_+) \cup_{\Gamma,\xi} (R_-,\lambda_-)$, where each $(R_{\pm},\lambda_{\pm})$ is a Liouville manifold which and such that there is a diffeomorphism $\Open(\Gamma) \cong \Gamma \times [-1,1]$, with $u$ the coordinate on $[-1,1]$, on which $\lambda_{\pm} = \beta/u$ for $\beta$ a contact form for $(\Gamma,\xi)$.

We topologize this set via the $C^{\infty}$ topology on the pair $(\xi,X)$ consisting of the germ of the contact structure and transverse contact vector field for each fixed underlying $2n$-dimensional manifold. A path of such germs is called a \textbf{Hamiltonian isotopy of the convex structure}.
\end{defn}

\begin{defn}
For $\Sigma_0,\Sigma_1 \in \scr{C}_{2n}$, a \textbf{convexomorphism} $\phi \colon \Sigma_0 \rightarrow \Sigma_1$ is a diffeomorphism of the underlying manifold such that
\begin{itemize}
	\item The restriction of $\phi$ to $(R^0_{\pm},\lambda^0_{\pm})$ is an \textbf{exact symplectomorphism} onto $(R^1_{\pm},\lambda^1_{\pm})$. (This means $\phi^*(\lambda^1_{\pm}) - \lambda^0_{\pm}$ is an exact form.)
	\item The restriction of $\phi$ to $(\Gamma^0,\xi^0)$ is a contactomorphism onto $(\Gamma^1,\xi^1)$.
\end{itemize}
\end{defn}

\begin{lem}[Moser lemma for convex hypersurfaces] \label{lem:moser}
If $\Sigma_t$ is a Hamiltonian isotopy of convex structures on an underlying manifold $\Sigma$, then there is a family of convexomorphisms $\psi_t \colon \Sigma_0 \rightarrow \Sigma_t$.
\end{lem}
\begin{proof}
This follows from the Moser lemma for ideal Liouville domains \cite{Giroux_Ideal}.
\end{proof}

\begin{lem} \label{lem:holonomy}
Suppose $(M,\xi,X,\phi)$ is a cylindrical convex cobordism, with $\Sigma_0 = \partial_{-} M$ and $\Sigma_1 = \partial_+ M$ (as elements in $\scr{C}_{2n}$, hence with the data of the germ of $X$). Then the holonomy map $\psi \colon \Sigma_0 \rightarrow \Sigma_1$ is a convexomorphism.

Conversely, if $\Sigma_{0}, \Sigma_1 \in \scr{C}_{2n}$, and $\psi \colon \Sigma_0 \rightarrow \Sigma_1$ is a convexomorphism, then there exists a cylindrical cobordism $(M,\xi,X,\phi)$ with $\partial_{-} M  \cong \Sigma_{0}$ and $\partial_+ M \cong \Sigma_1$ and with holonomy given by $\psi$.
\end{lem}

\begin{proof}
First, we prove that the holonomy map is indeed a convexomorphism. We use $X$ to trivialize the cobordism. That is, there is a uniquely defined domain $D \subset \Sigma_0 \times [0,\infty)$ and diffeomorphism $\chi \colon M \rightarrow D$ which is the identity along $\Sigma_0 \times \{0\}$ and such that $d\chi(X) = \partial_t$. The holonomy map in these coordinates is just  the map $\psi \colon \partial_-D \rightarrow \partial_+D$ given by flowing up in the $\partial_t$ direction, where the lower boundary $\partial_-D$ is just $\Sigma_0 \times \{0\}$, and the upper boundary $\partial_+D$ is just the graph of some function $H \colon \Sigma \rightarrow (0,\infty)$. Since $X$ preserves the contact structure, we may assume the contact form on $D$ is $t$-invariant, just $udt + \beta$ as usual. Note that using the coordinate $\wt{t} = t-H$ so that $\Sigma_1 \cong \partial_+D = \Sigma_0 \times \{0\} \subset \Sigma_0 \times \R_{\wt{t}}$ (and $X = \partial_{\wt{t}}$), we have that the contact form near $\partial_+D$ is $ud\wt{t}+udH+\beta$. Hence, $R_{\pm}$ of $\partial_+D$ have Liouville forms $\beta/u + dH$, and so the holonomy map is an exact symplectmorphism. Meanwhile, on the dividing set $\{u=0\}$, the contact form is still just $\beta$. Hence, the holonomy map does satisfy the required conditions on $D$, and $\chi^{-1}$ preserves these conditions.

The converse is realized via precisely the same graph construction. Write $\Sigma_0 = R_+^0 \cup_{\Gamma^0} R_-^0$ and $\Sigma_1 = R_+^1 \cup_{\Gamma^1} R_-^1$. Then since $\psi$ is an exact symplectomorphism on each of $R_{\pm}$, we have that there are functions $H_{\pm}$ such that $\psi^*(\lambda_{\pm}^1 + dH_{\pm}) = \lambda_{\pm}^0$. That $\psi$ is a diffeomorphism implies $H_{\pm}$ glue together smoothly to form $H$ along all of $\Sigma$, up to possibly shifting $H_-$ by a constant. By a further constant shift, we may assume $H > 0$. Then the domain $D$ in $\Sigma_0 \times \R$ with lower boundary $\Sigma_0 \times \{0\}$, upper boundary given by the graph of $H$, $X = \partial_t$, and $\xi$ the $t$-invariant contact form matching the germ of $\Sigma_0$ yields the desired cylindrical cobordism.
\end{proof}

\begin{rmk} \label{rmk:parametric_convexomorphism}
This converse further holds parametrically, in that if we are given a family of convexomorphisms, then it can be realized as the holonomy map of a family of cylindrical convex contact cobordisms. This is because the entire proof can be repeated parametrically. In particular, given a Hamiltonian isotopy $\Sigma_t$ for $\Sigma_0$ to $\Sigma_1$, we can find a family of cylindrical convex structures on $\Sigma \times I$, $(\xi_t,X_t,\phi_t)$, such that $X_0 = \partial_t$ and the holonomy map at time $t$ is precisely $\psi_t$ as constructed in Lemma \ref{lem:moser}.
\end{rmk}

As we have already seen, the procedure of using the vector field $X$ to trivialize a cylindrical cobordism is very fruitful. We will use it in the following key lemma to prove that cylindrical cobordisms are essentially trivial.

\begin{lem} \label{lem:triv_cyl}
Suppose we have two cylindrical convex contact cobordisms $(M_0, \xi_0, X_0, \phi_0)$ and $(M_1, \xi_1, X_1, \phi_1)$ such that the convex structures are identified on a neighborhood of their lower boundaries $\partial_- M_i$, by which we mean there is some $\wt{\psi} \colon \Open(\partial_- M_0) \rightarrow \Open(\partial_- M_1)$ intertwining the convex structures. Then there is a diffeomorphism $\psi \colon M_0 \rightarrow M_1$ extending $\wt{\psi}$, such that $M_0$ and $\psi^*M_1$ are strictly convex homotopic via a strict homotopy fixed on $\Open(\partial_- M_0)$. Furthermore, if the convex structures also match on $\Open(\partial_+M_i)$, meaning there is some $\widehat{\psi} \colon \Open(\partial_+ M_0) \rightarrow \Open(\partial_+ M_1)$, such that the holonomy maps are the same (under the identifications of $\wt\psi$ and $\widehat{\psi}$), then $\psi$ can be made to also extend $\widehat{\psi}$ and the strict homotopy can also be fixed on $\Open(\partial_+ M)$ such that the holonomy map remains fixed through the homotopy.
\end{lem}
\begin{proof}
Let $\Sigma$ be the underlying manifold of the lower boundaries of $M_0$ and $M_1$, which are identified. As in the proof of the previous lemma, we may find smooth maps $\psi_i \colon M_i \rightarrow \Sigma \times [0,\infty)$ determined by $d\psi_i(X_i) = \partial_t$. Each map $\psi_i$ is then a diffeomorphism onto its image, $D_i$. By construction, $(\psi_0)_*(X_0,\phi_0) = (\psi_1)_*(X_1,\phi_1)$ on $\Open(\Sigma \times \{0\})$. The upper boundary of $D_i$ is then the graph of some $H_i$. We may consider domains $D_\tau$ whose upper boundary is then given by the graph of $(1-\tau)H_0 + \tau H_1$. Then we may fix a family of diffeomorphisms $\chi_{\tau} \colon D_0 \rightarrow D_{\tau}$ such that each $\chi_\tau$:
\begin{itemize}
	\item is the identity on $\Open(\Sigma \times \{0\})$
	\item only stretches in the $[0,\infty)$ direction (so $\chi_{\tau}$ is of the form $(x,t) \mapsto (x,\chi_{x,\tau}(t))$)
\end{itemize}
Then $\chi_{\tau}$ pulls back $(\xi = \ker(udt+\beta),X=\partial_t)$ to a strict homotopy of convex structures on $D_0$, up to the information of the Morse function. The map $\psi_0$ then pulls this back to a strict homotopy on $M_0$ from $(\xi_0,X_0)$ to $(\psi_1^{-1} \circ \chi_1 \circ \phi_0)^*(\xi_1,X_1)$. The diffeomorphism $\psi_1^{-1} \circ \chi_1 \circ \phi_0 \colon M_0 \rightarrow M_1$ is the desired diffeomorphism ambiguity in the lemma. One can easily find a homotopy of the Morse function, since the space of $C^{\infty}$ functions $\phi$ with $d\phi(X_\tau) > 0$ is contractible (in fact convex), even if we fix $\phi$ on $\Open(\Sigma \times \{0\})$.

Furthermore, suppose the convex contact structures also agree near the upper boundary and have the same holonomy, in the sense of the statement of the lemma. Since the convex data is intertwined on $\Open(\partial_+ M_i)$, we have that the corresponding Liouville structures (meaning the primitive $\lambda$, not just the symplectic form) on the decompositions of the $\partial_+ M_i$ into $R_+ \cup R_-$ are also identified. But we also showed in Lemma \ref{lem:holonomy} that the Liouville forms along the top were of the form $\lambda + dH_0$ and $\lambda + dH_1$, where $\lambda$ is the Liouville form along the bottom boundary. Hence $H_0$ and $H_1$ only differ by a locally constant function. We may therefore take $\chi_{\tau}$ to further be such that in a neighborhood $U = \Open_{D_0}(\partial_+D_0)$, we have $\chi_{\tau}(x,t) = (x,t+\tau(H_1-H_0))$. This expression is equivariant by translation in the $t$ direction, which is what we need to conclude the stronger result, since $(\xi = \ker(udt+\beta),X = \partial_t)$ is also $t$-equivariant on $\Sigma \times \R$.
\end{proof}

\begin{cor} \label{cor:uncompleted_contactomorphism} Any two cylindrical convex contact cobordisms with the same holonomy map (in the sense of the previous lemma, in which the germs of the convex data of the upper and lower boundaries are identified) are contactomorphic.
\end{cor}
\begin{proof}
The previous lemma yields, up to diffeomorphism, a strict homotopy from one convex contact structure to the other which is fixed near the boundary. The Moser trick applies through this homotopy since the boundary contact structure remains fixed.
\end{proof}

To any convex contact cobordism $\scr{M} = (M,\xi,X,\phi)$, we may form an infinite version by attaching $\partial_-M \times (-\infty,0]$ to the negative boundary and $\partial_+M \times [0,\infty)$ to the positive boundary, such that on each piece $X = \partial_t$ and $\xi$ is invariant under translation in the $t$-direction. This extension is unique up to a choice of extension of the Morse function $\phi$. Let us call this completed convex contact cobordism $\widehat{\scr{M}}$. Note that strict homotopies, when restricted to the boundary, yield Hamiltonian isotopies of convex structures, and hence can be realized as holonomy maps. Since we have seen that the holonomy map essentially comes from taking the graph of some function $H$ in these infinite ends, the following result follows from the proofs already presented.

\begin{cor} \label{cor:completed_contactomorphism} Suppose $\scr{M}_t := (M,\xi_t,X_t,\phi_t)$ is a strict convex contact homotopy. Then $\widehat{\scr{M}_0}$ and $\widehat{\scr{M}_1}$ are contactomorphic.
\end{cor}

Finally, we prove one technical lemma, which proves that all strict homotopies can be modified near the upper boundary so that they are fixed on the upper boundary. This result allows us essentially to isolate strict convex contact homotopies away from given level sets.

\begin{lem} \label{lem:fix_top}
Suppose $(M,\xi^0_t,X^0_t)$ is a strict convex contact homotopy on a cobordism $M$ (where the Morse function is irrelevant). Then one can find a homotopy of strict convex contact homotopies $(\xi^s_t,X^s_t)$ such that:
\begin{itemize}
	\item $(\xi^s_t,X^s_t) = (\xi^0_t,X^0_t)$ except on $U=\Open(\partial_+ M)$
	\item $(\xi^1_0,X^1_0) = (\xi^0_0, X^0_0)$ everywhere
	\item $(\xi^1_t,X^1_t)$ is constant on $V=\Open(\partial_+ M)$ (where we think of $\overline{V} \subset U$)
\end{itemize}
\end{lem}
\begin{proof}
Consider just a fixed neighborhood $C = [-\epsilon,0] \times \partial_+ M$ which yields a cylindrical convex cobordism $C_t$ for all $t$ when we consider also the data $(\xi_t,X_t)$. The strict homotopy yields two Hamiltonian isotopies of convex structures: $\Sigma^-_t$ on $\{-\epsilon\} \times \partial_+ M$ and $\Sigma^+_t$ on $\{0\} \times \partial_+ M$. The holonomy of $C_t$ is a convexomorphism $\Sigma^-_t \rightarrow \Sigma^+_t$. We now may homotope the Hamiltonian isotopy $\Sigma^+_t$ to the constant homotopy at $\Sigma^+_{0}$ via $\Sigma^+_{s,t} := \Sigma^+_{(1-s)t}$. By Remark \ref{rmk:parametric_convexomorphism}, we have a $2$-parametric family $C_{s,t}$ of cylindrical convex structures on $C$ with $C_{0,t} = C_t$ and with holonomy a convexomorphism $\Sigma_-^t \rightarrow \Sigma^+_{s,t}$. In particular, at $s=1$, $C_{1,t}$ has fixed upper boundary, so that $(\xi^1_t,X^1_t)$ remains constant on $\Open(\partial_+M)$.
\end{proof}

\begin{rmk} \label{rmk:par_holonomy}
This lemma also holds in higher parametric families.
\end{rmk}

\subsection{Standard handles and convex contact handlebodies}

We are now ready to understand handle attachment. Recall that we proved in Theorem \ref{thm:standard_cp} that we could homotope a given convex contact structure in a neighborhood of the critical points into a standard form. The handles which we attach will therefore just be subsets of this standard neighborhood. We now describe how to attach handles to a given convex contact cobordism.

To start, consider $\R^2$ with coordinates $(p,q)$, and let $\psi(p,q) = p^2-q^2$. Fix some small $\delta > 0$ and let $\chi \colon [-1-\delta,1+ \delta] \times \R$ be a smooth even convex function matching $\sqrt{x^2-1}$ near the endpoints $\pm (1+\delta)$. Let $S$ be the region in $\R^2$ bounded by $\phi^{-1}(1)$ and such that $|q| \leq \chi(p)$. See Figure \ref{fig:RegionS}. For any $\epsilon > 0$, let $\epsilon S$ be the dilation of $S$ by constant $\epsilon$.

\begin{figure}[h!]
\centering
  \includegraphics[width=\textwidth]{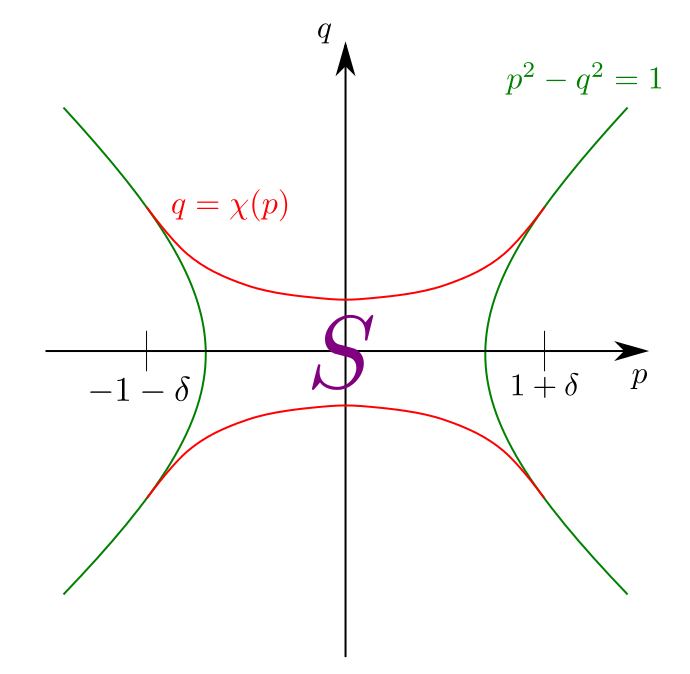}
  \caption{The region $S$ in $\R^2$.}\label{fig:RegionS}
\end{figure}

In $\R^{2n+1}$, our standard convex contact structure was written in Definition \ref{def:standard_cp} in such a way that the coisotropic and isotropic disks (the core and cocore, or vice versa, depending on whether we are in the sub- or super-critical case) were along complementary coordinate planes. Write $r_L$ for the radius in the isotropic directions and $r_C$ for the radius in the coisotropic directions. Explicitly, for example,
$$r_L^2 = \sum_{i=1}^{k}x_i^2.$$
Then we can consider the maps $f_{\sub},f_{\sup} \colon \R^{2n+1} \rightarrow \R^2$ given by $(r_L, r_C)$ and $(r_C,r_L)$ respectively. For any $\epsilon > 0$, denote by $H_{\sub}(\epsilon)$ and $H_{\sup}(\epsilon)$ the regions $f_{\sub}^{-1}(\epsilon S)$ and $f_{\sup}^{-1}(\epsilon S)$. The hyperboloid defined by $\psi(p,q) = p^2-q^2 = 1$ is pulled back to $\phi^{-1}(\phi_0(p) - \epsilon)$ in the subcritical case and $\phi^{-1}(\phi_0(p) + \epsilon)$ in the supercritical case.

\begin{defn} For an index $0 \leq k \leq 2n+1$ and any $\epsilon>0$, a \textbf{handle of index $k$} consists of either the set $H_{\sub}(\epsilon)$ when $k \leq n$ or $H_{\sup}(\epsilon)$ when $k \geq n+1$, together with the corresponding standard convex contact structure of index $k$.

The \textbf{attaching region} in either case consists of the lower boundary, which is a subset of a convex hypersurface which is a level set for $\phi$. The core intersects the attaching region along a sphere called the \textbf{attaching sphere}.
\end{defn}

Let $(R,S)$ be the attaching region and sphere of a handle $H$, and let $U$ be a small neighborhood of $R$ in $H$. Similarly, consider on some convex contact manifold $(M,\xi,X_M,\phi_M)$ sphere $S' \subset \partial_+ M$ with some $R' = \Open_{\partial_+ M}(S')$ and $U' = \Open_M(R')$. Suppose further that we have a diffeomorphism $\Psi \colon (R,S) \rightarrow (R',S')$ such that on $U \cup_\Psi U'$, the contact structures, contact vector fields, and Morse functions glue smoothly (such a gluing is unique if it exists). Then $H \cup M$ has a convex structure. We can apply Lemma \ref{lem:make_level} to form a new convex contact cobordism $\wt{H \cup M}$ by making the new upper boundary a level set.

\begin{defn}
The procedure described in the previous paragraph of going from $M$ to $H \cup M$ is called a \textbf{pre-handle attachment}, and $\wt{H \cup M}$ and \textbf{handle attachment}. A \textbf{convex contact handlebody} is a a convex contact cobordism obtained by obtained by a sequence of convex contact handle attachments starting from a cylindrical convex contact cobordism.
\end{defn}

Our goal will be to prove the following theorem.

\begin{thm}[= Theorem \ref{thm:main_thm}] \label{thm:main_thm_body}
Every convex contact cobordism $(M,\xi,X,\phi)$ is strictly convex homotopic to a convex contact handlebody.
\end{thm}

As stated, the theorem is not particularly helpful unless we can understand what data determines a handle attachment, since this will then yield a complete surgery-theoretic picture for convex contact cobordisms. A priori, a handle attachment may vary as the attaching region changes in a continuous way. In the course of proving our main theorem, we will show that this is not so. To proceed, we analyze the cases of subcritical and supercritical handle attachments separately. In each setting, we determine precisely what geometric information along a level set determines the handle attachment.

\begin{rmk}
In the following two subsections, we will sweep a technical detail partially under the rug. When we attach a handle, even in the standard smooth theory, we need to understand the attaching sphere as a \textit{parametrized} smoothly embedded sphere, i.e. as an explicit embedding $S^{k-1} \hookrightarrow \phi^{-1}(c)$ for some regular level set of $\phi$, where $k$ is the index of the handle.
\end{rmk}

\subsection{Attaching subcritical handles}

Consider a subcritical handle, $H_{\sub}(\epsilon)$, of index $k$. Recall that the descending disk of the critical point is isotropic and lies completely in the set in which $\alpha(X) = 0$. Hence, in this case, the attaching sphere is an isotropic in the dividing set $\Gamma^{2n-1}$ of the attaching region. Furthermore, a framing of the symplectic normal bundle of the descending disk at the critical point yields a framing over of the conformal symplectic normal bundle over the entire descending disk (up to contractible choice), which in turn restricts over its boundary to a framing of the conformal symplectic normal bundle of the sphere as an isotropic submanifold in the dividing set. By Proposition \ref{prop:attaching_data}, the attaching region is determined up to homotopy by this framing of the conformal symplectic normal bundle, which can be associated (non-canonically) with a map $S^{k-1} \rightarrow Sp(2(n-k))$. An isotropic sphere together with a framing of its conformal symplectic normal bundle will simply be referred to as a \textbf{framed isotropic sphere} for the rest of this section.

Although we originally defined attaching spheres in a model on the handle, we may generalize this definition as follows. Suppose we have a convex contact cobordism with a single critical point. Then the \textbf{attaching sphere} along the lower boundary is precisely the intersection of this level set with the descending disk. In the case of a subcritical point, this is automatically a framed isotropic sphere. We can see this as follows: if we standardize the critical point, then there is some level set just below it along which the attaching sphere remains fixed and becomes the attaching sphere of a standard handle. In fact, the whole descending manifold can be made to remain fixed via this standardization procedure, so also the attaching sphere along the bottom boundary stays fixed. Finally, the holonomy map between the bottom boundary and the level set just below the critical point preserves framed isotropic spheres.

\begin{lem} \label{lem:framed_isotropic_isotopy_to_holonomy} Given any isotopy of framed embedded isotropic spheres $(S_t,F_t)$ in the dividing set $\Gamma$ of a convex hypersurface $\Sigma$, there exists a convexomorphism $\psi \colon \Sigma \rightarrow \Sigma$ intertwining these framed spheres.
\end{lem}

\begin{proof}
There are two steps. \\

\noindent
\textbf{\emph{Step 1}} Every isotopy of framed embedded isotropic spheres in a contact manifold extends to an isotopy of contactomorphisms.\\

This is the usual isotopy extension theorem for isotropic submanifolds \cite[Theorem 2.6.2]{Geiges}, but in the special case where we include the framings. First of all, the usual isotopy extension theorem provides an isotopy of contactomorphisms $\psi_t \colon \Gamma \rightarrow \Gamma$ such that $\psi_t^* S_t = S_0$. Pulling back by $\psi_t$, it suffices to consider the case that $S_t$ is constant but the framing changes. But in this case, from the parametric neighborhood theorem, we can find a family of contactomorphisms $f_t \colon \Open(S_0) \rightarrow \Open(S_0)$, the identity along $S_0$, but such that $f_t^*F_t = F_0$. Then $f_t$ is generated by some contact vector field $X_t$, which in turn is a Hamiltonian vector field for some function $H_t$ with respect to a specified contact form. Using the contact vector field $\rho H_t$, for a bump function $\rho = 1$ near $S_0$ and $\rho = 0$ outside $\Open(S_0)$, we can integrate to form a contactomorphism $\chi_t \colon \Gamma \rightarrow \Gamma$ so that $\chi_t$ preserves $S_0$, is the identity away from $S_0$, and satisfies $\chi_t^*F_t = F_0$, as desired. \\

\noindent
\textbf{\emph{Step 2}} Any isotopy of contactomorphisms of the dividing set can be extended to a convexomorphism.\\ 

We will realize an isotopy of contactomorphisms by a Hamiltonian flow on each of $R_{\pm}$ supported in a neighborhood of $\Gamma$. In coordinates $udt+\beta$ near the dividing set, suppose $H$ is the Hamiltonian for $X_H$, the vector field generating the contactomorphism of the dividing set from Step 1, with respect to $\beta$. Then the vector field $X_H+u\mu \partial_{u}$ is Hamiltonian for $d(\beta/u)$, with Hamiltonian function $H/u$, on each of $R_{\pm}$. Using a cut-off function on the level of the Hamiltonian now yields a vector field generating an exact symplectomorphism of each of $R_{\pm}$ which restricts on $\Gamma$ to the desired contactomorphism, hence giving a convexomorphism.
\end{proof}

\begin{rmk} \label{rmk:stronger_subcrit_lemma}
In fact, more is true from the proof of this lemma. Suppose we have an isotopy of germs of neighborhoods of framed isotropic spheres. Then the convexomorphism can be made to intertwine these germs. We will technically need this stronger statement in what follows.
\end{rmk}

\begin{rmk} \label{rmk:parametric}
Further unravelling the proof of this lemma, we see that it is parametric. This is because we cook up a family of convexomorphisms $\Sigma \rightarrow \Sigma$ which realize the isotopy of the attaching sphere. By Remark \ref{rmk:par_holonomy}, it follows that there is a family of cylindrical convex contact cobordism structures $(X_t,\phi_t)$ on $\Sigma \times I$, fixed near the boundary, so that the holonomy maps $\psi_t$ send $(S_0,F_0)$ to $(S_t,F_t)$.
\end{rmk}

\begin{thm} \label{thm:subcrit_attach} Fix $0 \leq k \leq n$. Suppose $M^{2n+1}$ is a cobordism with two convex contact structures $(\xi_i,X_i,\phi_i)$, $i=0,1$, such that:
\begin{itemize}
	\item the convex structures match on $\Open(\partial_- M)$
	\item there is a unique critical point of index $k$
	\item the framed attaching spheres $(S_i,F_i)$ in the dividing set $\Gamma$ of $\partial_- M$ are isotopic through framed isotropic spheres in $\Gamma$
\end{itemize}
Then, up to an overall diffeomorphism of $M$ fixed on $\Open(\partial_-M)$, there is a strict homotopy $(\xi_t,X_t,\phi_t)$, fixed on $\Open(\partial_-M)$, from $(X_0,\phi_0)$ to $(X_1,\phi_1)$.
\end{thm}

\begin{proof}
We begin with some reductions. By modification of the Morse function, we may also assume that each $\phi_i$ satisfies $\phi_i(\partial_- M) = 0$, $\phi_i(\partial_+ M) = 1$, and $\phi_i(p_i) = 1/2$, where $p_i$ is the critical point. By applying our standardization procedure, Theorem \ref{thm:standard_cp}, we may furthermore assume that a neighborhood of $p_i$ matches the standard model, and hence, $\phi_i^{-1}[1/2-\epsilon,1/2+\epsilon]$ is given by a fixed handle attachment along some framed isotropic sphere $(S_i',F_i')$ in $\phi^{-1}(1/2-\epsilon)$. This modification may alter $(S_i,F_i)$, but only by an isotopy since $(S_i',F_i')$ remains fixed through this procedure and so $(S_i,F_i)$ are the preimage of these across a $1$-parameter family of holonomy maps. So this does not affect our assumption that the $(S_i,F_i)$ are isotopic. Hence, by including an extra cylinder below the critical point, Lemma \ref{lem:framed_isotropic_isotopy_to_holonomy} implies that we may actually assume that the $(S_i,F_i)$ are exactly the same.

As in the proof of Lemma \ref{lem:triv_cyl}, we use the holonomy map to build up our required overall diffeomorphism from which we can perform our strict homotopy. The difference is that we must be a little bit careful about what happens around the critical point, but since we are using identical models for the handle attachment, this is no real problem. We build up the diffeomorphism $\psi \colon M \rightarrow M$ in three blocks. First, $\psi$ restricts to the unique diffeomorphism $\phi_0^{-1}([0,1/2-\epsilon]) \rightarrow \phi_1^{-1}([0,1/2-\epsilon])$ which preserves $X$-trajectories and intertwines $\phi_0$ and $\phi_1$. This then also sends $(S_0',F_0')$ to $(S_1',F_1')$, and so (using the refinement of Remark \ref{rmk:stronger_subcrit_lemma}) the handles are attached along precisely the same neighborhood. Hence, on the handle itself, we have identical models, and we just map these via the identity. So far, we have a diffeomorphism on the prehandle attachments $\phi_0^{-1}([0,1/2-\epsilon]) \cup H \rightarrow \phi_1^{-1}([0,1/2-\epsilon]) \cup H$. On the final block (a cylindrical cobordism, but such that the lower boundary is not level), we define $\psi$ as we did on the first block, by following $X$-trajectories and intertwining $\phi_0$ and $\phi_1$.

Notice that on the handle, the convex contact structures match exactly. We now apply Lemma \ref{lem:triv_cyl} on the region $\phi_i^{-1}([0,1/2-\epsilon])$ to equilibrate the convex contact structures up to strict homotopy on these cylinders, and then also on the upper block. This completes the proof.
\end{proof}

\subsection{Attaching supercritical handles}

Consider a supercritical handle, $H_{\sup}(\epsilon)$. Just as in the subcritical case, one can read off the attaching data from the underlying contact structure.

\begin{lem} The attaching sphere of $H_{\sup}(\epsilon)$ is a framed balanced coisotropic sphere.
\end{lem}
\begin{proof}
Recall that in the model for the supercritical handle, we are taking $X = -z\partial_z + \sum_{i=1}^{k}(x_i\partial_{x_i}-2y_i\partial_{y_i}) - \frac{1}{2}\sum_{i=k+1}^n(x_i\partial{x_i}+y_i\partial_{y_i})$, and attaching along the level set $\phi = -z^2+\sum_{i=1}^{k}(x_i^2-y_i^2) + \sum_{i=k+1}^{n}(x_i^2+y_i^2) = -\epsilon$. The dividing set is given by the intersection of this level set with points in which $z+\frac{3}{2}\sum_{i=1}^{k}x_iy_i = 0$. The attaching disk is just the plane given by $x_1=\ldots=x_k = 0$ and it has characteristic foliation given by the span of the $y_1,\ldots,y_k$-directions. The boundary of the attaching disk, which is the attaching sphere, intersects the dividing set where $z=0$.

One now uses Proposition \ref{prop:project_foliation} to read off the foliation from the attaching disk, since the characteristic foliation is just the $X$-projection to the boundary, both along the dividing set $\Gamma$ (with $z=0$) or on $R_{\pm}$ (the regions with $z > 0$ and $z < 0$). We see that, automatically, we have a balanced coisotropic sphere. The framing is induced by a standard choice of framing of the cosymplectic conormal bundle to the leaf of the foliation on the attaching disk with boundary given by the binding of the equator of the attaching sphere.
\end{proof}

We are also interested in the attaching sphere of a general convex contact cobordism with a single supercritical point, not just the attaching sphere for the handle. When we standardize the critical point, we have that the level set just below the critical point undergoes a Hamiltonian isotopy, and so changes by the image of a convexomorphism by Lemma \ref{lem:moser}. But the attaching sphere along this level set is just a framed balanced coisotropic sphere, since it is the attaching sphere for a handle. Therefore, the original attaching sphere is the image under a convexomorphism of a framed balanced coisotropic sphere, hence a framed balanced coisotropic sphere itself, since convexomorphisms preserve this type of submanifold.

\begin{lem} \label{lem:balanced_isotopy_to_holonomy} Any isotopy of embedded framed balanced coisotropic spheres $C \subset \Sigma$ can be realized by a convexomorphism $\Sigma \rightarrow \Sigma$.
\end{lem}
\begin{proof}
Suppose we are given an isotopy of framed balanced coisotropic spheres $C_t$. Then one can find neighborhoods of these spheres which are isomorphic in the sense of Theorem \ref{thm:bcs_nbhd}, such that these neighborhoods are smoothly varying. Hence, one can interpolate these neighborhoods by the flow of a vector field $Y_t$ supported near $C_t$ such that $Y_t$ is tangent to and contact along $\Gamma$, and symplectic on $R_{\pm}$. Since $C_t \cap R_{\pm}$ are disks, we have that $Y_t$ is actually Hamiltonian in these regions. Hence, $Y_t$ is completely generated by some Hamiltonian, which can be cut-off away from $C_t$ to produce a global vector field $\wt{Y_t}$ whose flow is a contactomorphism of $\Gamma$ and generates exact symplectomorphisms of $R_{\pm}$. By definition, this flow is a convexomorphism.
\end{proof}

\begin{rmk} \label{rmk:stronger_supercrit_lemma} Just as in Remarks \ref{rmk:stronger_subcrit_lemma} and \ref{rmk:parametric}, this proof works on the level of germs of neighborhoods of framed balanced coisotropic spheres and is parametric. We technically need the refinement via germs in the proof of the following theorem.
\end{rmk}

\begin{thm} \label{thm:supercrit_attach}
Fix $n+1 \leq k \leq 2n+1$. Suppose $M^{2n+1}$ is a cobordism with two convex contact structures $(\xi_i,X_i,\phi_i)$, $i=0,1$, such that:
\begin{itemize}
	\item the convex structures match on $\Open(\partial_- M)$
	\item there is a unique critical point of index $k$
	\item the framed attaching spheres $(S_i,F_i)$ in the dividing set $\Gamma$ of $\partial_- M$ are isotopic through framed balanced coisotropic spheres in the dividing set
\end{itemize}
Then, up to an overall diffeomorphism of $M$ fixed on $\Open(\partial_- M)$, there is a strict homotopy $(X_t,\phi_t)$, fixed on $\Open(\partial_-M)$, from $(X_0,\phi_0)$ to $(X_1,\phi_1)$.
\end{thm}
\begin{proof}
The proof is the same as Theorem \ref{thm:subcrit_attach} but using Lemma \ref{lem:balanced_isotopy_to_holonomy} instead of Lemma \ref{lem:framed_isotropic_isotopy_to_holonomy} (and Remark \ref{rmk:stronger_supercrit_lemma} instead of Remark \ref{rmk:stronger_subcrit_lemma}).
\end{proof}

\subsection{Proof of Theorem \ref{thm:main_thm} = \ref{thm:main_thm_body}}
\begin{proof}
Suppose we start with a convex contact cobordism $(M,\xi,X,\phi)$. By arbitrarily small perturbation of the Morse function $\phi$, we may assume all of the critical points have distinct critical values. (The proof is no different from the smooth case; see \cite[Theorem 4.1]{Milnor}.) Hence, we can break the cobordism apart along level sets so that it is a composition of cobordisms with a single critical point. We may also assume, by Lemma \ref{lem:Morse_modify}, that $d\phi(X) = 1$ near each of these splitting levels. By Theorems \ref{thm:subcrit_attach} and \ref{thm:supercrit_attach}, each of these cobordisms with a single critical point may be realized by a handle attachment followed by a cylinder. Lemma \ref{lem:fix_top} furthermore ensures that the strict homotopy given can be made to fix the splitting level sets we chose to start. In this way, any elementary cobordism with a single critical point may be viewed as a handle followed by a cylinder. In particular, a cylinder followed by a handle attachment can also be viewed as a handle attachment followed by a cylinder. So all of the cylindrical cobordisms can be pushed to the top, until we are left with a sequence of handle attachments followed by a final cylindrical cobordism. This last cobordism can then be collapsed up to strict homotopy so that we are left with just a sequence of handle attachments.
\end{proof}

\subsection{Rearrangement and split convex structures}

In standard Morse theory, one quickly learns that it is relatively easy to take a Morse function and modify it so that it is self-indexing, without adding any new critical points. The idea is that if, for a Morse function $\phi$ with two critical points $p$ and $q$ of index $k_p$ and $k_q$ with $k_p < k_q$ and $\phi(p) > \phi(q)$, we can find a homotopy of $\phi$ with no birth-death type singularities such that $\phi(p) < \phi(q)$. Using this technique, we can always specify $\phi(p) = k_p$ and $\phi(q) = k_q$, if we so desire.

The same goes for convex contact cobordisms, and the proof is essentially the same.

\begin{lem} \label{lem:self-index}
Suppose $(M,\xi,X_0,\phi_0)$ is a convex contact cobordism. Then one can find a strict homotopy of convex structures $(X_t,\phi_t)$ (with $\xi$ fixed) from $(X_0,\phi_0)$ to $(X_1,\phi_1)$ such that $\phi_1 \colon M \rightarrow \R$ is self-indexing, i.e. for any critical point $p$ of index $k_p$, $\phi(p) = k_p$.
\end{lem}

\begin{proof}
We recall the proof of the smooth version to start; see Milnor's great exposition for more details \cite[Chapter 4]{Milnor}. First, suppose the following condition is satisfied:
\begin{itemize}
	\item There are no critical points $p$ and $q$ with $k_p < k_q$, $\phi(p) > \phi(q)$, and a gradient trajectory from $p$ to $q$. 
\end{itemize}
Then one can find a homotopy of the Morse function alone so that it is self-indexing. Hence, it suffices to find a perturbation via a strict homotopy so that the convex contact cobordism satisfies this condition. Furthermore, the condition appears to be generic in the sense that if the descending disk of $p$ and the ascending disk of $q$ are transverse, then they have no intersection by a dimension count.

This is precisely how the proof goes. By inserting cylindrical cobordisms between critical points not satisfying the condition, we may offset the attaching data of $p$ from the coattaching data of $q$ by a holonomy map - in fact by a Hamiltonian isotopy of convex structures. Since $k_p < k_q$, one of these sets of data will be an isotropic sphere in $\Gamma$, and since isotopies of isotropic spheres lift to Hamiltonian isotopies of convexomorphisms, which in turn are the holonomy maps of a family of cylindrical cobordisms, the result follows.
\end{proof}

We shall be interested in a slightly weaker notion than self-indexing. The definition below can be thought of as precisely the condition needed to provide a natural dictionary between Weinstein open book decompositions and convex structures. This will be expanded upon in Section \ref{sec:OBD}.

\begin{defn}
A convex contact cobordism $(M,\xi,X,\phi)$ is called \textbf{split} if there exists a regular level set $\phi^{-1}(c)$ such that all critical points in the region $\phi^{-1}(-\infty,c]$ are subcritical and all critical points in the region $\phi^{-1}[c,\infty)$ are supercritical.
\end{defn}

\begin{cor} \label{cor:split}
Every convex contact cobordism can be strictly homotoped to be split.
\end{cor}
\begin{proof}
Split is a weaker notion than self-indexing, so this is immediate from Lemma \ref{lem:self-index}.
\end{proof}

\section{Sutured convex contact structures} \label{sec:suture}

The attaching data of a subcritical handle consists solely of a framing for the conformal symplectic normal bundle of an isotropic sphere in the dividing set. But framed isotropic spheres in contact manifolds also appear as the attaching data for Weinstein handles. This equivalence is no coincidence: in this section we prove that subcritical handle attachment corresponds naturally to Weinstein handle attachment. Using this sutured description, we will prove in Section \ref{sec:OBD} that there is a correspondence between split convex contact structures and Weinstein open books, stated precisely in Theorem \ref{thm:split_is_OBD}. As a corollary, we deduce that every closed contact manifold has a (split) convex structure, and hence a handle decomposition.

\subsection{Sutured convex contact manifolds}
Colin, Ghiggini, Honda, and Hutchings proved that contact manifolds with convex boundary can be sutured, and conversely sutured contact manifolds can be un-sutured, such that suturing and un-suturing are inverse operations up to contact isotopy \cite{CGHH}. In this subsection, we shall follow the proofs of these facts, but add the extra structure of convexity. We will use the set-up in \cite{CGHH} for convenience.

\begin{defn}
A \textbf{sutured manifold} $(M,\Gamma,U(\Gamma))$ of dimension $n$ is an oriented manifold $M^n$ with boundary and corners such that
\begin{itemize}
	\item $\Gamma \subset \partial M$ is an oriented submanifold of dimension $n-2$
	\item $U(\Gamma)$ is a neighborhood of $\Gamma$, diffeomorphic as an oriented manifold to $[-1,1] \times (-\delta,0] \times \Gamma$ for some $\delta > 0$ such that in this product, $\Gamma$ is the submanifold $\{0\} \times \{0\} \times \Gamma$,
	\item all corners occur within $U(\Gamma)$ as $\{-1,1\} \times \{0\} \times \Gamma$, and
	\item $\partial M \setminus ((-1,1) \times \{0\} \times \Gamma)$ is a disjoint union of submanifolds $R_{\pm}$ with boundary $\{\pm 1\} \times \{0\} \times \Gamma$.
\end{itemize}

We write the \textbf{horizontal boundary} as $\partial_h M = R_+ \sqcup R_-$ and the \textbf{vertical boundary} $\partial_v M \cong [-1,1] \times \{0\} \times \Gamma \subset U(\Gamma)$.
\end{defn}

We use the coordinates $(s,\tau) \in [-1,1] \times (-\delta,0]$.

\begin{defn}
A \textbf{sutured contact manifold} is a contact structure $\xi$ on a sutured manifold $(M,\Gamma,U(\Gamma))$ such that one can write $\xi = \ker \alpha$ such that
\begin{itemize}
	\item $\alpha$ induces the structure of a Liouville domain on each of $R_{\pm}$
	\item In $U(\Gamma)$, $\alpha = Cds + e^{\tau}\beta_0$, where $\beta_0$ is a contact form on $\Gamma$.
\end{itemize} 
\end{defn}

\begin{defn}
A \textbf{sutured convex contact cobordism} is a sutured contact manifold $(M,\Gamma,U(\Gamma),\xi)$ together with a pair $(X,\phi)$ consisting of a contact vector field $X$ and Morse function $\phi$ so that the following two conditions hold:
\begin{itemize}
	\item $X$ is gradient-like for $\phi$ and everywhere transverse to the boundary. At the corners (in $U(\Gamma)$), by transversality, we require that $X$ is not tangent to either $\partial_v M$ or $\partial_h M$.
	\item $\{X \in \xi\} \cap \partial M = \Gamma$, so that $\Gamma$ is the dividing set for $X$.
\end{itemize}
\end{defn}

\begin{rmk}
For a sutured convex contact cobordism, we further break the boundary into two pieces based on whether $X$ is outward or inward pointing. We will always think of attaching handles along where $X$ is outward pointing.
\end{rmk}

\begin{defn}
A \textbf{strict convex contact homotopy} on a sutured manifold is a homotopy of sutured convex contact structures.
\end{defn}

Our goal is to understand how to go back and forth between the sutured and unsutured worlds of convex contact cobordisms. One direction is easy: if we begin with a sutured contact manifold, then we can unsuture by simply rounding off the corners. This procedure of rounding off takes place completely inside of $U(\Gamma)$ (in fact in a small neighborhood of the corners), and the same $\Gamma$ will remain the dividing set. We see further that the unsuturing of a sutured convex contact cobordism is well-defined up to strict convex contact homotopy, as any two different ways of rounding corners yields the same convex contact manifold up to some trivial cylindrical cobordisms.

Conversely, one wishes to suture a given convex contact cobordism $(M,\xi,X,\phi)$. In order to do this, it suffices to suture the convex contact structure near the dividing set $\Gamma_+ \subset \partial_+ M$ (one can similarly suture on the negative boundary by flipping orientations). By \cite[Lemma 4.1]{CGHH}, one can produce a sutured contact manifold out of an unsutured one, such that when we unsuture as described in the previous paragraph, we obtain the same contact manifold with convex boundary (up to homotopy). We do the same but in the convex setting by keeping track of the convex contact structure. (The proof presented here is also a little more explicit.)

\begin{lem} \label{lem:suturing}
Given a convex contact cobordism $(M,\xi,X,\phi)$, one can find a closed subset $M' \subseteq M$ such that:
\begin{itemize}
	\item $M'$ only differs from $M$ in a neighborhood of $\Gamma$
	\item The restriction of $(\xi,X,\phi)$ to $M'$ gives it the structure of a sutured convex contact cobordism, i.e. the subset $\Gamma \subseteq \partial M'$ along which $X \in \xi$ has a neighborhood $U(\Gamma)$ in $M'$ matching the required description
	\item Unsuturing $M'$ yields back $M$, up to strict convex contact homotopy
\end{itemize}
\end{lem}

\begin{proof}
All modifications will take place in a neighborhood of $\Gamma$. We work just with $\Gamma \subset \partial_+ M$; the case of $\Gamma \subset \partial_- M$ is the same. Using Proposition \ref{prop:div_set_nbhd}, any neighborhood of $\Gamma$ can be written so that it has contact form $udt + \beta$ where $\beta$ is a contact form along $\Gamma$, $X = \partial_t$, and the convex hypersurface is the set with $t=0$. Suppose $R_{\beta}$ is the Reeb vector field for $\beta$, and let $\psi^{r}_{R_{\beta}}$ be the time-$r$ Reeb flow along $\Gamma$. Consider the diffeomorphism $\phi \colon \Open(\Gamma) \rightarrow \Open(\Gamma)$ given by $(u,t,x) \mapsto \left(u,t,\psi^{\frac{1}{2}ut}(x)\right)$, which has $\phi^*\left(\frac{1}{2}(udt - tdu) + \beta\right) = udt + \beta$. We work in the new coordinates afforded by $\phi$, in which $\alpha = \frac{1}{2}(udt - tdu) + \beta)$ instead. In these coordinates, $X =\phi_* \partial_t =  \partial_t + \frac{1}{2}uR_{\beta_0}$.

Note that $\frac{1}{2}(udt - tdu) = \frac{1}{2}r^2d\theta$, where $r,\theta$ are polar coordinates for the $(u,t)$-plane, $\Open(\Gamma)$ corresponds to the region $\pi \leq \theta \leq 2\pi$. Our $M'$ will match $M$, and be specified by carving out some region in the $(u,t)$-plane. Namely, for some small $\epsilon > 0$, the upper boundary of $M'$ will consist of the graph of the function $t = f(u)$ with
\begin{itemize}
	\item $f(u)$ is even and piecewise smooth, with two corners at $u = \pm \epsilon/\sqrt{2}$ (and elsewhere smooth).
	\item For $|u| \leq \epsilon/\sqrt{2}$, we just take the negative arc of the circle $r = -\epsilon$, i.e. $t = -\sqrt{\epsilon^2-u^2}$.
	\item For $\epsilon/\sqrt{2} \leq |u| \leq \epsilon$, we take $f(u) = -|u|$ (i.e. along a constant $\theta$ value).
	\item For $|u| \geq 2\epsilon$, we take $f(u) = 0$.
	\item For $\epsilon \leq |u| \leq 2\epsilon$, we smoothly interpolate between $-|u|$ and $0$, so that $f(u)$ is always negative.
\end{itemize}
The graph of this region is shown in Figure \ref{fig:Sutured_Region}. Note that the dividing set of this manifold with corners is just the same $\Gamma$ shifted down to $u = 0$ and $t = -\epsilon$. Furthermore, if we look at the region of $M'$ with $\epsilon \leq r < 1.1\epsilon$, in the coordinates
$$(s,\tau) = \left(\frac{4}{\pi}\left(\theta-\frac{3\pi}{2}\right), -2\ln\left(\frac{r}{\epsilon}\right)\right)$$
this region is $[-1,1] \times (-2\ln(1.1),0] \times \Gamma$ with contact form $\alpha = \frac{\pi\epsilon^2}{8}ds + e^{\tau}\beta$, matching the sutured description. The pair $(X,\phi)$ is automatically convex for this sutured contact manifold by construction.

\begin{figure}[h!]
\centering
  \includegraphics[scale=0.75]{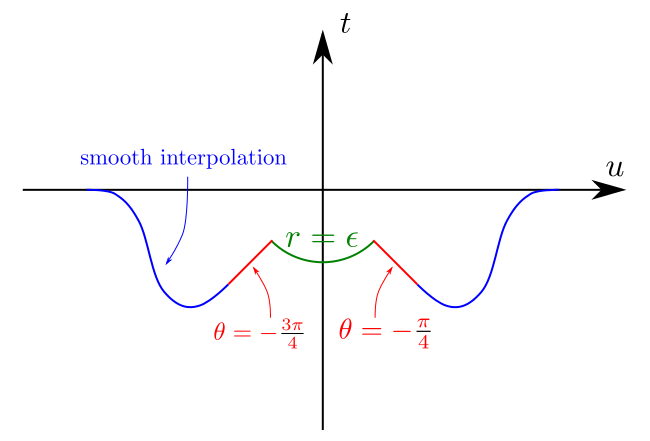}
  \caption{Carving $M'$ out of $M$.}\label{fig:Sutured_Region}
\end{figure}

For the last statement, we see that there is a trivial cobordism from the unsuturing of $M'$ to $M$, so they are the same up to strict convex contact homotopy.
\end{proof}

The previous lemma provides us with a possible way to suture a given convex contact cobordism. However, a priori, there may be many possible ways to suture. We wish to show that up to convex contact homotopy, there is a well-defined notion of suturing. The key is the following lemma.

\begin{lem}
If the unsuturings $u(M)$ and $u(N)$ of two sutured convex contact cobordisms $M$ and $N$ are strictly homotopic, then $M$ and $N$ are themselves strictly homotopic.
\end{lem}

\begin{proof}
Consider $\Gamma_M \subset \partial M$ and $\Gamma_N \subset \partial N$. Choose neighborhoods $V(\Gamma_M)$ and $V(\Gamma_N)$ which are not affected by the un-suturing (e.g. these neighborhoods do not include the corners) and which are isomorphic to the standard model $\alpha = uds+\beta$ and $X = \partial_s$. Then we can apply the suturing procedure of the previous lemma within these neighborhoods to obtain $s(u(M))$ and $s(u(N))$. Now, $u(M)$ and $u(N)$ are strictly homotopic, and any strict homotopy can be made to be relative to $V(\Gamma_M)$ and $V(\Gamma_N)$, on which the convex contact structures already match. Hence, $s(u(M))$ and $s(u(N))$ are also strictly homotopic, since the above homotopy supported away from $V(\Gamma_M)$ and $V(\Gamma_N)$ is not affected in the suturing procedure described. We claim that $s(u(M))$ is strictly homotopic to $M$, and similarly for $N$, which will complete the claim.

Let us focus on $s(u(M))$. We see that the unsuturing operation $M \mapsto u(M)$ occurs completely inside of $U(\Gamma_M)$. Meanwhile, $V(\Gamma_M)$ can be chosen small enough to stay within $U(\Gamma_M)$, so in fact, the suturing operation also occurs inside this region. We see then that $s(u(M))$ is just a subset of $M$, modified only inside of $U(\Gamma_M)$. One can then choose an isotopy of embeddings of $s(u(M))$ into $M$, modified only in $U(\Gamma_M)$ and interpolating between $s(u(M))$ and $M$, so that $X$ remains transverse to the boundary. This isotopy then yields a strict homotopy between convex structures on $s(u(M))$ and $M$. We leave the details to the reader.
\end{proof}

\begin{cor}
There is a well-defined operation, up to strict homotopy, called \textbf{suturing}, which takes as input a convex contact cobordism and outputs a sutured convex contact cobordism, and which is inverse to the unsuturing operation.
\end{cor}

%\subsection{Standardizing the convex structure near the suture}

\begin{defn} On $U(\Gamma) \cong [-1,1] \times (-1,0] \times \Gamma$ with sutured contact structure $\alpha = Cds + e^\tau \beta_0$, a \textbf{standard convex structure} $(X,\phi)$ is a convex structure such that $X = s\partial_s + \partial_\tau$ and $\phi = s^2+\tau$. (The precise choice of $\phi$ does not matter, except that it is Lyapunov for $X$.)
\end{defn}

\begin{lem} \label{lem:standard_suture}
Any sutured convex contact structure can be made standard up to strict convex contact homotopy.
\end{lem}

This lemma follows easily from how we were able to standardize the pair $(X,\phi)$ near any level set of an unsutured convex contact cobordism in Section \ref{sec:attach}. We include an explicit proof in which we standardize $X$ regardless, leaving the standardization of $\phi$ to the reader.

\begin{proof}
By unsuturing and then suturing, we may assume that $U(\Gamma) \cong [-1,1] \times (-2\ln(1.1),0] \times \Gamma$ is as constructed in Lemma \ref{lem:suturing}. A computation shows that the convex structure we produce has
$$X_0 = \frac{4e^{\tau/2}}{\pi\epsilon}\sin\left(\frac{\pi}{4}s\right) \partial_s + \frac{2e^{\tau/2}}{\epsilon}\cos\left(\frac{\pi}{4}s\right)\partial_{\tau} + \frac{\epsilon}{2e^{\tau/2}}\sin\left(\frac{\pi}{4}s\right)R_{\beta}$$
Let $$X_1 = s\partial_s + \partial_{\tau}$$
be the desired vector field. We wish to interpolate between $X_0$ and $X_1$. Let $\rho \colon [-2\ln(1.1),0] \rightarrow [0,1]$ be a monotonically nondecreasing function which is identically $0$ near the left endpoint and identically $1$ near the right endpoint, and let $H_{s} = (1-s\rho(\tau))H_0 + s\rho(\tau)H_1$ be a $1$-parameter family of contact Hamiltonians, where $H_0$ generates $X_0$ and $H_1$ generates $X_1$. Explicitly, since $H = \alpha(X)$ and $\alpha = Cds+e^{\tau}\beta$, where $C = \frac{\pi \epsilon^2}{8}$,
$$H_0 = \epsilon \cosh\left(\frac{\tau}{2}\right)\sin\left(\frac{\pi}{4}s\right)$$
$$H_1 = \frac{\pi \epsilon^2}{8}s.$$
Now,
$$dH_t = (1-t\rho)dH_0 +t\rho dH_1 + t\rho'(H_1-H_0)d\tau.$$
Since both $dH_0(\partial_s) > 0$ and $dH_1(\partial_s) > 0$, we have $\frac{1}{C}dH_t(R_{\alpha}) = dH_t(\partial_s) > 0$. Also, if we consider $R_{\beta}$, then we see $dH_t(R_{\beta}) = 0$. Therefore,
$$d\alpha(X_t,R_{\beta}) = dH_t(R_{\alpha})\alpha(R_{\beta}) - dH_t(R_{\beta}) = e^{\tau}dH_t(R_{\alpha}) > 0.$$
In other words, $X_t$ is nowhere zero, so we have an isotopy of convex vector fields, supported in $U(\Gamma)$, with no new critical points. Furthermore, we can explicitly write
$$X_t = (1-t\rho(\tau))X_0 + t\rho(\tau)X_1 + t(H_1-H_0)e^{-\tau}\rho'(\tau)R_{\beta},$$
from which it follows that since $X_0$ and $X_1$ are both transverse to the boundary, so is $X_t$.

That $\phi$ can be put into standard form near the boundary can be handled as in Section \ref{sec:attach}, and specifically Lemma \ref{lem:Morse_modify} and Lemma \ref{lem:make_level}, but in the more general setting where the boundary has corners, and where instead of having the boundary be a level set, we specify it to be of the described form $\phi = s^2+\tau$ in $U(\Gamma)$. Details are left to the reader.
\end{proof}

\subsection{Subcritical handles as thickened Weinstein handles}
By Lemma \ref{lem:standard_suture}, in a sutured contact manifold, we can always find a strict homotopy so that $U(\Gamma)$ splits as a Weinstein collar $((-\delta,0] \times \Gamma, e^{\tau}\beta,\partial_{\tau},\tau)$ cross a convex contact interval $([-1,1],s\partial_s,s^2)$. This allows us to transfer Weinstein handle attachments over to the contact setting.

\begin{defn}
Consider the attachment of a Weinstein handle $H$ along a standard Weinstein collar $C = ((-\delta,0] \times \Gamma, e^{\tau}\beta,\partial_{\tau},\tau)$, say along some framed isotropic $\Lambda$ in $\Gamma$. Then a \textbf{thickened Weinstein handle attachment along $\Lambda$} is the result of taking the product of this Weinstein handle attachment with the convex contact interval $I = ([-1,1],s,\partial_s,s^2)$ to obtain a sutured convex contact manifold, where $C\times I$ is identified with $U(\Gamma)$ in some sutured contact manifold.
\end{defn}

Since our procedure of unsuturing keeps the diving set fixed, there is a natural association of the framing of an isotropic submanifold $\Lambda \subset 
\Gamma$ between the sutured and unsutured models. The following theorem is now an easy consequence.

\begin{thm} \label{thm:subcrit_is_Weinstein}
Let $(M,\xi,\Gamma,U(\Gamma), X,\phi)$ be a sutured convex contact cobordism, and let $(M',\xi',X',\phi')$ be its unsuturing. Let $\Lambda \subset \Gamma$ be a framed isotropic submanifold. Then the following two processes yield strictly convex homotopic convex contact cobordisms:
\begin{itemize}
	\item Attaching a subcritical contact handle to $M'$ along $\Lambda$.
	\item The unsuturing of a thickened Weinstein handle attachment along $\Lambda$.
\end{itemize}
\end{thm}

\begin{proof}
Consider the second procedure. Unsuturing does not change the set of critical points and descending manifolds of the vector field $X$. Hence, after we attach a thickened Weinstein handle and unsuture, we see that there is one new critical point as compared to $M'$, and the neighborhood of its descending manifold is precisely the same as the one given by attaching the thickened convex Weinstein handle. Therefore, by Theorem \ref{thm:subcrit_attach}, this manifold is defined by a single subcritical contact handle attachment along $\Lambda$.
\end{proof}

\section{Weinstein open book decompositions and split convex contact structures} \label{sec:OBD}

We have just proved that subcritical handles correspond precisely to thickened Weinstein handle attachments. Using this description, it is not difficult to prove that split convex contact structures correspond to Weinstein open book decompositions. One direction of this correspondence, from Weinstein open books to convex contact structures, has already appeared in work of Courte and Massot \cite{CoMa}. Whereas Courte and Massot work with isotropic complexes which consist of the skeleta of two pages of the open book, we work explicitly with handle decompositions.

\subsection{Weinstein open book decompositions}

Recall the following definitions.

\begin{defn}
An \textbf{open book decomposition} $(B,\theta)$ for a manifold $M$ consists of:
\begin{itemize}
	\item a codimension $2$ submanifold $B \subset M$, called the \textbf{binding}, with trivial normal bundle, and
	\item a fibration $\theta \colon M \setminus B \rightarrow S^1$, with fibers called \textbf{pages}, such that on a smooth neighborhood $B \times D^2$ of the binding, $\theta$ is the angular coordinate
\end{itemize}
\end{defn}

\begin{defn}
A contact manifold $(M,\xi)$ is said to be \textbf{supported} by an open book $(B,\theta)$ for $M$ if there is a contact form $\alpha$ for $\xi$ such that
\begin{itemize}
	\item $\alpha|_B$ is a contact form on $B$
	\item $d\alpha$ is symplectic on each page
	\item the orientation of $B$ determined by $\alpha$ matches the boundary orientation for the page, where the page is oriented by $d\alpha$
\end{itemize}
If $P$ is a page of an open book supporting a contact manifold, then $\alpha|_P$ endows $P$ with the structure of a Liouville domain. We call $(B,\theta)$ a \textbf{supporting Weinstein open book} if, furthermore, there is a Morse function on the page endowing it with the structure of a Weinstein domain.
\end{defn}

\begin{rmk}
A choice of contact form $\alpha$ inducing the structure above is called a \textbf{Giroux form} \cite{CoMa}. One could also work with the notion of an \textbf{ideal Giroux form}, which is singular along the binding, but has the upshot that it turns the closure of each page into an ideal Liouville domain (see Remark \ref{rmk:ideal}). The two conventions are essentially equivalent \cite[Lemma 3.9]{CoMa}, and we stick with the more classical choice in this presentation.
\end{rmk}

Given an open book, the data of the fibration $\theta \colon M \setminus B \rightarrow S^1$ is encoded by the page $P$ together with the monodromy map $\phi \colon P \rightarrow P$.

\begin{defn}
An \textbf{abstract open book} consists of a pair $(P,\phi)$, where $P$ is a manifold with boundary, and $\phi$ is a diffeomorphism of $P$ fixing a neighborhood of the boundary.
\end{defn}

Given an abstract open book, one can form the mapping torus
$$MT(P,\phi) = (P\times [0,2\pi])/{\sim}$$
where ${\sim}$ is the relation generated by
$$(x,2\pi) \sim (\phi(x),0),~x \in P.$$
Then $\partial MT(P,\phi) = \partial P \times S^1$, so we can naturally form the manifold
$$M := OB(P,\phi ) = MT(P,\phi) \cup (\partial P \times \mathbb{D}^2),$$
gluing along the boundary. This manifold $M$ naturally has an open book structure, with binding $B := \partial P \times \{0\} \subset \partial P \times \mathbb{D}^2 \subset M$. The fibration $\theta \colon M \setminus B \rightarrow \R/2\pi \Z$ is given by:
\begin{itemize}
	\item $\theta([x,t]) = [t]$ on $MT(P,\phi )$
	\item $\theta$ is the angular coordinate on $\partial P \times (\mathbb{D}^2 \setminus 0)$.
\end{itemize}
Every open book decomposition of a smooth manifold is given up to diffeomorphism by this construction.

The following two theorems indicate to what extent the correspondence between abstract open books and honest open books extends to contact geometry. The first is an extension of a three-dimensional construction of Thurston and Winkelnkemper \cite{TW} to higher dimensions.

\begin{thm}[Giroux-Mohsen, see \cite{Giroux2}] \label{thm:WOBD_from_abstract} Suppose $(P,\phi)$ is an abstract open book with $P$ a Liouville domain and $\phi$ a symplectomorphism of $P$ fixed near the boundary. Then there is a contact form, unique up to isotopy, on $OB(P,\phi)$ such that its natural open book structure supports the contact form, and such that the completion of a page is exact symplectomorphic to $\overline{P}$, the completion of $P$, whereby completion means that we attach the positive symplectization of its boundary. Furthermore, if $P$ is a Weinstein domain, then this is a supporting Weinstein open book decomposition. The construction depends only upon $P$ up to Liouville homotopy and the class of $\phi$ in the symplectic mapping class group.
\end{thm}
\begin{proof}[Proof Sketch]
For details of the proof, we refer the reader to Section 7.3 of Geiges' book \cite{Geiges}. The idea is first to homotope the monodromy through symplectomorphisms fixed near the boundary to an exact symplectomorphism, which can be performed parametrically. Using this exact symplectomorphism, one explicitly writes a contact form on $MT(P,\phi)$. Finally, one extends this through $\partial P \times \mathbb{D}^2$ smoothly by choosing an appropriate interpolation between the behavior of the contact form near $\partial MT(P,\phi)$ and the standard contact form $\alpha + \frac{1}{2}r^2d\theta$ near the binding, where $\alpha$ is the contact form on $\partial P$ determined since $P$ is a Liouville domain.
\end{proof}

\begin{thm}[Giroux-Mohsen, see \cite{Giroux2}] \label{thm:WOBD} Every contact manifold has a supporting Weinstein open book decomposition.
\end{thm}

\subsection{Weinstein open books and split convex contact structures}

With this background out of the way, we are ready to state and prove the main theorem of this section. Recall that a split convex contact structure is one in which there is level set lying above all of the subcritical points and below all of the supercritical points.

\begin{thm}[= Theorem \ref{thm:proto_split_is_OBD}] \label{thm:split_is_OBD}
To every split convex contact manifold $(M,\xi,X,\phi)$ with splitting level set $\phi^{-1}(c) = \Sigma = R_+ \cup_{\Gamma} R_-$, there is a supporting Weinstein open book decomposition for the underlying contact structure with binding $\Gamma$ and such that $R_{\pm}$ are pages (and therefore automatically have exact symplectomorphic completions). Conversely, for every abstract Weinstein open book, the corresponding contact manifold can be endowed with a convex structure split by the closure of the union of two opposite pages $\overline{\theta^{-1}(\{0,\pi\})}$.

Furthermore, the composition of going from a supporting Weinstein open book to a convex structure and then back to a Weinstein open book will produce the same Weinstein open book up to exact symplectomorphism of the page and homotopy of monodromy within a fixed symplectic mapping class.
\end{thm}

\begin{proof}
We begin by associating to any split convex contact structure a Weinstein open book supporting the underlying contact structure. Notice that in the argument that follows, we allow for homotopies of the convex contact structure fixed near $\Sigma$, since a Moser argument implies that the contact structure itself may be assumed to remain fixed. The splitting level set $\Sigma$ separates $M$ into two pieces $M_{\pm}$, where $M_-$ contains only subcritical points, and $M_+$ contains only supercritical points. Then up to strict homotopy, we have sutured models for $M_{\pm}$ given as thickened Weinstein domains $W_{\pm} \times I$ by Theorem \ref{thm:subcrit_is_Weinstein}. By Lemma \ref{lem:fix_top}, we therefore have that, up to strict homotopy fixing the convex contact structure near $\Sigma$, we can view $M$ as $M = u(W_+\times I) \cup C \cup u(W_- \times I)$, where $C$ is a cylindrical cobordism. Notice that the unsuturing of a thickened Weinsten domain $u(P \times I)$ has boundary given by $R_+ \cup_{\Gamma} R_- = \overline{P^+} \cup_{\partial P} \overline{P^-}$, where $\overline{P}$ denotes the completion, and $P^{\pm}$ correspond to the positive and negative choice of Liouville form. So since cylindrical cobordisms induce exact symplectomorphisms via their holonomy maps, we must have that the completions of $W_+$ and $W_-$ are exact symplectomorphic. We may therefore take $W_+ = W_- =: W$.

On each $W \times I$, we have that the convex structure has $(\alpha,X) = (ds+\lambda,s\partial_s + V_{\lambda})$, where $\lambda$ is a Liouville form for $W$ and $V_{\lambda}$ is its Liouville vector field. We see that $X$ corresponds to the Hamiltonian given by $H =s$. Meanwhile, the Reeb vector field $R_{\alpha} = \partial_s$ corresponds to a Hamiltonian $H = 1$. We wish to interpolate between these two so that we may assume that we have a model in which $X = \pm R_{\alpha}$ along $R_{\pm}$ but away from $\Open(\Gamma \times I)$. We shall consider $U(\Gamma) = [-1,1] \times (-K,0] \times \Gamma$, where $K > 0$ is some sufficiently large constant, and on which $\alpha = ds + e^{\tau}\beta$. We simply take an interpolating family $H_{\tau}$ in $U(\Gamma)$ such that
\begin{itemize}	
	\item $H_0 = s$
	\item $H_{\tau}$ remains fixed on $\Open(\Gamma) \cup \Open(W \times \{0\})$
	\item $H_{\tau}^{-1}(0) = W \times \{0\}$
	\item $H_1 = 1$ on $\Open(R_+) \setminus \Open(\Gamma \times \{1\})$
	\item $H_1 = -1$ on $\Open(R_-) \setminus \Open(\Gamma \times \{-1\})$
	\item $\pm\left(H - \frac{\partial H}{\partial \tau}\right) > 0$ along $\Open(R_{\pm})$
\end{itemize}
The only difficult condition is the last one, but is one which can be realized by an interpolation so long as $K >> 0$. These conditions imply that we may homotope $X$ so that it remains fixed on $\Open(\Gamma \times I) \cup \Open(W \times 0)$ and does not introduce any new critical points. The last condition implies that the vector field remains outwardly transverse to the boundary at all times.

So now, $M = u(W \times I) \cup C \cup u(W \times I)$, as before. When we modify $X$ by the homotopy described along each copy of $W \times I$, this can be extended to all of $M$ by a strict homotopy supported near $\Open_C(\partial_{\pm} C)$. Set $B := \{X \in \xi\} \cap C$. Then away from $\Open_M(B)$, we obtain a nowhere zero contact vector field $Y$ given as follows:
\begin{itemize}
	\item On each $(W \times I) \setminus \Open(B)$, we simply take $Y$ to be given by the Reeb vector field
	\item On $C \setminus \Open(B)$, we take $Y$ to be the vector field given by $X$ on the region with $\alpha(X) > 0$, and $-X$ on the region with $\alpha(X) < 0$.
\end{itemize}
This contact vector field $Y$, being everywhere nonzero, is hence the Reeb field for some choice of contact form $\alpha^Y$. We have a fibration from $M \setminus \Open(B) \rightarrow S^1$ given by taking pages $W \times \{s\}$ from each $W \times I$, together with the pages given by the positive and negative sides of each level set of $\phi$. But since $Y$ is transverse to these pages, we have that $d\alpha^Y$ is symplectic on each of them. Notice that the flow of $Y$ induces a monodromy map (where it is defined) which is an exact symplectomorphism of pages. It suffices now to extend the open book decomposition and choice of $\alpha^Y$ through $\Open(B)$. But $\Open(B)$ is just contactomorphic a standard neighborhood of $\Gamma$, and so the contact form filling it is precisely the same one described as in the proof of Theorem \ref{thm:WOBD_from_abstract}. Hence, we have constructed our open book as desired.

Conversely, suppose we have a contact manifold supported by a Weinstein open book, with page $P$ and monodromy $\phi$. Consider the convex contact structure given by $u(P \times I) \cup C \cup u(P \times I)$, where $C$ is a cylindrical convex contact cobordism with holonomy map given by the identity on $R_+$ and $\phi^{-1}$ on $R_-$. (This does indeed form the holonomy of a cylindrical cobordism by Lemma \ref{lem:holonomy} since this is a convexomorphism.) Then the above forward direction implies that this convex contact structure is supported by precisely the same Weinstein open book. The result now follows.
\end{proof}

\begin{cor}[= Corollary \ref{cor:existence}] \label{cor:CCS_exist}
Every contact manifold has a (split) convex contact structure, and therefore is a convex contact handlebody.
\end{cor}
\begin{proof} By Theorem \ref{thm:WOBD}, every contact manifold has a supporting Weinstein open book decomposition. Hence Theorem \ref{thm:split_is_OBD} implies there is a naturally associated convex structure.
\end{proof}

\section{Birth-death type singularities and convex contact homotopy} \label{sec:homotopy}

\subsection{Subcritical cancellation}

So far, we have been discussing convex contact manifolds up to strict homotopy. This is too restrictive, as it does not allow for any critical point creation or cancellation as occurs in Morse theory or Weinstein surgery theory. Certainly we should allow for this possibility. Recall that we defined birth-death type singularities in Definition \ref{def:birth-death}.

\begin{defn} A \textbf{convex contact homotopy} is a homotopy of convex contact structures $(\xi,X,\phi)$ on a cobordism $M$ which is allowed to pass through birth-death type singularities.
\end{defn}

Any birth-type singularity will produce two new handles of neighboring indices, and any death-type singularity will cancel two handles of neighboring indices. However, the underlying contact geometry places restrictions on how these neighboring index handles interact with each other. Luckily, the answer has already been worked out for Weinstein handles, and so the subcritical case is immediate.

\begin{prop}[= Corollary \ref{cor:subcrit_cancel_intro}] \label{prop:subcrit_cancel} Suppose $(M^{2n+1},\xi,X,\phi)$ is a convex contact cobordism with exactly two critical points $p,q$ of index $k$ and $k+1$, where $0 \leq k \leq 2n$ and $k \neq n$. Suppose there is precisely one $X$-trajectory from the index $k$ critical point to the index $k+1$ critical point, along which the ascending disk of $p$ and the descending disk of $q$ intersect transversely. Then the critical points can be cancelled; that is, $(M,\xi,X,\phi)$ is homotopic to a cylindrical convex contact cobordism.
\end{prop}
\begin{proof}
We begin with the case $0 \leq k \leq n-1$. Since these are subcritical handle attachments, by Theorem \ref{thm:subcrit_is_Weinstein}, we can view them instead as thickened Weinstein handle attachments, where the corresponding Weinstein handles also have a single trajectory between critical points. But these handles can be cancelled up to Weinstein homotopy (see \cite[Proposition 12.22]{CE} for details), which yields the result.

For $n+1 \leq k \leq 2n$, we simply use $(-X,-\phi)$ instead.
\end{proof}

\begin{rmk}
One need not have that there is only one trajectory between critical points to cancel handles, since the attaching and coattaching data for the neighboring critical points are only determined up to isotopy of attaching and coattaching spheres. It is then a more subtle question to determine when an isotopy which does produce only one trajectory exists. However, in many cases, this question reduces to the smooth one.

Let us suppose that $n \geq 3$, so that we are dealing with contact manifolds of dimension at least $7$. If $k \leq n-2$, or if $k = n-1$ but the attaching sphere of the corresponding $(k+1)$-handle is a loose Legendrian in the dividing set between the critical points, then there is flexibility, in the same sense as discussed in the Weinstein setting \cite[Chapter 14]{CE}. That is, smooth isotopies of the attaching sphere of the $(k+1)$-handle can be approximated by an isotropic or Legendrian isotopy. So it is actually sufficient for the algebraic intersection of the equator of the cosphere of the index $k$ handle and the attaching sphere of the index $k+1$ handle in the dividing set to be equal to $\pm 1$ in order for them to cancel, since then this h-principle allows one to isotope the attaching sphere of the index $(k+1)$-handle so that it intersects the coattaching sphere of the $k$-handle precisely once, transversely. The $k \geq n+2$ case, and also the $k=n+1$ with loose coattaching sphere, are handled identically.
\end{rmk}

\subsection{Smoothly cancelling pairs}

The remaining case is the cancellation of an $n$- and $(n+1)$-handle. In particular, we wish to understand the following question.

\begin{quest} How can we tell when we have a cancelling $n$- and $(n+1)$-handle pair?
\end{quest}

In order to cancel, we need that they at least smoothly cancel, and so up to isotopy of the attaching data of the $(n+1)$-handle, we need there to be one transverse trajectory between the critical points. We describe a particular setting in which this occurs.

The following definition matches with terminology introduced in Honda and Huang's paper defining bypasses in higher dimensions \cite{HH}.

\begin{defn}
A \textbf{smoothly cancelling pair} is a convex contact cobordism consisting of a single $n$- and $(n+1)$-handle with a single transverse gradient trajectory between them.
\end{defn}

We describe what data specifies a smoothly cancelling pair. The reader is referred to Figure \ref{fig:SCP} for an illustration of the following discussion in the case of contact $3$-manifolds, i.e. $2n+1=3$. Suppose, starting from a convex hypersurface $\Sigma$, we attach a smoothly cancelling pair. The $n$-handle is attached along a Legendrian $\Lambda_n$ in the dividing set $\Gamma = \partial R_+ = \partial R_-$. The level set $\Sigma'$ in between the two critical points is given by $\Sigma' = R'_+ \cup R'_-$, where $R'_{\pm} = R_{\pm} \cup H_{\pm}$ are given by Weinstein handle attachment along $\Lambda_n$. Meanwhile, the $(n+1)$-handle is attached along a balanced coisotropic sphere $S$ in $\Sigma'$. In this dimension, a balanced coisotropic sphere is simply a pair of Lagrangian disk fillings $D_{\pm}$ of a Legendrian sphere $\Lambda_{n+1}$ in the dividing set. Since isotoping the attaching sphere of the $(n+1)$-handle does not affect the resulting cobordism up to strict homotopy, we may assume that $\Lambda_{n+1}$ is not in the boundary of the Weinstein handle $H_{\pm}$, and so is a Legendrian in $\partial R_+$ not intersecting $\Lambda_n$. Furthermore, since there is a single trajectory between the critical points, we may assume that the attaching sphere of the $(n+1)$-handle passes through either $H_+$ or $H_-$ (but not both) precisely once, along the core of the handle. Let us assume without loss of generality that it passes through the handle on the positive side, since we may equally well consider the negative case by simply swapping the coorientation of the underlying contact structure. Then $R_+ \cap S$ is an exact Lagrangian annulus, filling the union of $\Lambda_n$ and the $\Lambda_{n+1}$. Meanwhile, $S \cap R'_-$ never enters the handle, and is just a Lagrangian disk filling in $R_-$ of the equator of $\Lambda_{n+1}$. This discussion shows that we may represent a smooth cancelling pair by the following data along $\Sigma$:
\begin{itemize}
	\item a Legendrian sphere $\Lambda_n$ in the dividing set $\Gamma$, giving the attaching data of the $n$-handle
	\item a Legendrian sphere $\Lambda_{n+1}$ in $\Gamma$, giving the equator of the attaching data of the $(n+1)$-handle
	\item a Lagrangian disk $D_- \subset R_-$ filling $\Lambda_{n+1}$, giving one half of the attaching sphere of the $(n+1)$-handle
	\item a Lagrangian annulus $A_+ \subset R_+$ filling $\Lambda_n \cup \Lambda_{n+1}$, which when extended by the core of a Weinstein handle attached along $\Lambda_n$ yields a Lagrangian disk giving the other half of the attaching sphere of the $(n+1)$-handle
\end{itemize}
We may of course swap the roles of $R_+$ and $R_-$ to provide another model for a smoothly cancelling pair.

\begin{figure}[h!]
\centering
  \includegraphics[width=\textwidth]{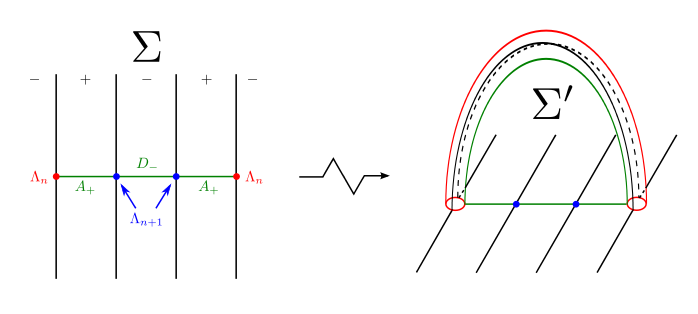}
  \caption{This figure depicts a smoothly cancelling pair in the case of $2n+1 = 3$. On the left is the smoothly cancelling attaching data on a convex surface $\Sigma$, representing an $n$ and $n+1$ handle attachment which smoothly cancel. On the right, we show how $D_- \cup A_+$ extends in the convex surface $\Sigma'$ obtained as the level set after attaching an $n$-handle to $\Sigma$ along $\Lambda_n$ to yield the attaching sphere of the $(n+1)$-handle (in green).}\label{fig:SCP}
\end{figure}

\begin{defn}
The quadruple $(\Lambda_n,\Lambda_{n+1},A_+,D_-)$ will be referred to as \textbf{smoothly cancelling attaching data}.
\end{defn}

\begin{rmk}
In our definition of a smoothly cancelling pair, we artificially include the hypothesis that there is a single trajectory between the critical points, so that we explicitly see that they smoothly cancel. Notice that it is not necessarily the case that a pair of $n$ and $(n+1)$-handles will smoothly cancel if and only if they can be convex contact homotoped to have a single gradient trajectory, since convex contact homotopy is not the same as smooth homotopy of the pair $(X,\phi)$. We suspect there are examples of a pair of handles which smoothly cancel but cannot be strictly convex contact homotoped so that there is a single gradient trajectory between them, and that they can be detected by some version of wrapped Floer cohomology.
\end{rmk}

\begin{rmk}
That $\Lambda_{n+1}$ has a disk filling in $R_-$ is to say that it is \textbf{Lagrangian slice} with respect to $R_-$. This condition is already quite restrictive, even for Legendrian knots in standard $S^3$ filled by the standard $4$-ball. A complete list of Lagrangian slice knots up to 14 crossings can be found in \cite{CNS}.
\end{rmk}

\subsection{Examples of critical cancellation}

The goal in this subsection is to provide a specific model for which a smoothly cancelling pair does cancel in the convex contact setting. Consider the following family of contact vector fields $X_{\epsilon}$ for $\alpha =  -dz-\frac{1}{2}\sum(x_idy_i-y_idx_i)$  on $\R^{2n+1}$, where $-1 < \epsilon < 1$:
$$X_{\epsilon} = (z^2+\epsilon)\partial_z + \sum \left(x_i(z-\frac{3}{2})\partial_{x_i}+y_i(z+\frac{3}{2})\partial_{y_i}\right)$$
(The weird choice of sign for $\alpha$ is motivated by the fact that we want the trajectory connecting the critical points to lie in $R_+$.) This is just the vector field corresponding to the Hamiltonian $H_{\epsilon} = -z^2 - \frac{3}{2}\sum x_iy_i - \epsilon$. Note that $X_{\epsilon}$ passes through a death-type singularity as $\epsilon$ increases through $\epsilon = 0$. When $\epsilon < 0$, there are two critical points $p_{\pm}$ with coordinates $x_i = y_i = 0$ and $z = \pm \sqrt{|\epsilon|}$. The point $p_+$ has index $n$ with ascending manifold along the plane given by $x_1 = \cdots = x_n = 0$, whereas $p_-$ has index $n+1$ and descending manifold along the plane given by $y_1 = \cdots = y_n = 0$. We see from this description that there is a unique gradient trajectory from $p_{+}$ to $p_{-}$ along the $z$-axis.

We can turn this family of contact vector fields into a convex structure with the addition of a family of Morse functions $\phi_{\epsilon}$ so that the pair $(X_{\epsilon},\phi_{\epsilon})$ is of birth-death type. With these choices, $p_{+}$ is a critical point of index $n$, with descending disk along the plane parallel to the $x_i$-directions and ascending disk along the other coordinate directions, while $p_{-}$ is a critical point of index $n+1$, with descending disk along the $x_i$ and $z$-directions and ascending disk along the $y_i$-directions.

The attaching disks for both the $n$- and $n+1$-handle are along the plane spanned by the $x_i$ and $z$-directions. Consider the following pieces of convex hypersurfaces, transverse to $X$, and where $\delta > 0$ is a small constant:
$$\Sigma_1 = \{|\vec{y}| \leq 1/2, |\vec{x}| = \delta, -1-\delta \leq z < 1+\delta\}$$
$$\Sigma_2 = \{1/2 \leq |\vec{y}| < 1, |\vec{x}| = \delta, -1-\delta \leq z < 1+\delta\}$$
$$\Sigma_3 = \{|\vec{y}| \leq 1, |\vec{x}| = \delta, 1+\delta \leq z \leq 1+2\delta\}$$
$$\Sigma_4 = \{|\vec{y}| \leq 1/2, z = -1-\delta\}$$
$$\Sigma_5 = \{1/2 \leq |\vec{y}| < 1, z = -1-\delta\}$$
$$\Sigma_6 = \{|\vec{y}|=1/2, |\vec{x}| \leq \delta, |z|\leq 1+\delta\}$$
$$\Sigma_7 = \{|\vec{y}| \leq 1/2, |\vec{x}| \leq \delta, z=1+\delta\}$$
This set-up is drawn in Figure \ref{fig:Cancel_Region}. Take
$$\Sigma_- = \Sigma_1 \cup \Sigma_2 \cup \Sigma_3 \cup \Sigma_4 \cup \Sigma_5$$
$$\Sigma_+ = \Sigma_2 \cup \Sigma_3 \cup \Sigma_5 \cup \Sigma_6 \cup \Sigma_7.$$

\begin{figure}[h!]
\centering
  \includegraphics[width=\textwidth]{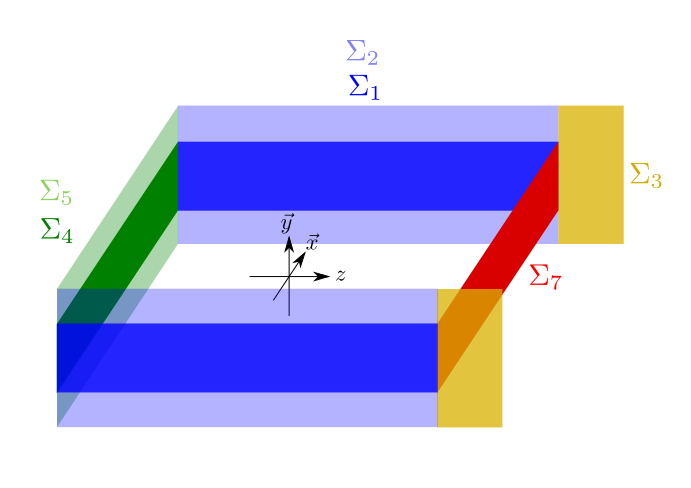}
  \caption{The set-up for the model of cancellation. The origin is at the center of the diagram. Missing from the picture is $\Sigma_6$, which consists of the horizontal planes which, along with the darker vertical regions $\Sigma_1$, $\Sigma_4$, and $\Sigma_7$, enclose a box giving a cobordism from $\Sigma_-$ to $\Sigma_+$.}\label{fig:Cancel_Region}
\end{figure}

We see that $\Sigma_{\pm}$ are piecewise smooth surfaces, and they can be smoothed out so that $X$ remains transverse to them. Regardless of $\epsilon$, we have a convex contact cobordism between $\Sigma_-$ and $\Sigma_+$ given by appending the region contained in the `box' given by $\{|\vec{x}| \leq \delta, |\vec{y}| \leq 1/2, |z| \leq 1+\delta\}$. Hence, we have a family of convex contact structures on the cobordism between $\Sigma_-$ and $\Sigma_+$ which passes through a death-type singularity. For $\epsilon > 0$ there are no critical points, and so this is a trivial cylinder. When $\epsilon < 0$, we have that the two critical points are contained in the box, and their attaching data, thought of as a smoothly cancelling pair, is completely contained in $\Sigma_-$. Hence, we have obtained a model for a cancelling pair of critical points. That is, if we find $\Sigma_-$ lying in a convex hypersurface, and we have an $n$- and $(n+1)$-handle attached which match with the model provided by $\Sigma_-$, then we can cancel the handles.

We note that for this model, when $\epsilon < 0$, the trajectory connecting the critical points lies in the region where $\alpha(X) > 0$. (If we wanted a model where the trajectory had $\alpha(X) < 0$, we would have used $-\alpha$ instead.)

As a smoothly cancelling pair, the data along $\Sigma_-$ must consist of the data of two Legendrians together with an annulur filling of their unions and a disk filling of one of them. The following proposition summarizes our model.

\begin{prop}
The following smoothly cancelling attaching data yields a pair of contact handles which cancel up to convex contact homotopy.
\begin{itemize}
	\item $\Lambda_n$ is any Legendrian in the dividing set
	\item $\Lambda_{n+1}$ is a small positive Reeb push-off of $\Lambda_n$
	\item In $R_-$, we choose any disk filling $D_-$ of $\Lambda_{n+1}$
	\item In $R_+$, we choose the standard annulur filling $A_+$ of $\Lambda_n \cup \Lambda_{n+1}$. (There is a standard choice filling a Legendrian with its Reeb push-off, such that the filling is contained in a neighborhood of $\Gamma$. This cobordism is decomposable, in that it is realized by the trace of a Legendrian isotopy together with a single surgery `pinch' move.)
\end{itemize}
\end{prop}

\begin{proof}
We begin by proving that the attaching data along $\Sigma_-$ is of this type. Notice that $D_-$ does not stay near the dividing set along $\Sigma_3 \subset \Sigma_-$, and any two Lagrangian disks have isomorphic neighborhoods, so $D_-$ may indeed be arbitrarily chosen. As for $D_+$, we see that for $\epsilon < 0$ of small enough magnitude, $D_+$ remains in a collar neighborhood of $\Gamma$ along $\Sigma_1$.

We explicitly have that $\Gamma$ along $\Sigma_1$ is given by the subset in which the Hamiltonian $H = z^2+\frac{3}{2}\sum_{i=1}^{n}x_iy_i - |\epsilon|$ is equal to zero. Recall also that $\Sigma_1$ lies along the surface with $\sum x_i^2 = \delta^2$. One computes from this data that the kernel of $d\alpha|_{\Gamma}$ is given by the span of $-\sum x_iz\partial_{y_i} + \frac{3}{4}\delta^2 \partial_z$. Therefore, one checks that the Reeb vector field is
$$R_{\alpha|_{\Gamma}} = \frac{4}{(3-2z)\delta^2}\left(\sum x_iz\partial_{y_i} - \frac{3}{4}\delta^2 \partial_z\right).$$
We notice that the Reeb vector field along $\Gamma \subset \Sigma_1$ always flows in the negative $z$ direction ($z < 3/2$ on $\Sigma_1$), and we see that the flowlines are symmetric about reflection through the plane $z=0$. Hence, since $\Lambda_{n+1}$ at $z=-\sqrt{|\epsilon|}$ and $\Lambda_n$ at $z = \sqrt{|\epsilon|}$ are reflected copies of each other through this plane, we see that the Reeb vector field flows from $\Lambda_{n}$ to $\Lambda_{n+1}$.

We leave the proof of the description of $A_+$ to the reader. It can be computed using the explicit model on $\Sigma_1$.

To conclude the general result, we need furthermore that if we have attaching data of the form described in the statement of the proposition, then it has a neighborhood which is equivalent to a neighborhood of the attaching data along $\Sigma_-$, since the cancellation of these critical points is local to this neighborhood in $\Sigma_-$. The standard annular filling of a Legendrian knot and its Reeb push-off automatically has a standard neighborhood, as does any Lagrangian disk filling a Legendrian; the result is essentially an exercise.
\end{proof}

We note that this matches one of the two models of \textbf{trivial bypass attachment} given by Honda and Huang \cite{HH}. There is another choice, in which the smoothly cancelling data is as follows:
\begin{itemize}
	\item $\Lambda_n$ is any Legendrian in the dividing set which is Lagrangian slice
	\item $\Lambda_{n+1}$ is a meridional standard Legendrian unknot linking $\Lambda_n$
	\item For $D_-$, we choose the standard disk filling of $\Lambda_{n+1}$ (in a neighborhood of $\Gamma$)
	\item For $A_+$, there is a standard choice of annulus filling $\Lambda_n \cup \Lambda_{n+1}$ given by any choice of slice disk for $\Lambda_n$
\end{itemize}
This other model is essentially the same as the previous model up to homotopy of the attaching data of the $(n+1)$-handle. The key technical detail can be found in \cite[Lemma 5.2.3]{HH}, which details how the attaching changes as $\Lambda_{n+1}$ passes through $\Lambda_n$ by a handle slide.

\begin{rmk}
Honda and Huang \cite{HH} prove that attachment of a trivial bypass yields a positive stabilization of a supporting partial open book. They also define anti-bypass attachments, which yield negative stabilizations of a supporting partial open book. Since negatively stabilized open books are automatically overtwisted \cite{CMP}, we have that this yields examples of smoothly cancelling attaching data which cannot possibly cancel in the convex contact setting, since introduction of a negative stabilization of a tight contact structure yields an overtwisted, hence non-contactomorphic, contact structure.
\end{rmk}

\begin{rmk}
We leave the question of when smoothly cancelling attaching data yields a trivial convex contact cylinder to future work. It is not clear, for instance, whether the trivial bypasses are the only possibility. A description would require further analysis of embryonic critical points in the convex contact setting.
\end{rmk}

\bibliography{Contact_Handles}{}
\bibliographystyle{plain}

%--------------------END-HERE--------------------------------------------
%------------------------------------------------------------------------

\end{document}